\documentclass[11pt]{ip-journal}
\usepackage{amsmath}
\usepackage{amssymb}
\usepackage{amscd}
\usepackage{amsthm}
\usepackage{graphicx}
\usepackage[all,cmtip]{xy}

\newtheorem{theorem}{Theorem}[section]
\newtheorem{proposition}[theorem]{Proposition}
\newtheorem{lemma}[theorem]{Lemma}

\newtheorem{corollary}[theorem]{Corollary}
\newtheorem{conjecture}[theorem]{Conjecture}

\theoremstyle{definition}
\newtheorem{definition}[theorem]{Definition}
\newtheorem{example}[theorem]{Example}
\newtheorem{remark}[theorem]{Remark}
\newtheorem{notation}[theorem]{Notation}

\def\Id{\operatorname{Id}}
\def\Ker{\operatorname{Ker}}
\def\Coker{\operatorname{Coker}}
\def\Hom{\operatorname{Hom}}
\def\Ext{\operatorname{Ext}}
\def\Aut{\operatorname{Aut}}
\def\Pic{\operatorname{Pic}}
\def\im{\operatorname{Im}}
\def\Spec{\operatorname{Spec}}
\def\scrO{\mathcal{O}}
\def\spcheck{^{\vee}}

\newcommand{\R}{\mathbb{R}}
\newcommand{\C}{\mathbb{C}}
\newcommand{\Z}{\mathbb{Z}}

\newcommand{\Q}{\mathbb{Q}}

\newcommand{\F}{\mathbb{F} }
\newcommand{\Pee}{\mathbb{P}}

\newcommand{\twopartdef}[4]
{
	\left\{
		\begin{array}{ll}
			#1 & \mbox{if } #2 \\
			#3 & \mbox{if } #4
		\end{array}
	\right.
}

\newcommand{\twobytwo}[4]
{
	\begin{pmatrix}
		#1 & #2 \\
		#3 & #4 \\
	\end{pmatrix}
}

\title[Smoothings and RDP Adjacencies for Cusp Singularities]{Smoothings and Rational Double Point Adjacencies for Cusp Singularities}
\author{Philip Engel}
\thanks{Research partially
    supported by NSF grant DMS-1502585.} 
\author{Robert Friedman}
\begin{document}

\maketitle

\begin{abstract}A cusp singularity is a surface singularity whose minimal resolution is a cycle of smooth rational curves meeting transversely. Cusp singularities come in dual pairs. Looijenga proved in 1981 that if a cusp singularity is smoothable, the minimal resolution of the dual cusp is the anticanonical divisor of some smooth rational surface. In 1983, the second author and Miranda gave a criterion for smoothability of a cusp singularity, in terms of the existence of a K-trivial semistable model for the central fiber of such a smoothing. We study these ``Type III degenerations" of rational surfaces with an anticanonical divisor---their deformations, birational geometry, and monodromy. Looijenga's original paper also gave a description of the rational double point configurations to which a cusp singularity deforms, but only in the case where the resolution of the dual cusp has cycle length 5 or less. We generalize this classification to an arbitrary cusp singularity, giving an explicit construction of a semistable simultaneous resolution of such an adjacency. The main tools of the proof are (1) formulas for the monodromy of a Type III degeneration, (2) a construction via surgeries on integral-affine surfaces of a degeneration with prescribed monodromy, (3) surjectivity of the period map for Type III central fibers, and (4) a theorem of Shepherd-Barron producing a simultaneous contraction to the adjacency.\end{abstract}

\section*{Introduction}

Let $(X,x)$ be the germ of an isolated singularity. Some basic questions that one can ask about the deformation theory of $(X,x)$ are: Is $(X,x)$ smoothable? If so, how can we describe the smoothing components of $(X,x)$? What are the singularities adjacent to $(X,x)$? In particular, for a surface singularity, what are the rational double points adjacent to $(X,x)$? In case $(X,x)$ is itself a rational double point, the answers to these questions are well-known. In this case, $(X,x)$ is smoothable and there is a unique smoothing component, as it is a hypersurface singularity, and the adjacent rational double points correspond to subgraphs of the Dynkin diagram corresponding to $(X,x)$.

The next level of complexity for a normal surface singularity is the case where $(X,x)$ is a minimally elliptic singularity in Laufer's terminology, i.e.\  a Gorenstein elliptic singularity, and here the simplest case is when the fundamental cycle
of $(X,x)$ is reduced. In this case, $(X,x)$ is either a simple elliptic, cusp, or triangle singularity.
These three types admit a resolution whose exceptional fiber is a smooth elliptic curve, a cycle of smooth rational curves, or either cuspidal, two smooth rational curves meeting at a tacnode, or three smooth rational curves meeting at a point, respectively.

The deformation theory of simple elliptic and triangle singularities is by now well understood, by work of Pinkham and Looijenga \cite{Pink1}, \cite{Pink2}, \cite{Looij1},   \cite{Looij2}. 
The crucial point here is that the singularities in question have a $\C^*$-action. Thus questions about smoothings and adjacent singularities can be related to global questions about the existence of certain compact algebraic surfaces with configurations of curves on them: (generalized) del Pezzo surfaces with a smooth  section of the anticanonical divisor in the case of simple elliptic singularities, $K3$ surfaces with so-called $T_{p,q,r}$ configurations in the case of triangle singularities. For example, in case $(X,x)$ is simple elliptic, let $(\widetilde{X}, E)$ be the minimal resolution of $X$ and set $E^2 = -d$. Then we may realize the germ $(\widetilde{X}, E)$ on an elliptically ruled surface with zero section $E$ and infinity section $E'$ satisfying $(E')^2=d$. We preserve the divisor $E'$ on the nearby fibers of a negative weight deformation of $(X,x)$ to get pairs $(X_t, E')$, where $X_t$ is a normal projective surface such that $E'\in |-K_{X_t}|$.

It follows that nontrivial negative weight deformations of $(X,x)$ correspond to generalized del Pezzo surfaces $X_t$, hence $d\leq 9$ whenever $(X,x)$ is smoothable (the converse also holds). Furthermore, the adjacent configurations of rational double points to $(X,x)$ are those that appear on such an $X_t$. These in turn can be described by the configurations of curves, all of whose components are isomorphic to $\Pee^1$ and of self-intersection $-2$, which appear as the exceptional curves on a minimal resolution $Y_t \rightarrow X_t$. For $d$ small, $(X,x)$ is a  complete intersection singularity ($d\leq 4$) or a Pfaffian singularity ($d=5$). In this case, there is an associated root system $R$ given by taking the vectors of square $-2$ in $(K_{Y_t})^\perp$. (By the usual algebraic geometry conventions, all root systems in this paper will be \textbf{negative definite}.) Then the possible Dynkin diagrams of configurations of rational double points on some $X_t$ are given by root systems contained in $R$ satisfying a mild extra condition coming from the period map, which can in turn be described via embeddings into the extended Dynkin diagram of $R$.
For $d$ large, however, there can be several smoothing components of $(X,x)$ and $(K_{Y_t})^\perp$ need not be a root lattice.
Nonetheless, the possible adjacent rational double point configurations can still be described lattice-theoretically.

For the case of cusp singularities, Looijenga gave a beautiful description of the possible adjacencies in his breakthrough paper \cite{Looij}, but only in the case of multiplicity at most $5$. In this case, the cusp singularity is either a complete intersection or Pfaffian, hence automatically smoothable and   there is a unique smoothing component. As in the case of simple elliptic singularities of multiplicity at most $5$, there is again a diagram which determines the possible adjacent singularities, and it is the intersection graph for the root basis of a generalized root system of hyperbolic type.

As above, a key ingredient in Looijenga's analysis is the existence of a complete surface containing the cusp $(X,x)$ on which one globalizes deformations. But one must leave the world of algebraic geometry to produce such a surface. The Inoue surface $V_0$ associated to a cusp $(X,x)$ is a compact complex surface of type VII${}_0$ such that $D+D'\in |-K_{V_0}|$, where an analytic neighborhood of $D'=\sum_{i=1}^{r'}D_i'$ is a minimal resolution of the cusp $(X,x)$, and by definition an analytic neighborhood of $D=\sum_{j=1}^rD_j$ is a minimal resolution of the \textsl{dual cusp} $(\widehat{X}, \hat{x})$ to $(X,x)$. (For notational reasons, it is more convenient to denote by $D'$ the minimal resolution of the cusp we are interested in deforming and by $D$ the minimal resolution of its dual.) Here $D$ is a cycle of smooth rational curves, or an irreducible nodal curve if $r=1$, and similarly for $D'$. By convention, $D_i\cdot D_{i\pm 1} = 1$, where the integer $i$ is taken mod $r$, and $D_i\cdot D_j = 0$ otherwise. The \textsl{type} or \textsl{self-intersection sequence} of $D$ is the sequence $(D_1^2, \dots, D_r^2)$ mod cyclic permutations and order-reversing permutations. There is a somewhat complicated  recipe for obtaining $D$, together with its self-intersection sequence, from $D'$.

Looijenga showed that, if $\overline{\overline{V}}_0$ is the compact singular analytic surface obtained by contracting the curves $D'$ and $D$, then the deformation functor $\mathbf{Def}_{\overline{\overline{V}}_0}$ of $\overline{\overline{V}}_0$ is naturally isomorphic to  $\mathbf{Def}_{(\widehat{X}, \hat{x})} \times \mathbf{Def}_{(X,x)}$. Thus,  the deformations of $(X,x)$, viewed as deformations of $(X,x) \amalg  (\widehat{X}, \hat{x})$ for which the deformation of $(\widehat{X}, \hat{x})$ is trivial, can be identified with deformations of the pair $(\overline{V}_0, D)$ keeping $D$ constant, where $\overline{V}_0$ is the compact singular analytic surface obtained by contracting $D'$. A smoothing of $(X,x)$ then yields a pair $(Y_t,D_t)$, where $Y_t$ is a smooth rational surface and  $D_t\in |-K_{Y_t}|$ is a cycle of rational curves of the same type as $D\subseteq V_0$. Briefly, we say that $D$ \textsl{sits on a rational surface} $Y$ (we identify $D_t$ with $D$). Hence a smoothing component of $(X,x)$ determines a deformation type $[(Y,D)]$ of \textsl{anticanonical pairs} $(Y,D)$,  i.e.\ pairs consisting of a smooth rational surface $Y$ and a section $D\in |-K_Y|$ which is a cycle of rational curves of the same type as the minimal resolution of the dual cusp to $(X,x)$. We define a smoothing component of $(X,x)$ to be \textsl{of type $[(Y,D)]$} in this case. In particular, the existence  of such a pair $(Y,D)$ is a \textit{necessary} condition for the cusp $(X,x)$ to be smoothable. In \cite{Looij}, Looijenga conjectured that it is also a \textit{sufficient} condition:

\begin{conjecture}[Looijenga's conjecture] The cusp $(X,x)$ is smoothable if and only if the minimal resolution of the dual cusp sits on a rational surface.
\end{conjecture}

Motivated by the corresponding picture for degenerations of $K3$ surfaces, the second author and R.\ Miranda showed in \cite{FriedmanMiranda} that Looijenga's conjecture was equivalent to the existence of a certain semistable model $Y_0$ for the smoothing, whose description we recall in Section 2. The main idea is to show that  the semistable fiber $Y_0 = \bigcup_{i=0}^nV_i$ deforms in a smooth one parameter family $\mathcal{Y} \to \Delta$. Here $V_0$ is the Inoue surface with the cycles $D'$ and $D$, and the $V_i$, $i\geq 1$, are rational surfaces meeting $V_0$ along the cycle $D'$. A theorem of Shepherd-Barron \cite{Shepherd-Barron} shows that the divisor $\bigcup_{i= 1}^nV_i$ is exceptional in $\mathcal{Y}$, although we might also have to contract some rational curves in the general fibers. (The theorem of Shepherd-Barron is stated in the context of degenerations of $K3$ surfaces, but the proof applies in this setting as well.) The result is a singular complex threefold $\overline{\mathcal{Y}}$ and a flat morphism $\bar{\pi}\colon \overline{\mathcal{Y}}\to \Delta$ such that the fiber over $0$ is the singular Inoue surface $\overline{V}_0$ and the general fiber is a rational surface, possibly with rational double points. In particular, $(X,x)$ is smoothable. However, the complexity of constructing the possible semistable singular fibers seemed, in the words of \cite{FriedmanMiranda},  ``rather daunting."

Looijenga's conjecture went unproven until quite recently. The first proof, due to Gross-Hacking-Keel \cite{GHK2}, is based on ideas from the Gross-Siebert program \cite{gs} for proving mirror symmetry. The second proof, due to the first author \cite{Engel}, uses ideas from symplectic geometry to construct the semistable models that are required in the approach of \cite{FriedmanMiranda}. The two proofs can be viewed, respectively, as part of the ``algebraic" and the ``symplectic" aspects of mirror symmetry for the anticanonical pair $(Y,D)$. While we have no proven claims regarding mirror symmetry, it serves as an underlying motivation for some of the constructions in this paper. 

Mirror symmetry is, very roughly, a prediction originating in physics that Calabi-Yau varieties come in mirror pairs $X$ and $X^{mir}$ on which algebraic and symplectic data are interchanged. Kontsevich \cite{hms} formalized these physical ideas in homological mirror symmetry---that the derived category of coherent sheaves on $X$ is equivalent to the Fukaya category of $X^{mir}$. The Strominger-Yau-Zaslow program \cite{syz} sought to derive a more geometric origin of mirror symmetry---that $X$ and $X^{mir}$ have dual special Lagrangian torus fibrations over a common base $B$.

Hitchin \cite{hitchin} observed that the base $B$ carries two natural integral-affine structures. These structures are interchanged by mirror symmetry. The Gross-Siebert program built on this observation, providing a new interpretation of these integral-affine structures on $B$ in terms of tropical geometry. Instead of just a Calabi-Yau variety $X$, one should consider its ``large complex structure limit"---a maximally unipotent degeneration $\mathcal{X}\rightarrow \Delta$ of $X$ over a disc. Then one of the affine structures on $B$ is identified via the tropical geometry with the dual complex $\Gamma(X_0)$ of the central fiber. This affine structure ought to agree with the affine structure on $B$ coming from the symplectic geometry of the Lagrangian torus fibration $X^{mir}\rightarrow B$.

It is possible to understand the results of this paper within such a framework. Observe that a smoothing of $\overline{V}_0-D$ is a degeneration of the open Calabi-Yau variety $Y-D$. This degeneration is the large complex structure limit. One expects then that the dual complex $B=\Gamma(Y_0)$ of the semistable resolution of the central fiber has two natural integral-affine structures. In Section 2.2 we describe one of these structures,  which is the other structure than the one considered by Gross-Hacking-Keel. Our affine structure is identified with the base of a Lagrangian torus fibration of the mirror of $Y-D$ over $B$.

While \cite{GHK2} builds a smoothing of the cusp by a ``polytope" construction from their integral-affine structure on $B$, we follow \cite{Engel} to build a semistable resolution of a smoothing by a ``fan" construction from the other structure. One advantage of the approach of \cite{GHK2} is the construction of canonical {\it theta functions} on the deformation, analogous to theta functions on degenerations of abelian varieties. Furthermore, our approach only produces a one-parameter smoothing as opposed to a versal deformation of the cusp.
However, the construction here and in \cite{Engel} has many other advantages: the geometry of the degeneration, central fiber, and significantly, the semistable resolution, are much more explicit. One can algorithmically build the central fiber of the resolution of a smoothing of the cusp. Our methods yield a strengthening of Looijenga's conjecture:

\begin{theorem}\label{0.2} Let $(X,x)$ be a cusp singularity with dual cusp of type $D$. For each deformation type $[(Y,D)]$ of an anticanonical pair, there exists a smoothing component of $(X,x)$ of type $[(Y,D)]$. Hence the number of smoothing components of $(X,x)$ is at least as large as the number of deformation types $[(Y,D)]$.
\end{theorem}

There are examples where the number of smoothing components of $(X,x)$ is in fact larger than the number of deformation types $[(Y,D)]$. For smoothable simple elliptic singularities of high multiplicity, this phenomenon was already observed by Looijenga-Wahl \cite{LooijengaWahl}. However, it seems likely that, as in the simple elliptic case, all smoothing components of a cusp singularity of type $[(Y,D)]$ are essentially isomorphic.

This paper has two goals: to analyze in detail the geometry of one parameter smoothings of cusp singularities via their semistable models, and then to use this study to describe the possible adjacencies of a cusp singularity to a union of rational double points (briefly, a \textsl{rational double point configuration}). To describe the first goal, we begin by recalling some facts and terminology about anticanonical pairs using \cite{Friedman2}  as a reference (but these many of these ideas also appear in \cite{GHK} and, in essence, date back to \cite{Looij}).   

\begin{definition} Given a deformation type $[(Y,D)]$ of anticanonical pairs, we have the positive cone
$$\mathcal{C} =\{x\in H^2(Y; \R): x^2 >0\},$$
which has two components. Exactly one component $\mathcal{C}^+$ contains the classes of ample divisors. Define $\mathcal{A}_{\text{\rm{gen}}}\subseteq \mathcal{C}^+$ to be the ``generic ample cone,"   i.e.\ the ample cone of a very general deformation of $(Y,D)$.  It is invariant under the  monodromy group of the deformation. Define $\Lambda = \Lambda(Y,D)$ be the orthogonal complement in $H^2(Y;\Z)$ of the classes $[D_i]$ of the components of $D$ and set $\Lambda_\R = \Lambda \otimes \R$. Let $\overline{\mathcal{A}}_{\text{\rm{gen}}}$ be the closure of $\mathcal{A}_{\text{\rm{gen}}}$ in the positive cone $\mathcal{C}$ and finally set $\mathcal{B}_{\text{\rm{gen}}}$ to be the interior in $\Lambda_\R$ of the intersection $\overline{\mathcal{A}}_{\text{\rm{gen}}} \cap \Lambda_\R$. 
\end{definition}

One can then describe the monodromy group of the pair $(Y,D)$ as follows:

\begin{definition} Let $\Gamma= \Gamma(Y,D)$, the group of \textsl{admissible isometries} of $(Y,D)$, be the  the group of integral isometries $\gamma$ of $H^2(Y;\Z)$ such that  $\gamma([D_i]) = [D_i]$ for all $i$, $\gamma(\mathcal{C}^+) =\mathcal{C}^+$, and 
$$\gamma (\overline{\mathcal{A}}_{\text{\rm{gen}}}(Y)) = \overline{\mathcal{A}}_{\text{\rm{gen}}}(Y).$$
In particular, $\Gamma$ acts on the sets $\mathcal{B}_{\text{\rm{gen}}}$ and $\mathcal{B}_{\text{\rm{gen}}}\cap \Lambda$.
\end{definition}

Given a one parameter smoothing $\overline{\pi}\colon (\overline{\mathcal{Y}}, \mathcal{D}) \to \Delta$ of the pair $(\overline{V}_0, D)$ over the disk $\Delta$ such that $\mathcal{D}\cong D\times \Delta$, we can find a ``good" semistable model $\pi \colon (\mathcal{Y}, \mathcal{D}) \to \Delta$ (possibly after a base change). Here good roughly means that $\mathcal{Y}$ is smooth and that the scheme-theoretic fiber $Y_0$ over $0\in \Delta$ has reduced normal crossings and satisfies $\omega_{Y_0} \cong \scrO_{Y_0}(-D)$, where $\omega_{Y_0}$ is the dualizing sheaf. We shall refer to the family $\pi \colon (\mathcal{Y}, \mathcal{D}) \to \Delta$ as a \textsl{Type III degeneration of rational surfaces}. The detailed properties of $Y_0$  are described in Definition~\ref{defTypeIII}. 

We turn next to the description of the monodromy of a Type III degeneration of rational surfaces. It is easy to show that, if $(Y,D)$ is a negative-definite anticanonical pair, then there is an exact sequence
$$0 \to \Z  \to H_2(Y-D;\Z) \to \Lambda \to 0,$$
where $\Z = \Z\cdot \gamma$ is the radical of intersection pairing $\bullet$ on $H_2(Y-D;\Z)$. In particular, we can define the pairing between an element of $\Lambda$ and an element of $H_2(Y-D;\Z)$. Our first main result is then roughly as follows:

\begin{theorem}\label{0.5} Let $\pi \colon (\mathcal{Y}, \mathcal{D}) \to \Delta$ be a Type III degeneration of rational surfaces as above. Let $T$ be the monodromy of the fiber, which acts on $\Lambda$ and on $H_2(Y-D;\Z)$. Then: \begin{enumerate}
\item[\rm(i)] For the action of $T$ on $H_2(Y; \Z)$, $T=\Id$
\item[\rm(ii)] There is a unique class $\lambda \in \Lambda$ such that, for all $x\in H_2(Y-D;\Z)$, 
$$T(x) = x-\langle \lambda , x\rangle \gamma.$$
\item[\rm(iii)] With $\lambda$ as above,  $\lambda^2 = v$, where $v$ is the number of triple points of the singular fiber $Y_0$.
\item[\rm(iv)] For a unique choice of sign, $\pm \lambda\in \mathcal{B}_{\text{\rm{gen}}}\cap \Lambda$.
\item[\rm(v)] The class $\lambda$ generates over $\Q$ the cokernel of the specialization map $\operatorname{sp}\colon H^2(Y_0; \Q) \to H^2(Y_t; \Q)$. \end{enumerate}
\end{theorem}

It is natural to always choose the sign so that $\lambda\in \mathcal{B}_{\text{\rm{gen}}}\cap \Lambda$. The class $\lambda$ is only well-defined modulo the action of $\Gamma(Y,D)$. We define the \textsl{monodromy invariant} to be the image of $\lambda$ in $\Gamma\backslash(\mathcal{B}_{\text{\rm{gen}}}\cap \Lambda)$, but continue by abuse of notation to denote it by $\lambda$. Our next result, which is a significant generalization of Theorem~\ref{0.2},  states that all possible $\lambda$ arise.  

\begin{theorem}\label{0.6} Let $D'$ be the minimal resolution of a cusp singularity for which the dual $D$ sits on a rational surface. Then for every deformation type of anticanonical pairs $(Y,D)$ and for every  class $\lambda\in \mathcal{B}_{\text{\rm{gen}}}\cap \Lambda$, there exists a Type III degeneration of rational surfaces whose general fiber is deformation equivalent to $(Y,D)$ and whose monodromy invariant is $\lambda$ mod $\Gamma(Y,D)$.
\end{theorem} 

Though the statements of Theorem \ref{0.2} and Theorem \ref{0.6} are purely algebraic, the method of proof relies on symplectic geometry, and in particular the geometry of Lagrangian torus fibrations. There is a naturally defined singular integral-affine structure on the dual complex $\Gamma(Y_0)$ of the central fiber of a Type III degeneration of rational surfaces, defined in Section 2.2. When $Y_0$ is \textsl{generic}, see Definition \ref{generic}, work of Symington \cite{Symington} allows us to construct a Lagrangian \textsl{almost toric fibration} of a symplectic $4$-manifold with a degenerate symplectic anticanonical divisor $D$: $$\mu\colon (X,D,\omega)\rightarrow \Gamma(Y_0).$$

In Section 3.3, we prove that $(X,D)$ and $(Y_t,D_t)$ are naturally diffeomorphic, and that the fiber class of $\mu$ is identified under this diffeomorphism with $\gamma$ while the class $[\omega]$ is identified with $\lambda$. The validity of this result is essentially a fluke of two dimensions---the $4$-manifold $(X,D,\omega)$ is the symplectic mirror of the algebraic degeneration $\mathcal{Y}\rightarrow \Delta$ and thus ought to carry a dual Lagrangian torus fibration, but dual fibrations of $2$-tori are diffeomorphic because the inverse transpose of a $2\times 2$ matrix happens to be conjugate to the original matrix. Regardless, it is this result which allows us to bound below the number of smoothing components, by bounding below the diffeomorphism types of general fibers $(Y_t,D_t)$.

The identification of the dual complex $\Gamma(Y_0)$ with the base of a Lagrangian torus fibration also provides a method for constructing smoothings. If one can construct a Lagrangian fibration $$\mu\colon (X,D,\omega)\rightarrow S^2$$ such that the natural integral-affine structure on the base of $\mu$ admits an appropriate triangulation, then the base is, as an integral-affine manifold, the dual complex of the central fiber of a Type III degeneration. Thus, producing a degeneration with prescribed monodromy invariant $\lambda$ reduces to the construction of a Lagrangian torus fibration $(X,D,\omega)\rightarrow S^2$ such that $[\omega]=\lambda$. This is the goal of Sections 4 and 5.

By analogy with the case of $K3$ surfaces, it is natural to conjecture that the monodromy invariant $\lambda$ mod $\Gamma$ is a complete invariant of the singular fiber of a Type III degeneration of rational surfaces up to locally trivial deformation and standard birational modifications. We discuss this conjecture in more detail in Section 5.

Finally, we apply the above results to the study of rational double point adjacencies of cusps. First, we recall Looijenga's definition of certain distinguished elements of square $-2$ in $\Lambda(Y,D)$.

\begin{definition} Let $(Y,D)$ be an anticanonical pair, and let $\beta \in \Lambda(Y,D)$ with $\beta^2=-2$. Then $\beta$ is a \textsl{Looijenga root} or briefly a \textsl{root} if there exists a deformation of $(Y,D)$ over a connected base $S$ and a fiber $(Y_s, D_s)$ over $s\in S$ such that, in $\Lambda(Y_s, D_s)$, $\beta$ is the class of a smooth curve $C\cong \Pee^1$ disjoint from $D_s$. We let $R\subseteq \Lambda(Y,D)$ be the set of roots. If $D$ is negative-definite and has at most $5$ components (equivalently, the dual cusp has multiplicity at most $5$), then $R$ is a generalized root system of hyperbolic type, and in particular it spans $\Lambda$. However, when the number of components of $D$ is larger than $5$, there are many possibilities for the behavior of $R$.
\end{definition}

We can then describe the possible rational double point adjacencies  of a cusp $(X,x)$ as follows. Let $D$ be the minimal resolution of the dual cusp and let $(Y,D)$ be an anticanonical pair of type $D$. 

\begin{definition} Let $\Upsilon$ be a negative definite sublattice of $\Lambda(Y,D)$. Then $\Upsilon$ is \textsl{good} if 
\begin{enumerate}
\item[\rm(i)] $\Upsilon$ is spanned by elements of $R$. 
\item[\rm(ii)] There exists a homomorphism $\varphi\colon \Lambda(Y,D) \to \C^*$ such that $\Ker \varphi\cap R = \Upsilon\cap R$.
\end{enumerate}
The lattice $\Upsilon$  determines an RDP configuration (possibly consisting of more than one singular point) by taking the appropriate type (i.e.\ Dynkin diagram) of a set of simple roots for $\Upsilon \cap R$. We say that the corresponding rational double point configuration  is \textsl{of type $\Upsilon$}. 
 \end{definition}
 
 We can then state our result about adjacencies of cusps to rational double points as follows:
 
 \begin{theorem}\label{0.9} Suppose that $(X,x)$ is adjacent to a rational double point configuration on a smoothing component of type $[(Y,D)]$. Then the components of the exceptional fibers of a minimal resolution span a good negative definite sublattice $\Upsilon$ of $\Lambda(Y,D)$, and the rational double point configuration is of type $\Upsilon$.
 
 Conversely, if $\Upsilon$ is a good negative definite sublattice of $\Lambda(Y,D)$, then there exists an adjacency of $(X,x)$ to a rational double point configuration of type $\Upsilon$ on some  smoothing component of type $[(Y,D)]$
\end{theorem}

An equivalent and more geometric form of the above characterization is given in Proposition~\ref{meaningofgood}.

\begin{remark} It is natural to conjecture that, given the good sublattice $\Upsilon$ of $\Lambda(Y,D)$, then there exists an adjacency of $(X,x)$ to a rational double point configuration of type $\Upsilon$ on \textit{every} smoothing component of type $[(Y,D)]$.
\end{remark}

The methods of this paper say nothing  about adjacencies of the cusp $(X,x)$ to non-rational singularities. If $(X,x)$ is adjacent to a union of singularities, then at most one can be non-rational and all others are rational double points. Moreover, the non-rational singularity, if it exists, is either a cusp singularity or a simple elliptic singularity. Wahl described various adjacencies between elliptic singularities in \cite{Wahl}, and P.\ Hacking (unpublished) has shown that these are the only possible adjacencies for cusp singularities. More generally, one can make fairly precise conjectures about the possible adjacencies from a cusp to a union of an elliptic singularity and a number of rational double points (conjecturally a union of $A_k$ singularities), which generalize Looijenga's results for multiplicity at most $5$. However, we shall not do so here.

A description of the contents of this paper is as follows: Section 1 deals with certain preliminary results on the geometry and topology of anticanonical pairs. In Section 2, we study Type III degenerations of rational surfaces and in particular the possible singular fibers $Y_0$. Such a fiber has reduced normal crossings, and there is a unique component isomorphic to an Inoue surface $V_0$, with the dual cycle $D$ contained in the smooth locus of $Y_0$. We call the pair $(Y_0, D)$ a Type III anticanonical pair. Although the Inoue surface $V_0$ is not K\"ahler, there are analogues of a mixed Hodge structure and a limiting mixed Hodge structure for $Y_0$. We study the relevant spectral sequences and determine when they degenerate. A general theme is that it is more useful to consider the topology of the pair $(Y_0, D)$ as opposed to that of $Y_0$. We also introduce a singular integral-affine structure on the dual complex $\Gamma(Y_0)$, and describe the almost toric fibration $(X,D,\omega)\rightarrow \Gamma(Y_0)$. We also describe the deformation theory of the pair $(Y_0, D)$ in an appropriate sense. Section 3 deals with the monodromy of a smoothing of the pair $(Y_0, D)$. In particular, we prove Theorem~\ref{0.5}. We also prove that $(X,D,\omega)$ is naturally diffeomorphic to the general fiber $Y_t$ of the Type III degeneration of rational surfaces with central fiber $Y_0$. Theorem \ref{0.2} follows. Furthermore, we show that the monodromy invariant $\lambda$ from Theorem \ref{0.5} is identified under this diffeomorphism with the class of the symplectic form $[\omega]$.  In Sections 4 and 5 we prove Theorem~\ref{0.6}. Section 4 collects results concerning birational modification and base change of Type III degenerations, and records their effect on the integral-affine structure defined on $\Gamma(Y_0)$. In Section 5, we construct a Type III degeneration with arbitrary monodromy invariant $\lambda\in \mathcal{B}_{\text{\rm{gen}}}\cap\Lambda$, the method being to first construct the dual complex of the central fiber via surgeries on integral-affine surfaces. Section 6 uses the description of the monodromy and the limiting cohomology to give an asymptotic formula for the period map of a general fiber in a Type III degeneration of rational surfaces. Sections 7 and 8 are devoted to a study of an analogue of the period map for a Type III anticanonical pair. We calculate the differential of the period map in Section 7 and, using this calculation, show in Section 8 that the period map is surjective for
topologically trivial deformations of the pair $(Y_0, D)$. Moreover, for suitable lattices $\Upsilon$ inside an appropriate quotient of $\Pic Y_0$, we show that we can arrange a deformation of the pair $(Y_0, D)$ so that the image of $\Upsilon$ is in the kernel of the period map for all smooth fibers $(Y_t, D)$. Finally, in Section 9, we prove Theorem~\ref{0.9}. The strategy of the proof is as follows. Given a good lattice $\Upsilon$ in $\Lambda(Y, D)$, we show that there exists a $\lambda \in \mathcal{B}_{\text{\rm{gen}}}\cap \Lambda$  such that $\Upsilon$ is orthogonal to $\lambda$. By Theorem~\ref{0.6}, there exists a Type III degeneration of rational surfaces $(\mathcal{Y}, \mathcal{D})$ with invariant $\lambda$.  Using (v) of Theorem~\ref{0.5}, we can identify $\Upsilon$ with a quotient of $\Pic Y_0$. By the surjectivity of the period map for $(Y_0,D)$, we can assume after deforming $(Y_0, D)$ that $\Upsilon$  is exactly the kernel of the period map and that the same is true for all nearby smooth fibers $(Y_t, D)$. It then follows by invoking the full strength of Shepherd-Barron's contraction result that, after blowing down all components of the central fiber not equal to $V_0$ as well as some curves in the general fiber, we exactly contract the curves in a general fiber $Y_t$ whose classes lie in $\Upsilon$. Thus we obtain an adjacency of the cusp to a rational double point of type  $\Upsilon$.

\medskip
\noindent\textbf{Acknowledgements.} It is a pleasure to thank Eduard Looijenga and Radu Laza for helpful discussions and correspondence. In addition, we thank for the referee for their careful reading and comments.

\section{Preliminaries}

\subsection{Topology of anticanonical pairs}

\begin{definition} Let $D=D_1+\dots+D_r$ be a reduced cycle of smooth rational curves for $r\geq 2$, or an irreducible nodal curve whose normalization is smooth rational for $r=1$. The integer $r=r(D)$ is the \textsl{length} of $D$. By convention, we take the indices $i$ mod $r$ and  $D_i\cap D_j =\emptyset$ for $j\neq i\pm 1$. An \textsl{orientation} for $D$ is a choice of generator for $H_1(D;\Z) \cong \Z$. For $i\geq 3$, an orientation determines and is determined by an indexing of the components of $D$ as above up to cyclic permutation, and we always assume that the indexing and the orientation are compatible in this sense. Given an indexing of the components, we always assume that there is an associated \textsl{type} of $D$, which by definition is a sequence $(d_1, \dots, d_r)$ of integers $d_i$, such that $d_i \leq -2$ for all $i$ and $d_i \leq -3$ for at least one $i$. If $D$ is given as a Cartier divisor in a surface, we always assume that $d_i = D_i^2$. 

An \textsl{anticanonical pair} or simply \textsl{pair} $(Y,D)$ is a rational surface $Y$ with an anticanonical divisor $D\in |-K_Y|$ which forms a reduced cycle of smooth rational curves $D=D_1+\dots+D_r$ unless $r=1$ in which case $D$ is a nodal rational curve. The pair $(Y,D)$ is of \textsl{type} $D$ if the type of $D$ is $(D_1^2, \dots, D_r^2)$. We say $(Y,D)$ is \textsl{toric} if $V$ is a smooth toric surface and $D$ is the toric boundary.  \end{definition}

Let $(Y,D)$ be an anticanonical pair. Enumerate the components of $D$ as $D_1, \dots, D_r$ and the set of nodes $T$ as $t_1, \dots, t_r$, with the convention that $D_i \cap D_{i+1} = \{t_i\}$. If $[D_i]$ denotes the class of the component $D_i$ in $H^2(Y;\Z)$, let $\Lambda =\Lambda(Y,D) = \{[D_1], \dots, [D_r]\}^\perp \subseteq H^2(Y; \Z).$ By definition $\Lambda$ is a primitive sublattice of $H^2(Y; \Z)$. Its dual lattice $\Lambda\spcheck$ is described as follows: let $\widehat{\operatorname{span}}\{[D_1], \dots, [D_r]\}$ be the saturation of the subgroup of $H^2(Y; \Z)$ generated by $[D_1], \dots, [D_r]$. Then $$\Lambda\spcheck = H^2(Y; \Z)/\widehat{\operatorname{span}}\{[D_1], \dots, [D_r]\}.$$

Via the Gysin spectral sequence (the Leray spectral sequence applied to the inclusion $i\colon Y-D \hookrightarrow Y$) there is an  exact sequence
$$0 \to H^2(Y; \Z)/ \operatorname{span} \{[D_1], \dots, [D_r]\}\to H^2(Y-D; \Z) \to \Z \to 0,$$
where the right hand $\Z$ is the kernel of the Gysin map.
Thus, if $\overline{H}^2(Y-D; \Z)$ denotes $H^2(Y-D; \Z)$ mod torsion, there is an exact sequence
$$0 \to \Lambda\spcheck \to \overline{H}^2(Y-D; \Z) \to \Z \to 0.$$
Dually, there is an exact sequence
$$0 \to \Z \to H_2(Y-D; \Z) \to \Lambda \to 0,$$
 arising from the exact sequence of the pair $(Y, Y-D)$:
$$0=H_3(Y; \Z) \to H_3(Y, Y-D; \Z) \to H_2(Y-D; \Z) \to H_2(Y; \Z).$$
Poincar\'e-Alexander duality identifies $H_3(Y, Y-D; \Z)$ with $H^1(D;\Z) = \Z$, and  the image of $H_2(Y-D; \Z) \to H_2(Y; \Z)$ corresponds exactly to those $2$-cycles which have algebraic intersection $0$ with every component of $D$, and thus via Poincar\'e duality with $\Lambda$. In terms of compactly supported cohomology, we have a commutative diagram
$$\minCDarrowwidth20pt \begin{CD}
H_3(Y, Y-D; \Z) @>>> H_2(Y-D; \Z) @>>>  H_2(Y; \Z) @. {}
\\
 @V{\cong}VV  @V{\cong}VV @V{\cong}VV @.\\
 H^1(D;\Z) @>>> H^2_c(Y-D;\Z) @>>> H^2(Y;\Z) @>>> H^2(D;\Z).
\end{CD}$$

Let $V$ be a small tubular neighborhood of $D$ in $Y$, homotopy equivalent to $D$, so that $Y-V$ is homotopy equivalent to $Y-D$. Let $\partial =\partial V$. Using the Mayer-Vietoris sequence applied to the decomposition $Y=(Y-D)\cup V$,
one easily computes that $H^1(\partial;\Z)\cong H^1(D;\Z) \cong \Z$ and that $H^2(\partial; \Z)$ has rank one. The exact sequence for the pair $(Y-V, \partial)$, which we shall write as $(Y-D, \partial)$, is then (all groups with $\Z$-coefficients)
$$H^1(Y-D)=0 \to H^1(\partial) \to H^2(Y-D, \partial) \to H^2(Y-D) \to H^2(\partial).$$
Lefschetz duality identifies $H^2(Y-D, \partial)$ with $H_2(Y-D; \Z)$ and identifies the image of $H^1(\partial)$ with the subgroup $\Z$ above.
Moreover, Lefschetz duality is compatible with 
Poincar\'e duality on $\partial$, which identifies $H^1(\partial)$ with $H_2(\partial)$, and with the exact sequence of the pair $(Y-D, \partial)$, so that the above exact sequence is identified with
$$H_2(\partial) \to H_2(Y-D) \to H_2(Y-D, \partial) \to H_1(\partial) \to 0.$$  Poincar\'e duality identifies $\overline{H}_2(Y-D, \partial)$ with $\overline{H}^2(Y-D; \Z)$, and hence there is an inclusion of $\Lambda\spcheck$ in $\overline{H}_2(Y-D, \partial)$ for which the following diagram commutes:
$$\minCDarrowwidth20pt\begin{CD}
H_2(Y-D;\Z) @>>> \overline{H}_2(Y-D, \partial)\\
@VVV @AAA\\
\Lambda @>>> \Lambda\spcheck.
\end{CD}$$
 
 It is easy to check that a generator of $H_2(\partial)$ is the class $\gamma$ which is the tube (with respect to some Hermitian metric) over a simple closed curve in $D_i$ enclosing a double point: for example, there exists a curve $\gamma'$ on $\partial$ projecting onto the generator of $H_1(D;\Z)$, and then $\gamma \bullet \gamma' = \pm 1$. Identifying $\gamma$ with its image in $H_2(Y-D; \Z)$, we see that the kernel of $H_2(Y-D; \Z) \to \Lambda$ is $\Z\gamma$. Note that $\gamma$ is only well-defined up to sign, but is determined by an orientation of the cycle $D$.

\subsection{The generic ample cone and effective cone}

We recall the following terminology from \cite{Friedman2} (see also \cite{GHK}). Let $\mathcal{C}\subseteq H^2(Y;\R)$ be the positive cone and let $\mathcal{C}^+$ be the component of $\mathcal{C}$ containing the classes of ample divisors. Inside $\mathcal{C}^+$, we have the convex cone $\overline{\mathcal{A}}_{\text{\rm{gen}}}$, which is the closure in $\mathcal{C}^+$ of the ample cone for a generic small deformation of $Y$. We define a \textsl{$-2$-curve} on $Y$ to be a smooth rational curve  on $Y$ of self-intersection $-2$ which is not a component  of $D$, and define $Y$ to be \textsl{generic} if there are no $-2$-curves on $Y$. If $Y$ is  generic, then 
\begin{gather*}
\overline{\mathcal{A}}_{\text{\rm{gen}}} = \{x\in \mathcal{C}^+: x\cdot [E] \geq 0 \text{ for all exceptional curves $E$ }\\
 \text{ and $x\cdot [D_i] \geq 0$ for all $i$}\}.
\end{gather*}
For a generic $Y$, we may replace the condition $x\cdot [E] \geq 0$ for all exceptional curves $E$ by the condition $x\cdot \alpha \geq 0$ for all effective numerical exceptional curves $\alpha$, i.e.\ cohomology classes $\alpha$ such that $\alpha^2=\alpha\cdot [K_{Y}] = -1$ and $\alpha$ is the class of an effective divisor, or equivalently $\alpha \cdot [H] \geq 0$ for some ample divisor $H$ \cite[\S5]{Friedman2}. We set $\mathcal{A}_{\text{\rm{gen}}}$ equal to the interior of $\overline{\mathcal{A}}_{\text{\rm{gen}}}$.

We can omit the assumption that $x\in \mathcal{C}$ lies in the component $\mathcal{C}^+$:

\begin{lemma}\label{altcharA}
If $x\in \mathcal{C}$, and $x\cdot \alpha \geq 0$ for all effective numerical exceptional curves $\alpha$ and $x\cdot [D_i] \geq 0$ for all $i$, then $x\in \mathcal{C}^+$.
\end{lemma}
\begin{proof} An easy induction on the number of blowups required to obtain $Y$ from a minimal surface shows that the union of the $D_i$ and the effective numerical exceptional curves supports a nef and big divisor $H$. Thus, for $x$ as in the hypothesis of the lemma, $x\cdot H \geq 0$ and so $x\in \mathcal{C}^+$.
\end{proof}

\begin{definition}\label{defB} Let $\mathcal{B}_{\text{\rm{gen}}}$ be the interior of $\overline{\mathcal{A}}_{\text{\rm{gen}}}\cap \Lambda_\R$ in $\Lambda_\R =\Lambda \otimes \R$. Thus by definition, for $x\in\Lambda_\R$,  $x\in \mathcal{B}_{\text{\rm{gen}}}$ $\iff$ $x\in \Lambda_\R\cap \mathcal{C}^+$ and $x\cdot \alpha > 0$ for all effective numerical exceptional curves $\alpha$. Note that, since the set of walls defined by the numerical exceptional classes and the $[D_i]$ is locally finite in $\mathcal{C}^+$, $\mathcal{B}_{\text{\rm{gen}}}$ is a nonempty open convex subset of $\Lambda_\R$.
\end{definition}

\begin{definition}\label{internal} A \textsl{corner blow-up} of $(Y,D)$ is a blow-up at a node of $D$ and an \textsl{internal blow-up} of $(Y,D)$ is a blow-up on the smooth locus of $D$. Both the corner and internal blow-up have an anticanonical divisor $\pi^*D-E$ where $\pi$ is the blow-up and $E$ is the exceptional divisor. An \textsl{internal exceptional curve} is an exceptional curve which is not a component of $D$. \end{definition} 

The dual of the ample cone is the effective cone. The extremal rays of the effective cone are the classes of curves of negative self-intersection---generically, the internal exceptional curves and components of $D$. With a mild assumption, every effective divisor on $Y$ is linearly equivalent to a union of components $D$ and disjoint internal exceptional curves:

\begin{proposition}\label{effdecomp} Let $(Y,D)$ be an anticanonical pair with $r(D)\geq 3$ and no $-2$-curves, besides components of $D$. If $A$ is an effective curve on $Y$, then $A$ is linearly equivalent to a curve of the form 
$$\sum a_jD_j+\sum b_iE_i,$$ 
where  the $E_i$ are disjoint internal exceptional curves and $a_j, b_j\geq 0$.
 \end{proposition}

\begin{proof}

We begin with the following lemma:

\begin{lemma}\label{effdecomplemma} Let $G$ be an effective, nef, and nonzero  divisor on $Y$. Then there exists a component $D_j$ of $D$ and an element of 
$|G|$ of the form $B + D_j$, where $B$ is effective.
\end{lemma}
\begin{proof} We may clearly assume that no component of $D$ is a fixed component of $|G|$. Since $H^2(Y; \scrO_Y(G))=0$, Riemann-Roch implies that
$$h^0(Y; \scrO_Y(G)) = h^1(Y; \scrO_Y(G)) + \frac12(G^2 + G\cdot D) + 1.$$
Hence $h^0(Y; \scrO_Y(G)) \geq 2$ if $G\cdot D \neq 0$. If $G\cdot D = 0$, then $h^1(Y; \scrO_Y(G))\geq 1$ by \cite[Lemma 4.13]{Friedman2}, and so again $h^0(Y; \scrO_Y(G)) \geq 2$. It follows that $\dim |G| \geq 1$, so that, for every point $y$ of $Y$, there exists a curve in $|G|$ passing through  $y$. Hence, if there exists a component $D_j$ of $D$ such that $G\cdot D_j =0$, then there exists a curve in $|G|$ meeting $D_j$ and thus necessarily of the form $B+D_j$ for some effective $B$. In particular, if $G\cdot D \leq 2$, then the assumption $r(D)\geq 3$ implies that such a component $D_j$ must exist.

If $G^2 = 0$, then by \cite[Theorem 4.19]{Friedman2}, $G$ is either of the form $kC$, $k\geq 1$, where $C\cdot D = 2$, or of the form $kE$ where $E\cdot D =0$. In either case, there is a component $D_j$ of $D$ such that $G\cdot D_j =0$ and we are done by the previous paragraph. Thus, we may assume that $G^2> 0$, i.e.\ that $G$ is big, and that $G\cdot D_j > 0$ for every $j$. By Ramanujam's lemma (cf.\ \cite[Lemma 4.10]{Friedman2}), the restriction map $H^0(Y; \scrO_Y(G)) \to H^0(D; \scrO_Y(G)|D)$ is surjective. The line bundle $\scrO_Y(G)|D$ has positive degree on every component of $D$. Choosing a component $D_j$, it is then easy, using the assumption $r(D) \geq 3$, to construct a nonzero section of $\scrO_Y(G)|D$ which vanishes identically on $D_j$. Lifting this section to a section of $\scrO_Y(G)$ then gives a curve in $|G|$ which contains $D_j$ as desired.
\end{proof}

Returning to the proof of Proposition~\ref{effdecomp}, every curve in   $|A|$ is of the form $\sum a_jD_j+\sum b_iC_i$, where the $a_j, b_i\geq 0$ and the $C_i$ are distinct irreducible curves which are not components of $D$.  Clearly the sum $\sum a_j$ is bounded. We may thus choose such a curve for which  $\sum a_j$ is maximal,  possibly zero.  Since there are no $-2$-curves on $Y$, one of the following holds:
\begin{enumerate}
\item[(i)] Some $C_i$ is not exceptional.
\item[(ii)] Every $C_i$ is exceptional but two intersect.
\item[(iii)] The $C_i$ are disjoint exceptional curves. 
\end{enumerate}
If $C_i$ is not exceptional, then it is nef and we can apply Lemma~\ref{effdecomplemma} to find an effective curve $C_i'$ linearly equivalent to $C_i$ which contains $D_k$ for some $k$. But then  there is a curve in $|A|$ of the form $\sum a_j'D_j+\sum b_iC_i'$ as above with $\sum_ja_j ' > \sum a_j$, contradicting the hypothesis on $\sum a_j$ is maximal. Likewise,  if there exist $i_1, i_2$ such that $C_{i_1}\cdot C_{i_2} >0$, then $G= C_{i_1} + C_{i_2}$ is nef, and we again reach a contradiction by Lemma~\ref{effdecomplemma}. Thus we are in Case (iii), which is the statement of Proposition~\ref{effdecomp}.
\end{proof}

\begin{remark}
It is easy to see that the assumption $r(D)\geq 3$ is necessary. For example, the class of a fiber  on the a minimal anticanonical pair $(\F_N,D_1+D_2)$, where $D_1$ and $D_2$ are both irreducible sections, cannot be expressed as a linear combination of the components of $D$.
\end{remark}

\section{Type III degenerations}

\subsection{Preliminaries}

\begin{definition}\label{defTypeIII} A \textsl{Type III degeneration of anticanonical pairs} is a proper morphism $\pi\colon (\mathcal{Y}, \mathcal{D})\to \Delta$, where $\mathcal{Y}$ is a smooth complex $3$-fold, $\Delta$ is the unit disk, and $\mathcal{D}\rightarrow \Delta$ is a locally trivial relative normal crossings divisor in $\mathcal{Y}$ whose intersection with the fiber $Y_t$ is denoted by $D_t$ for all $t$, such that:

\begin{enumerate}
\item[(i)] The general fibers $\pi^{-1}(t) = (Y_t, D_t)$, $t\neq 0$ are anticanonical pairs and the central fiber $\pi^{-1}(0) = (Y_0, D_0)$ with $Y_0= \bigcup_{i=0}^{f-1}V_i$ is a reduced normal crossings divisor.
\item[(ii)] $V_0$ is the Hirzebruch-Inoue surface with dual cycles $D$ and $D'$. The normalization $\tilde{V}_i$ of $V_i$ is a smooth rational surface for $i\neq 0$. Furthermore $D_0=D$.
\item[(iii)] Let $C_i$ denote the union of the double curves $C_{ij}=V_i\cap V_j$ which lie on $V_i$. Then the inverse image of $C_i$ under the normalization map is an anticanonical divisor on $\tilde{V_i}$ for $i>0$ and $C_0=D'$.  Equivalently, $K_{\mathcal{Y}} = \scrO_{\mathcal{Y}}(-\mathcal{D})$.
\item[(iv)] The triple point formula holds: $$\left(C_{ij}\big{|}_{\tilde{V}_i}\right)^2+\left(C_{ij}\big{|}_{\tilde{V}_j}\right)^2=\twopartdef{-2}{C_{ij} \textrm{ is smooth}}{0}{C_{ij} \textrm{ is nodal.}}$$
\item[(v)] The dual complex $\Gamma(Y_0)$ is a triangulation of the sphere.
\end{enumerate}

A \textsl{Type III anticanonical pair} is a pair $(Y_0, D_0)$, where $Y_0$ is a reduced surface with simple normal crossings and $D_0$ is a Cartier divisor on $Y_0$, satisfying (ii)-(v) above. We say $(Y_0, D_0)$ is \textsl{$d$-semistable} if in addition 
$$T^1_{Y_0} =Ext^1(\Omega^1_{Y_0}, \scrO_{Y_0}) \cong \scrO_{(Y_0)_{\text{sing}}},$$
which is an analytic refinement of the triple point formula. By \cite[(2.6) and (2.9)]{FriedmanMiranda}, every Type III anticanonical pair can be deformed via a locally trivial deformation to one which is $d$-semistable, and (as we shall discuss later) every $d$-semistable Type III anticanonical pair can be smoothed in a one parameter family with smooth total space.\end{definition}

\begin{notation} To simplify the notation, we will henceforth suppress the tildes on $\tilde{V}_i$ so that $(V_i,C_i)$ denotes a smooth anticanonical pair. In addition, we introduce the convention $$C_{ij}=C_{ij}\big{|}_{V_i}\,\,\textrm{ and }\,\,C_{ji}=C_{ij}\big{|}_{V_j}$$ so that $C_{ij}$ always denotes a curve on the smooth surface $V_i$. Then $C_{ij}$ and $C_{ji}$ have the same image in $Y_0$ but may not be isomorphic. In fact, the image of $C_{ij}$ in $Y_0$ is nodal if and only if exactly one of $C_{ij}$ or $C_{ji}$ is nodal. We define \vspace{-3pt}$$c_{ij}:=\twopartdef{-c_{ij}^2}{r(C_i)\geq 2}{2-c_{ij}^2}{r(C_i)=1.}$$ Then, the triple point formula states that $c_{ij}+c_{ji}=2$ in all cases. \end{notation}

\begin{definition} The \textsl{charge} of a pair $(V,C)$ is the integer $$Q(V,C):=12+\sum (c_i-3)$$ where $$c_i= \twopartdef{-C_i^2}{r(C)>1}{2-C_i^2}{r(C)=1}.$$ Then $Q(V,C)=0$ if and only if $(V,C)$ is toric and otherwise $Q(V,C)>0$. A corner blow-up keeps the charge constant, while an  internal blow-up increases the charge by one. \end{definition}

\begin{remark} When $(V,C)$ is an anticanonical pair, 
$$Q(V,C)=\chi_{\text{top}}(V-C),$$ where $\chi_{\text{top}}$ is the (topological) Euler characteristic. When $V$ is the Hirzebruch-Inoue surface, and $C$ is one of the cycles on $V$, this formula does not hold. \end{remark}

\begin{proposition}[Conservation of Charge] \label{charge} Let $(Y_0,D_0)$ be a Type III anticanonical pair. Then  $\displaystyle\sum Q(V_i,C_i)=24.$ \end{proposition}

\begin{proof} See \cite{FriedmanMiranda}, Proposition 3.7. \end{proof}

\subsection{The integral-affine structure on the dual complex}

\begin{definition} An \textsl{integral-affine manifold} of dimension $n$ is a manifold with charts to $\R^n$ such that the transition functions are valued in the integral-affine transformation group $SL_n(\Z)\ltimes \R^n$. A \textsl{lattice manifold} is an integral-affine manifold whose transition functions are in $SL_n(\Z)\ltimes \Z^n$. \end{definition}

Note that a lattice manifold contains a lattice of points with integer coordinates in some (any thus any) chart. We will endow $\Gamma(Y_0)$ with the structure of a triangulated lattice surface with singularities. Define the following notation for the dual complex:

\begin{enumerate}
\item[(i)] The vertices $v_i$ correspond to the components $V_i$,
\item[(ii)] the directed edges $e_{ij}=(v_i,v_j)$ correspond to double curves $C_{ij}$,
\item[(iii)] the triangular faces $f_{ijk}=(v_i,v_j,v_k)$ correspond to triple points $p_{ijk}$.
\end{enumerate} 

\noindent We denote the $i$-skeleton of $\Gamma(Y_0)$ by $\Gamma(Y_0)^{[i]}.$ There is a natural (non-singular) lattice manifold structure on $$\Gamma(Y_0)- \{v_i\colon Q(V_i,C_i)>0\textrm{ or }\,i=0\}$$ which we now define, see also Remark 1.11 of Version 1 of \cite{GHK2}. We declare each triangular face $f_{ijk}$ to be integral-affine equivalent to a \textsl{basis triangle}, that is, a lattice triangle of area $1/2$. Then, we define the integral-affine structure on the union of two triangular faces $f_{ijk}$ and $f_{ik\ell}$ that share an edge $e_{ik}$ by gluing two basis triangles in the plane along their shared edge. We glue in the unique way such that $$c_{ik}e_{ik}=e_{ij}+e_{i\ell},$$ where we view the directed edges $e_{ij}$, $e_{ik}$ and $e_{i\ell}$ as integral vectors. By Propositions 2.2 and 2.3 of \cite{Engel}, this defines a unique integral-affine structure on $\Gamma(Y_0) -  \Gamma(Y_0)^{[0]}$ which naturally extends to the vertices $v_i$ corresponding to toric pairs $(V_i,C_i)$. 

\begin{definition} Let $(Y,D)$ be a pair with cycle components $D=D_1+\dots+D_n.$ The \textsl{pseudo-fan} of $(Y,D)$ is an integral-affine surface $\mathfrak{F}(Y,D)$ PL-equivalent to the cone over the dual complex of $D$. For each intersection point $t_i=D_i\cap D_{i+1}$ there is an associated face $f_{i,i+1}$ of this cone integral-affine equivalent to a basis triangle. Let $e_i$ denote the primitive integral vector that originates at the cone point corresponding to some component $D_i$. We glue $f_{i-1,i}$ and $f_{i,i+1}$ together by an element of $SL_2(\Z)$ in the unique manner such that $e_{i-1}+e_{i+1}=d_ie_i$ where $$d_i =\twopartdef{-D_i^2}{n>1}{2-D_i^2}{n=1}.$$ The integral-affine structure has at most one singularity, at the cone point. Compare to Section 1.2 of \cite{GHK2}.  \end{definition}

Note that the union of the triangles containing a vertex $v_i\in \Gamma(Y_0)$ is the pseudo-fan $\mathfrak{F}(V_i,C_i)$. 

\begin{example} We can visualize the lattice structure on $\Gamma(Y_0)$ by cutting it open along a spanning tree of the singular vertices. The resulting disc has an integral-affine chart to $\R^2$ whose image is a lattice polygon. Figure \ref{example} below depicts such an open chart. The intersection complex of $Y_0$ is overlaid by the dual complex, and each component of the anticanonical cycle of $(V_i,C_i)$ is labelled by its self-intersection. The surface $Y_0$ visibly satisfies conditions (ii)-(v) of Definition \ref{defTypeIII}.

There are three anticanonical pairs $(V_i,C_i)$ with charge $Q(V_i,C_i)=1$, whose cycles of self-intersections are $(3,-1,-1,-2)$, $(6,-1,-1,-5)$, and $(2, -1, -1, -1)$. The Hirzebruch-Inoue surface $V_0$ meets the other components of $Y_0$ along the cycle $D'=C_0$ whose components have self-intersections $(-6,-9)$. Every other component of $Y_0$ is toric.

\begin{figure}
\begin{centering}
\includegraphics[width=4in]{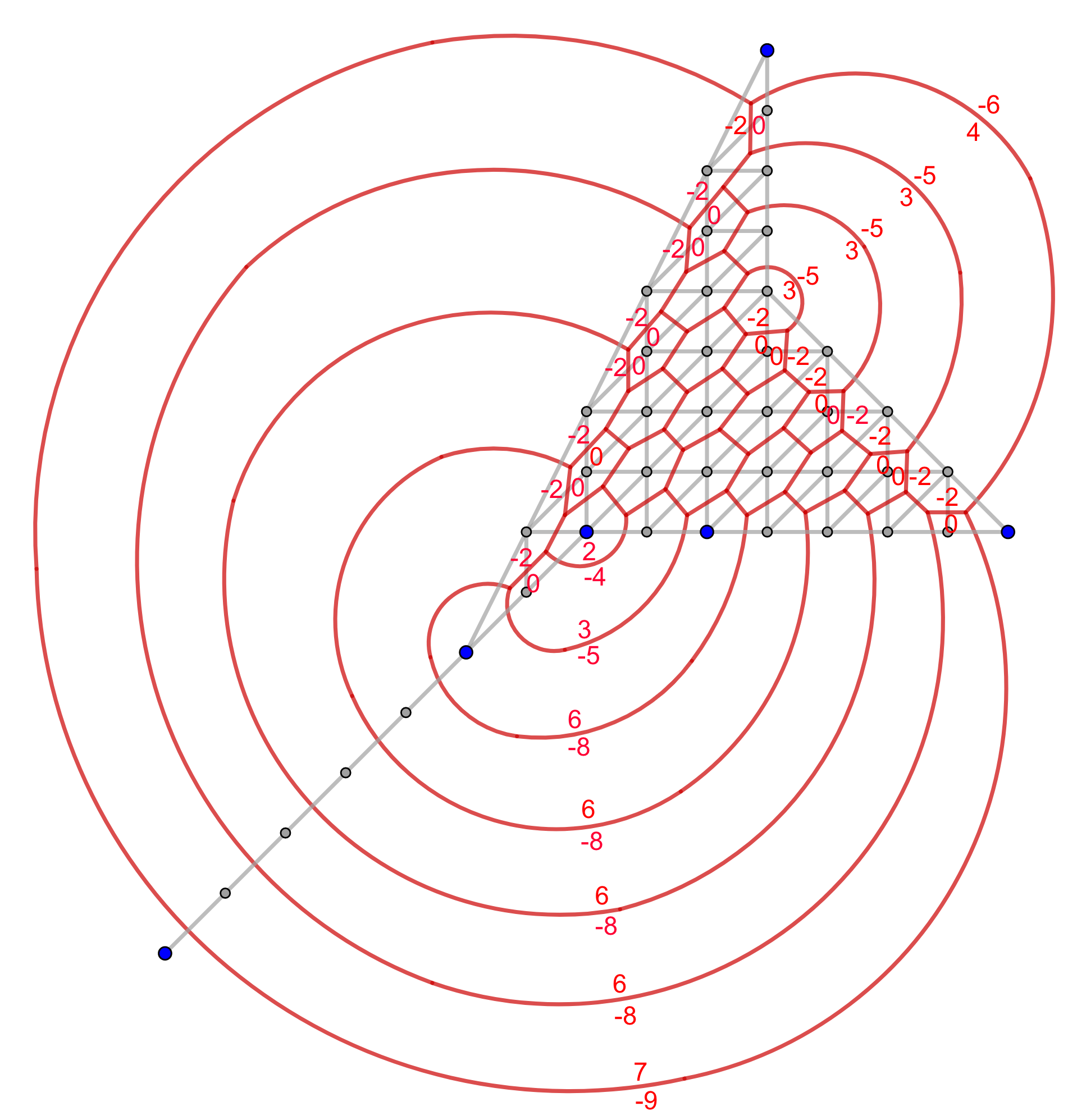}
\caption{A Type III anticanonical pair $Y_0$ and the lattice manifold structure on its dual complex. Double curves with self-intersection $-1$ on both components are unlabeled.}
\label{example}
\end{centering}
\end{figure}
\end{example}

\begin{definition}\label{generic} We say that $Y_0$ is \textsl{generic} if $Q(V_i,C_i)\leq 1$ for all $i\neq 0$. \end{definition} Because of the formula $Q(D)+Q(D')=24$, genericity is  equivalent to the condition that $\Gamma(Y_0)$ has $Q(D)+1$ singularities. 

\begin{remark}\label{fibration} Integral-affine manifolds appear naturally in symplectic geometry as bases of Lagrangian torus fibrations. Let $\mu\colon(X,\omega)\rightarrow B$ be a Lagrangian $2$-torus fibration. There are natural coordinates from $B$ to $\R^2$ well defined up to $GL_2(\Z)\ltimes \R^2$ defined as follows: Choose a base point $p_0\in B$ and let $U\ni p_0$ be a contractible open set. Let $\{\alpha,\beta\}$ be a basis of $H_1(\mu^{-1}(b_0);\Z)$. Let $p\in U$ be a point, and let $\gamma\colon [0,1]\to U$ satisfy $\gamma(0)=p_0$, $\gamma(1)=p$. Let $C_\alpha$ and $C_\beta$ be cylinders in $X$ which fiber over $\gamma$ and whose boundaries in $\mu^{-1}(b_0)$ are homologous to $\alpha$ and $\beta$, respectively. We may then define coordinates $U\to \R^2$ by $$(x(p),y(p))=\left(\int_{C_{\alpha}} \omega,\int_{C_{\beta}}\omega\right).$$ These integrals are well-defined because the fibers of $\mu$ are Lagrangian. The only ambiguities are the choice of base point, and basis of $H_1$ of a fiber; thus the coordinates are well-defined up to $GL_2(\Z)\ltimes \R^2$, and in fact lie in $SL_2(\Z)\ltimes \R^2$ if there is a consistent choice of orientation on the fibers of $\mu$. 

Conversely, given an integral-affine manifold $B$ (without singularities), there is a natural lattice $T_\Z B\subset TB$ in the tangent bundle consisting of tangent vectors which are integral in some, and thus any, chart. Dualizing gives a lattice $T_\Z^*B\subset T^*B$ in the cotangent bundle. The cotangent bundle $T^*B$ admits a natural symplectic structure, and addition of a section of $T_\Z^*B$ is a symplectomorphism. Thus, there is a Lagrangian torus fibration $$\mu\,:\,T_\Z^*B\backslash T^*B\rightarrow B$$ and this procedure locally describes an inverse to the above procedure. In the local coordinates $(x,y)$ of an integral-affine chart on $B$, the symplectic form is $\omega=dx\wedge dp + dy\wedge dq$ where $p,q\in \R/\Z$ are coordinates on the torus fibers which descend from the coordinates on fibers of the cotangent bundle.

Symington \cite{Symington} was able to extend this correspondence in dimension two to Lagrangian fibrations with certain singular fibers. Suppose there is a Lagrangian fibration of a symplectic $4$-manifold $\mu\colon (X,\omega)\rightarrow B$ whose general fiber is a $2$-torus, but degenerates over some fibers to a necklace of $k$ Lagrangian $2$-spheres. Then the base $B$ admits an integral-affine structure with an {\it $A_k$ singularity} at the image of a singular fiber, around which the $SL_2(\Z)$ component of monodromy of the integral-affine structure is conjugate to $$\twobytwo{1}{k}{0}{1}.$$ Conversely, given an integral affine surface $B$ with $A_k$ singularities, there is a Lagrangian torus fibration over $B$ which is unique up to (not necessarily fiber-preserving) symplectomorphism if the base either has a nonempty boundary or is noncompact. 

Returning to Type III degenerations, suppose that $Y_0$ is generic. Then every singularity of $\Gamma(Y_0)$ other than $v_0$ is of type $A_1$---this follows easily from Propositions 2.9 and 2.10 of \cite{Engel}. Thus, there is a Lagrangian torus fibration $$\mu\,:\,(X,\omega) \rightarrow \Gamma(Y_0) - \{v_0\}$$ with irreducible nodal fibers exactly over the vertices $v_i$ corresponding to pairs with $Q(V_i,C_i)=1$. The edges in $\Gamma(Y_0)^{[1]}$ also have an interpretation in the symplectic $4$-manifold $X$: For any straight line segment $y=mx$ of rational slope in the base, there is a Lagrangian cylinder which fibers over it, whose fiber is the circle $p=-mq$. Thus, there is a Lagrangian cylinder $A_{ij}$ which fibers over each edge $e_{ij}$.

Proposition \ref{symp} below proves that $X$ is in fact diffeomorphic to $Y_t-D_t$. Furthermore, under this diffeomorphism, the class of the symplectic form and the class of the fiber of $\mu$ are naturally identified with the two invariants describing the monodromy on $H_2(Y_t-D_t;\Z)$. We prove that every monodromy-invariant class in $H_2(Y_t-D_t;\Z)$ may be constructed by patching together the $A_{ij}$ by chains lying over $\Gamma(Y_0)^{[0]}$. Thus, the monodromy-invariant classes are represented by Lagrangians and perpendicular to $[\omega]$.

\end{remark}

\begin{proposition}\label{extend} Let $Y_0$ be a generic Type III anticanonical pair. The Lagrangian fibration $\mu_0\colon (X_0,\omega)\rightarrow \Gamma(Y_0)-\{v_0\}$ extends to a continuous map $\mu\,:\,(X,D,\omega)\rightarrow \Gamma(Y_0)$ where $X$ is a smooth, compact $4$-manifold, $D=\mu^{-1}(v_0)$ is a cycle of $2$-spheres with the appropriate self-intersections, and $\omega$ is the symplectic form on $X_0=X-D$. \end{proposition}

\begin{proof} Let $U\ni v_0$ be a contractible open set in $\Gamma(Y_0)$ containing no singularities other than $v_0$. Then the Lagrangian torus fibration over $U-\{v_0\}$ is, as a smooth $2$-torus fibration, uniquely determined by the $SL_2(\Z)$ monodromy around $v_0$. We note that the Hirzebruch-Inoue surface is topologically quite simple: It is a two-torus fibration over an annulus that undergoes symplectic boundary reduction over its two circular boundaries to produce the cycles $D$ and $D'$. It then follows easily from computing the monodromy around the vertex of the pseudo-fan $\mathfrak{F}(V_0,C_0)$ that $\mu_0^{-1}(U-\{v_0\})$ is diffeomorphic to $V_0-D-D'$ and that the boundary component of the annulus over which $D$ fibers corresponds to $v_0$. Thus, we may extend $\mu_0$ in the desired manner by gluing $X_0$ to a neighborhood of $D$ in $V_0$ and declaring $\mu(D)=v_0$.  \end{proof}

\subsection{Topology and geometry of the central fiber}

We now collect various results on the cohomology of $Y_0$ (both topological and analytic) and the limiting cohomology. For notational simplicity, we shall just consider the case where $Y_0$  has global normal crossings; minor modifications handle the general case. Let $p_{ijk}=V_i\cap V_j\cap V_k$ denote a triple point, and let \begin{align*} a_0& \colon \coprod _iV_i \to Y_0 \\ a_1& \colon \coprod _{i< j}C_{ij} \to Y_0 \\ a_2& \colon\!\!\! \coprod_{i< j< k} \!\! p_{ijk}\to Y_0 \end{align*} be the natural morphisms.

\begin{lemma}\label{Lemma1} $H^0(Y_0; \scrO_{Y_0}) \cong \C$, $H^1(Y_0; \scrO_{Y_0}) = H^2(Y_0; \scrO_{Y_0}) = 0$.
\end{lemma}
\begin{proof} The sheaf $\scrO_{Y_0}$ is quasi-isomorphic to the complex
$$\textstyle (a_0)_*\left(\bigoplus_i\scrO_{V_i}\right) \to (a_1)_*\left(\bigoplus_{i< j}\scrO_{C_{ij}}\right) \to (a_2)_*\left(\bigoplus_{i< j< k}\scrO_{p_{ijk}}\right).$$
Thus there is a spectral sequence converging to $H^*(Y_0; \scrO_{Y_0})$ whose $E_1$ page is
\begin{center}
\begin{tabular}{|c|c|c}
$H^1(V_0; \scrO_{V_0})$ &{} &{} \\ \hline
$\bigoplus_i\C$ & $\bigoplus_{i< j}\C$ & $\bigoplus_{i< j< k}\C$ \\ \hline
\end{tabular}
\end{center}
where $H^1(V_0; \scrO_{V_0}) \cong \C$. 
The $E_2$ page is then 
\begin{center}
\begin{tabular}{|c|c|c}
$\C$ &{} &{} \\ \hline
$\C$ & {} & $\C$ \\ \hline
\end{tabular}
\end{center}
and the only possibly nonzero differential is $d_2^{0,1}\colon \C\to \C$. Then we have $H^1(Y_0; \scrO_{Y_0}) \cong \Ker d_2^{0,1}$ and $H^2(Y_0; \scrO_{Y_0})\cong \Coker d_2^{0,1}$. By Serre duality, $H^2(Y_0; \scrO_{Y_0})$ is dual to $H^0(Y_0; \omega_{Y_0}) = H^0(Y_0; \scrO_{Y_0}(-D))=0$. Thus $H^2(Y_0; \scrO_{Y_0})=0$, so that $d_2^{0,1}$ is an isomorphism and hence $H^1(Y_0; \scrO_{Y_0}) =0$ as well.
\end{proof}

\begin{corollary}\label{Corollary2} $\Pic Y_0 \cong H^2(Y_0; \Z) \cong H^2(\mathcal{Y}; \Z)\cong \Pic \mathcal{Y}$.
\end{corollary}
\begin{proof} The first isomorphism is immediate from the exponential sheaf sequence and Lemma~\ref{Lemma1}, and the second is clear because $\mathcal{Y}$ deformation-retracts onto $Y_0$. To prove the final statement, it suffices to note that $H^1(\mathcal{Y}; \scrO_{\mathcal{Y}}) = H^2(\mathcal{Y}; \scrO_{\mathcal{Y}}) = 0$ via the Leray spectral sequence.
\end{proof}

\begin{lemma}\label{Lemma3} $H^0(Y_0; \C) \cong \C$,   $H^1(Y_0; \C) =  0$, and $H^4(Y_0; \C) \cong \C^f$. If $d \colon \bigoplus_iH^2(V_i;\C) \to \bigoplus_{i< j}H^2(C_{ij};\C)$ is the natural map, then 
\begin{align*} H^2(Y_0; \C) &\cong \Ker d\\
H^3(Y_0; \C) &\cong \Coker d \oplus \C . 
\end{align*}  
\end{lemma}
\begin{proof} The Mayer-Vietoris spectral sequence for $H^*(Y_0; \C)$ has $E_1$ page
\begin{center}
\begin{tabular}{|c|c|c}
$\C^f$ & {} & {} \\ \hline
$\C$  & {} & {} \\ \hline
$\bigoplus_iH^2(V_i;\C)$ & $\bigoplus_{i< j}H^2(C_{ij};\C)$ & {} \\ \hline
$\C$  & {} & {} \\ \hline
$\C^f$ & $\C^e$ & $\C^v$\\ \hline
\end{tabular}
\end{center}
The $E_2$ page is 
\begin{center}
\begin{tabular}{|c|c|c}
$\C^f$ & {} & {} \\ \hline
$\C$  & {} & {} \\ \hline
$\Ker d$ & $\Coker d$ & {} \\ \hline
$\C$  & {} & {} \\ \hline
$\C$ & {} & $\C$\\ \hline
\end{tabular}
\end{center}
It suffices to prove that $d_2^{0,1}$ is an isomorphism. However, $H^1(V_0; \C) \cong H^1(V_0; \scrO_{V_0})$, $H^0(p_{ijk}; \C) \cong H^0(p_{ijk}; \scrO_{p_{ijk}})$, $H^0(C_{ij}; \C) \cong H^0(C_{ij}; \scrO_{C_{ij}})$, so the fact that $d_2^{0,1}$ is an isomorphism for this spectral sequence follows from the corresponding fact for the spectral sequence for $\scrO_{Y_0}$ appearing in Lemma~\ref{Lemma1}. 

\end{proof}

\begin{remark} For the Mayer-Vietoris spectral sequence, $E_3 = E_\infty$. We will see later that $\Coker d = 0$.
\end{remark}

\begin{lemma}\label{notorsinH1} $\Pic Y_0 = H^2(Y_0; \Z)$ is torsion free and 
$$\textstyle \Pic Y_0 \cong \Ker \left(\bigoplus_iH^2(V_i;\Z) \to \bigoplus_{i< j}H^2(C_{ij};\Z)\right).$$
\end{lemma}
\begin{proof} By Corollary~\ref{Corollary2}, $\Pic Y_0 \cong H^2(Y_0; \Z)$. The arguments of Lemma~\ref{Lemma3}, but with $\Z$-coefficients, then give an exact sequence
\begin{align*} \textstyle 0 \to \Coker d_2^{0,1}  \to &H^2(Y_0; \Z) \to  \\ &\textstyle \Ker \left(\bigoplus_iH^2(V_i;\Z) \to \bigoplus_{i< j}H^2(C_{ij};\Z)\right) \to 0,\end{align*}
where now $d_2^{0,1}$ refers to the differential in the corresponding spectral sequence with $\Z$-coefficients. By Lemma~\ref{Lemma3}, $\Coker d_2^{0,1}$ is torsion. 
It suffices to prove that $H_1(Y_0; \Z)$ is torsion free, and in fact we will show directly that $H_1(Y_0;\Z)=0$. Examining the Mayer-Vietoris homology spectral sequence computing $H_n(Y_0; \Z)$, we must show that the differential $d^2_{2,0} \colon E_{2,0}^2 = H_2(|\Gamma(Y_0)|; \Z) \cong \Z\to E_{0,1}^2  = \bigoplus_iH_1(V_i; \Z) = H_1(V_0; \Z)$ is an isomorphism, where $|\Gamma(Y_0)|$ is the topological realization of the dual complex of $Y_0$. It is straightforward to check that, in this special case, if $1$ is a generator of $H_2(|\Gamma(Y_0)|; \Z)$, then $d^2_{2,0}(1)$ is represented by the $1$-cycle in $V_0$  formed by connecting the double points  in the cycle $\bigcup_{i=1}^{r'}C_{0i}=D'$ by (real) curves in the components.   By \cite[(4.10)]{Inoue},   inclusion induces isomorphisms $\pi_1(D', *)\cong \pi_1(V_0, *) \cong \pi_1(D, *)$ and thus $H_1(D', \Z)\cong H_1(V_0, \Z)$. Hence $d^2_{2,0}(1)$ is   a generator of $H_1(V_0;\Z)$, so that $d^2_{2,0}$ is surjective and $H_1(Y_0;\Z)=0$.
\end{proof}

Let $\Omega_{Y_0}^*/\tau_{Y_0}^*$ be the quotient of the complex $\Omega_{Y_0}^*$ by the subcomplex of torsion differentials (i.e.\ those sections supported on   $(Y_0)_{\text{sing}}$). Here $\Omega_{Y_0}^* = \bigwedge^*\Omega_{Y_0}^1$, where $\Omega_{Y_0}^1$ is the sheaf of K\"ahler differentials on $Y_0$. By \cite[(1.5)]{Friedman1}, we have the following:

\begin{lemma} {\rm(i)} There is an exact sequence 
$$\textstyle 0 \to \Omega_{Y_0}^1/\tau_{Y_0}^1 \to (a_0)_*\left(\bigoplus_i\Omega^1_{V_i}\right) \to (a_1)_*\left(\bigoplus_{i< j} \Omega^1_{C_{ij}}\right) \to 0.$$

\smallskip
\noindent {\rm(ii)} $\Omega_{Y_0}^2/\tau_{Y_0}^2 \cong (a_0)_*\left(\bigoplus_i\Omega^2_{V_i}\right)$. \qed
\end{lemma}

\begin{corollary}\label{Corollary7} {\rm(i)} With notation as in Lemma~\ref{Lemma3}, $H^0(Y_0; \Omega_{Y_0}^1/\tau_{Y_0}^1) = 0$, $H^1(Y_0; \Omega_{Y_0}^1/\tau_{Y_0}^1) \cong \Ker d$, and $H^2(Y_0; \Omega_{Y_0}^1/\tau_{Y_0}^1) \cong \Coker d$.

\smallskip
\noindent {\rm(ii)} $H^0(Y_0; \Omega_{Y_0}^2/\tau_{Y_0}^2) = 0$, $H^1(Y_0; \Omega_{Y_0}^2/\tau_{Y_0}^2)\cong H^1(V_0; \Omega^2_{V_0}) \cong \C$, and 
$$\textstyle H^2(Y_0; \Omega_{Y_0}^2/\tau_{Y_0}^2) \cong \bigoplus_i H^2(V_i; \Omega^2_{V_i}) \cong \bigoplus_iH^4(V_i; \C) = \C^f.$$
\end{corollary}

By \cite[(1.5)]{Friedman1}, there is a spectral sequence with $E_1$ term $H^q(Y_0;\Omega_{Y_0}^p/\tau_{Y_0}^p)$  converging to $H^{p+q}(Y_0; \C)$. 

\begin{lemma}  The above spectral sequence degenerates at $E_1$.
\end{lemma}
\begin{proof} The $E_1$ page looks like
\begin{center}
\begin{tabular}{|c|c|c}
$\C^f$ & {} & {} \\ \hline
$\C$  & {} & {} \\ \hline
$\Ker d$ & $\Coker d$ & {} \\ \hline
{}  & {} & {} \\ \hline
$\C$ & {} & {}\\ \hline
\end{tabular}
\end{center}
Comparing this with Lemma~\ref{Lemma3}, we see that no differentials in the spectral sequence can ever be nonzero.
\end{proof} 

A similar discussion handles the case of the complex $\Omega_{Y_0}^*(\log D)/\tau_{Y_0}^*$. We record for future  reference the only cases we shall need:

\begin{lemma}\label{Lemma1.9} $H^0(Y_0; \Omega_{Y_0}^1(\log D)/\tau_{Y_0}^1) = 0$ and $H^1(Y_0; \Omega_{Y_0}^1(\log D)/\tau_{Y_0}^1)\cong H^2(Y_0;\C)/\bigoplus_i\C[D_i]$. Moreover $H^1(Y_0; \Omega_{Y_0}^1(\log D)/\tau_{Y_0}^1)\cong H^2(Y_0; \Omega_{Y_0}^1/\tau_{Y_0}^1)$.
\end{lemma}
\begin{proof} From the exact sequence 
$$\textstyle 0 \to \Omega_{Y_0}^1/\tau_{Y_0}^1 \to \Omega_{Y_0}^1(\log D)/\tau_{Y_0}^1 \to \bigoplus_i\scrO_{D_i} \to 0,$$
we get an exact sequence 
\begin{gather*}\textstyle 
H^1(Y_0; \Omega_{Y_0}^1/\tau_{Y_0}^1)=0 \to  H^1(Y_0;\Omega_{Y_0}^1(\log D)/\tau_{Y_0}^1 )\to \bigoplus_i H^0(\scrO_{D_i}) \\
\to H^1(Y_0; \Omega_{Y_0}^1/\tau_{Y_0}^1) =\Ker d\to H^1(Y_0; \Omega_{Y_0}^1(\log D)/\tau_{Y_0}^1) \to 0.
\end{gather*}
As the map $\bigoplus_i H^0(\scrO_{D_i}) \to  \Ker d\subseteq \bigoplus_jH^2(V_j;\C)$ is injective, 
$$H^0(Y_0; \Omega_{Y_0}^1(\log D)/\tau_{Y_0}^1) = 0$$
and $H^1(Y_0; \Omega_{Y_0}^1(\log D)/\tau_{Y_0}^1)\cong H^2(Y_0;\C)/\bigoplus_i\C[D_i]$. The last statement follows since $H^1(D_i; \scrO_{D_i}) = H^2(D_i; \scrO_{D_i}) = 0$. 
\end{proof}

We turn next to limiting cohomology. We have the relative log complex $\Lambda^*_{\mathcal{Y}/\Delta}= \Omega^*_{\mathcal{Y}/\Delta}(\log Y_0)$, and we denote the restriction of the relative log complex to $Y_0$ by $\Lambda^*_{Y_0}$. Note that, for $t\neq 0$,  the restriction to the  fiber $Y_t$ is just $\Omega^*_{Y_t}$. The complex $\Lambda^*_{\mathcal{Y}/\Delta}(\log \mathcal{D})= \Omega^*_{\mathcal{Y}/\Delta}(\log (Y_0+\mathcal{D}))$  is defined in the obvious way. Its restriction to $Y_t$ for $t\neq 0$ is  $\Omega^*_{Y_t}(\log D)$ and its restriction  to $Y_0$ will be denoted by $\Lambda^*_{Y_0}(\log D)$.
By \cite{Steenbrink} (see \cite[Theorem 11.16]{PetersSteenbrink}), the hypercohomology $\mathbb{H}^*(Y_0; \Lambda^*_{Y_0})$ may be identified with the cohomology of the canonical fiber  $Y_\infty = \mathcal{Y}\times _{\Delta^*}\widetilde{\Delta^*}$, where $\Delta^*$ is the punctured unit disk and $\widetilde{\Delta^*}$ is its universal cover. It follows that the sheaves $\mathbb{R}^q\pi_*\Lambda^*_{\mathcal{Y}/\Delta}$ are locally free. Minor modifications of the above arguments show that $\mathbb{H}^*(Y_0; \Lambda^*_{Y_0}(\log D))$ may be identified with the cohomology of    $Y_\infty -D_\infty= (\mathcal{Y}-\mathcal{D})\times _{\Delta^*}\widetilde{\Delta^*}$ and hence that the sheaves $\mathbb{R}^q\pi_*\Lambda^*_{\mathcal{Y}/\Delta}(\log \mathcal{D})$ are locally free.

We compute the cohomology of the sheaves $\Lambda^*_{Y_0}$ and $\Lambda^*_{Y_0}(\log D)$ as follows.

\begin{lemma}\label{somecalcs} $H^q(Y_0; \Lambda^p_{Y_0})$ and $H^q(Y_0; \Lambda^p_{Y_0}(\log D))$ are zero except in the following cases:
\begin{enumerate}
\item[\rm(i)] $p=0$, hence $\Lambda^0_{Y_0}= \Lambda^0_{Y_0}(\log D) = \scrO_{Y_0}$, and $q = 0$.
\item[\rm(ii)] $p=q =1$.
\item[\rm(iii)] $p=2$. In this case $\Lambda^2_{Y_0} =\omega_{Y_0} =\scrO_{Y_0}(-D)$ and $\Lambda^2_{Y_0}(\log D) = \scrO_{Y_0}$. Hence $H^2(Y_0; \Lambda^2_{Y_0}) \cong H^0(Y_0;\Lambda^2_{Y_0}(\log D))\cong \C$ and $H^q(Y_0; \Lambda^2_{Y_0})$ and $H^q(Y_0; \Lambda^2_{Y_0}(\log D))$ are zero in all other cases.
\end{enumerate}
\end{lemma}
\begin{proof} The cases $p=0$, $p=2$, and hence (i) and (iii), follow easily from Lemma~\ref{Lemma1}.

To deal with the case $p=1$, let us first show that $H^0(Y_0;\Lambda^1_{Y_0}) =0$ (which is in fact already proved in \cite[Lemma (2.7)]{FriedmanMiranda}). Let $K$ be the kernel of the natural morphism
$$\textstyle (a_1)_*\left(\bigoplus_{i< j}\scrO_{C_{ij}}\right) \to (a_2)_*\left(\bigoplus_{i< j< k}\scrO_{p_{ijk}}\right).$$ 
By \cite[p.\ 268]{PetersSteenbrink} or \cite[Proposition (3.5)]{Friedman1}, there is an exact sequence
$$0 \to \Omega^1_{Y_0}/\tau^1_{Y_0} \to \Lambda^1_{Y_0} \to K \to 0.$$
Moreover $H^0(Y_0; \Omega^1_{Y_0}/\tau^1_{Y_0}) =0$ by Corollary~\ref{Corollary7}. The argument on p.\ 108 of \cite{Friedman1} then shows that $H^0(Y_0;\Lambda^1_{Y_0}) =0$. By Serre duality, $H^2(Y_0;\Lambda^1_{Y_0}) =0$ as well. This handles the case of $\Lambda^1_{Y_0}$.

Finally, we deal with $\Lambda^1_{Y_0}(\log D)$. First note that we have the usual Poincar\'e residue sequence
$$\textstyle 0 \to \Lambda^1_{Y_0} \to \Lambda^1_{Y_0}(\log D) \to  \bigoplus_i\scrO_{D_i}  \to 0.$$
It follows that $H^2(Y_0; \Lambda^1_{Y_0}(\log D)) \cong  H^2(Y_0; \Lambda^1_{Y_0}) =0$. Moreover, since  $H^0(Y_0;\Lambda^1_{Y_0}) =0$, it suffices to prove that the coboundary map 
$$\textstyle \delta \colon \bigoplus_iH^0(\scrO_{D_i})  \to H^1(Y_0;\Lambda^1_{Y_0})$$
 is injective. 
 
As before, let $K=\Ker (a_1)_*\left(\bigoplus_{i< j}\scrO_{C_{ij}}\right) \to (a_2)_*\left(\bigoplus_{i< j< k}\scrO_{p_{ijk}}\right)$. It follows easily that $H^0(Y_0; K) \cong \C^{f-1}$ and that $H^1(Y_0; K) \cong \C$. Then there is a commutative diagram whose top row is exact
 $$\begin{CD}
 H^0(Y;K) @>>> H^1(Y_0; \Omega^1_{Y_0}/\tau^1_{Y_0}) @>>> H^1(Y_0; \Lambda^1_{Y_0})\\
 @| @VVV @.\\
 \C^{f-1} @>>> \bigoplus_iH^1(V_i; \Omega^1_{V_i}) @. {}\\
 @. @VVV @.\\
 {} @. \bigoplus_{i< j}H^1(C_{ij}; \Omega^1_{C_{ij}}) @. {}
 \end{CD}$$
 The classes $[D_i]$ are linearly independent in $H^0(V_0; \Omega^1_{V_0})$ and map to $0$ in $\bigoplus_{i< j}H^1(C_{ij}; \Omega^1_{C_{ij}})$, since the $D_i$ are disjoint from the double curves $C_{ij}$. Moreover, the image of $H^0(Y_0; K) \cong \C^{f-1}$ in $H^1(Y_0; \Omega^1_{Y_0}/\tau^1_{Y_0})$ consists of the $\C$-span of the Chern classes of the $f$ line bundles $\scrO_{\mathcal{Y}}(V_i)\big{|}_{Y_0}$, subject to the relation 
 $$(\scrO_{\mathcal{Y}}(V_0)\big{|}_{Y_0}) \otimes \cdots \otimes (\scrO_{\mathcal{Y}}(V_{f-1})\big{|}_{Y_0}) = \scrO_{\mathcal{Y}}(Y_0)\big{|}_{Y_0} = \scrO_{Y_0}.$$
 In particular, the image of $\scrO_{\mathcal{Y}}(V_i)\big{|}_{Y_0}$ in $H^1(V_j; \Omega^1_{V_j})$ is $[C_{ij}]$ if $i\neq j$ and 
$-[C_i]$ if $i=j$. Then by the same argument as before, the classes $[D_i]$ are linearly independent in $H^0(V_0; \Omega^1_{V_0})$ modulo the image of $H^0(Y;K)$. It follows that their images in $H^1(Y_0; \Lambda^1_{Y_0})$ are linearly independent. Thus, $\delta$ is injective.
\end{proof}

We then have: 

\begin{corollary}\label{subbundles} {\rm (i)} The spectral sequence with $E_1$ term $H^q(Y_0;\Lambda^p_{Y_0})$ converging to $\mathbb{H}^{p+q}(Y_0;\Lambda^*_{Y_0})$ degenerates at $E_1$.

\smallskip
\noindent {\rm (ii)} The spectral sequence with $E_1$ term $H^q(Y_0;\Lambda^p_{Y_0}(\log D))$ and converging to $\mathbb{H}^{p+q}(Y_0;\Lambda^*_{Y_0}(\log D))$ degenerates at $E_1$.

\smallskip
\noindent {\rm (iii)} The sheaves $R^q\pi_*\Lambda^p_{\mathcal{Y}/\Delta}$ and $R^p\pi_*\Lambda^p_{\mathcal{Y}/\Delta}(\log \mathcal{D})$ are locally free, and are the successive quotients of a filtration of $\mathbb{R}^q\pi_*\Lambda^*_{\mathcal{Y}/\Delta}(\log \mathcal{D})$ by holomorphic subbundles.
\end{corollary}
\begin{proof} To see (i) and (ii), a glance at the $E_1$ pages of the spectral sequences shows that there are no nonzero differentials for any $E_r$, $r\geq 1$. The argument that (i) and (ii) $\implies$ (iii) is standard, given that the sheaves $\mathbb{R}^q\pi_*\Lambda^*_{\mathcal{Y}/\Delta}(\log \mathcal{D})$ are locally free (see for example the proof of \cite[Corollary 11.24]{PetersSteenbrink}).
\end{proof}

\begin{corollary}\label{Cokerd} Let $d$ be as in Lemma~\ref{Lemma3}. The following are equivalent:
\begin{enumerate}
\item[\rm(i)] $\Coker d = 0$. 
\item[\rm(ii)] $H^2(Y_0; \Omega^1_{Y_0}/\tau^1_{Y_0}) = 0$.
\item[\rm(ii)$'$] $H^2(Y_0; \Omega^1_{Y_0}/\tau^1_{Y_0}(\log D)) = 0$.
\item[\rm(iii)] The image of the specialization map $\operatorname{sp}_\C\colon H^2(Y_0; \C) \to H^2(Y_t; \C)$ has codimension one.
\item[\rm(iii)$'$] The image of the specialization map $\operatorname{sp}_\C\colon H^2(Y_0-D_0; \C) \to H^2(Y_t-D_t; \C)$ has codimension one. 
\item[\rm(iv)] The specialization map $\operatorname{sp}_\C\colon H^2(Y_0; \C) \to H^2(Y_t; \C)$ is not surjective.
\item[\rm(iv)$'$] The specialization map $\operatorname{sp}_\C\colon H^2(Y_0-D_0; \C) \to H^2(Y_t-D_t; \C)$ is not surjective.
\end{enumerate}
\end{corollary}
\begin{proof} (i) $\iff$ (ii) follows from Corollary~\ref{Corollary7}. The equivalence (i) $\iff$ (ii)$'$ follows from the last statement of Lemma~\ref{Lemma1.9}. Next, we have an exact sequence defined in the proof of Lemma~\ref{somecalcs}:
$$
H^1(Y_0; \Omega^1_{Y_0}/\tau^1_{Y_0})  \to H^1(Y_0;\Lambda^1_{Y_0}) \to H^1(Y_0; K)  
 \to  H^2(Y_0; \Omega^1_{Y_0}/\tau^1_{Y_0})\to 0,
$$
using the fact that $H^1(Y_0;\Lambda^1_{Y_0}) =0$.
Moreover  $\dim H^1(Y_0; K) = 1$, again by the proof of  Lemma~\ref{somecalcs}. Thus $H^2(Y_0; \Omega^1_{Y_0}/\tau^1_{Y_0}) = 0$ $\iff$ the homomorphism $H^1(Y_0;\Lambda^1_{Y_0}) \to H^1(Y_0; K)$ is surjective $\iff$ the image of the homomorphism $H^1(Y_0; \Omega^1_{Y_0}/\tau^1_{Y_0}) \to H^1(Y_0;\Lambda^1_{Y_0})$ has codimension one $\iff$ the homomorphism $H^1(Y_0; \Omega^1_{Y_0}/\tau^1_{Y_0}) \to H^1(Y_0;\Lambda^1_{Y_0})$ is not surjective. Finally,  $H^1(Y_0; \Omega^1_{Y_0}/\tau^1_{Y_0}) \cong H^2(Y_0; \C)$ and $H^1(Y_0;\Lambda^1_{Y_0}) \cong H^2(Y_t; \C)$, and we can identify the homomorphism $H^1(Y_0; \Omega^1_{Y_0}/\tau^1_{Y_0}) \to H^1(Y_0;\Lambda^1_{Y_0})$ with the specialization map (compare \cite[\S11.3]{PetersSteenbrink}). This proves the equivalence of (ii) with (iii) and (iv), and the equivalence of (ii)$'$ with (iii)$'$ and (iv)$'$ is similar. 
\end{proof} 

\subsection{Smoothings of Type III pairs}

Let $(Y_0, D)$ be a $d$-semistable Type III anticanonical pair. For simplicity, we shall just consider the case $r(D) > 1$ in what follows, i.e.\ all components $D_i$ of $D$ are smooth. The case where $D$ is irreducible is handled by minor modifications of the arguments. Let $D_i^2 = -d_i$. The tangent space 
for the  deformation functor of $Y_0$, keeping the divisor $D$ with normal crossings, is given by 
$$\mathbb{T}^1_{Y_0}(-\log D) = \Ext^1(\Omega^1_{Y_0}(\log D), \scrO_{Y_0}),$$
and there is the usual local to global exact sequence
\begin{gather*}
0 \to H^1(Y_0; T_{Y_0}(-\log D)) \to \mathbb{T}^1_{Y_0}(-\log D) \to H^0(Y_0; T^1_{Y_0}) \\
\to H^2(Y_0; T_{Y_0}(-\log D)) \to \mathbb{T}^2_{Y_0}(-\log D) \to H^1(Y_0; T^1_{Y_0}).
\end{gather*}

The term $H^1(Y_0; T_{Y_0}(-\log D))$ is the tangent space to the functor of locally trivial deformations. By $d$-semistability,  $T^1_{Y_0} \cong \scrO_{(Y_0)_{\text{sing}}}$.  The functor of locally trivial deformations is unobstructed,  the map $\mathbb{T}^1_{Y_0}(-\log D) \to H^0(Y_0; T^1_{Y_0})$ is surjective, and the map $\mathbb{T}^2_{Y_0}(-\log D) \to H^1(Y_0; T^1_{Y_0})$ is an isomorphism  because of the following:

\begin{lemma} $H^2(Y_0; T_{Y_0}(-\log D)) =0$.
\end{lemma}
\begin{proof} By \cite[(2.10)]{Friedman1}, $H^2(Y_0; T_{Y_0}(-\log D))$ is Serre dual to 
$$H^0(Y_0; (\Omega^1_{Y_0}(\log D)/\tau_{Y_0}^1)\otimes \omega_{Y_0}) = H^0(Y_0; (\Omega^1_{Y_0}(\log D)/\tau_{Y_0}^1)(-D)).$$
But   $H^0(Y_0; (\Omega^1_{Y_0}(\log D)/\tau_{Y_0}^1) (-D))\subseteq H^0(Y_0; \Omega^1_{Y_0}(\log D)/\tau_{Y_0}^1)$, and 
$$H^0(Y_0; \Omega^1_{Y_0}(\log D)/\tau_{Y_0}^1) = 0$$ by Lemma~\ref
{Lemma1.9}.
\end{proof}

Note that the long exact $\Ext$ sequence associated to the exact sequence
$$0 \to \Omega^1_{Y_0} \to \Omega^1_{Y_0}(\log D) \to \bigoplus_i\scrO_{D_i} \to 0$$
gives the usual sequence 
\begin{gather*}\textstyle 
0= \bigoplus_iH^0(D_i; N_{D_i/Y_0} ) \to \mathbb{T}^1_{Y_0}(-\log D) \to \mathbb{T}^1_{Y_0} \to \\ \textstyle \to \bigoplus_iH^1(D_i; N_{D_i/Y_0} ) \to \mathbb{T}^2_{Y_0}(-\log D) \to \mathbb{T}^2_{Y_0} \to  0.
\end{gather*}
Moreover $\mathbb{T}^2_{Y_0}(-\log D) \to \mathbb{T}^2_{Y_0}$ is an isomorphism, as both are identified with $H^1(Y_0; T^1_{Y_0})$.  

Fixing a nowhere zero section $\xi$ of $H^0(Y_0; T^1_{Y_0})$, Lie bracket defines a surjective homomorphism $T_{Y_0}(-\log D) \to T^1_{Y_0}$ whose kernel we denote by $S_{Y_0}(-\log D)$. By the proof of \cite[(4.4)]{Friedman1}, the dual $(S_{Y_0}(-\log D))\spcheck$ is isomorphic to $\Lambda^1_{Y_0}(\log D)$, which can be defined independently of the existence of a smoothing of $Y_0$ or of the pair $(Y_0,D)$.

\begin{lemma}\label{evenmorecalcs} {\rm{(i)}} In the above notation, $H^2(Y_0; S_{Y_0}(-\log D))= 0$.
 
\smallskip 
\noindent {\rm{(ii)}} There is an exact sequence
\begin{gather*}
H^0(Y_0; T^1_{Y_0})\cong \C \to H^1(Y_0; S_{Y_0}(-\log D)) \to H^1(Y_0; T_{Y_0}(-\log D))\\
 \to H^1(Y_0; T^1_{Y_0})\cong  H^1((Y_0)_{\text{\rm{sing}}};\scrO_{(Y_0)_{\text{\rm{sing}}}})\to 0.
 \end{gather*}
 
 \smallskip 
\noindent {\rm{(iii)}} The image of $H^1(Y_0; S_{Y_0}(-\log D))$ in $ H^1(Y_0; T_{Y_0}(-\log D))$  is the tangent space to the locally trivial deformations of $(Y_0, D)$ which are $d$-semistable.

\smallskip 
\noindent {\rm{(iv)}} The image of $H^1(Y_0; S_{Y_0}(-\log D))$ in $ H^1(Y_0; T_{Y_0}(-\log D))$ has codimension $f-1$. 

\smallskip 
\noindent {\rm{(v)}} There is an exact sequence 
$$\textstyle 0 \to H^1(Y_0; S_{Y_0}(-\log D)) \to H^1(Y_0; S_{Y_0})\to \bigoplus_iH^1(D_i; N_{D_i/Y_0}) \to 0,$$
and hence the image of $H^1(Y_0; S_{Y_0}(-\log D))$ in $H^1(Y_0; S_{Y_0})$ has codimension $\sum_i(d_i-1)$.
\end{lemma}
\begin{proof} (i) By Serre duality, we have
$$H^2(Y_0; S_{Y_0}(-\log D))\spcheck \cong H^0(Y_0; (S_{Y_0}(-\log D))\spcheck \otimes \scrO_{Y_0}(-D)).$$ 
Furthermore, 
\begin{align*}
H^0(Y_0; (S_{Y_0}(-\log D))\spcheck \otimes \scrO_{Y_0}(-D)) &\cong H^0(Y_0; \Lambda^1_{Y_0}(\log D)\otimes \scrO_{Y_0}(-D))\\
&\subseteq H^0(Y_0; \Lambda^1_{Y_0}(\log D)).
\end{align*}
By Lemma~\ref{somecalcs}, $H^0(Y_0; \Lambda^1_{Y_0}(\log D))=0$, hence $H^2(Y_0; S_{Y_0}(-\log D))= 0$ as well. The exact sequence in (ii) is then the long exact  cohomology sequence associated to
$$0\to S_{Y_0}(-\log D) \to T_{Y_0}(-\log D) \to T^1_{Y_0} \to 0.$$

\smallskip 
\noindent (iii) This follows from \cite[(4.5)]{Friedman1}.

\smallskip 
\noindent (iv) The codimension in question is 
$$h^1(Y_0, T^1_{Y_0}) = h^1(\scrO_{(Y_0)_{\text{sing}}}) =  \dim \Pic (Y_0)_{\text{sing}}.$$  The proof that $h^1(\scrO_{(Y_0)_{\text{sing}}}) = f-1$  follows easily from the identities $2e=3v$ and $v-e+f=2$ \cite[p.\ 25]{FriedmanScattone}.

\smallskip 
\noindent (v) The exact sequence in the first statement of (v) follows by taking the cohomology long exact sequence associated to
$$\textstyle 0 \to S_{Y_0}(-\log D) \to S_{Y_0} \to \bigoplus_iN_{D_i/Y_0} \to 0$$
and using (i). The final statement follows since $N_{D_i/Y_0}$ is a line bundle on $D_i$ of degree $-d_i < 0$.
\end{proof}

We now analyze the smoothings of the pair $(Y_0, D)$, following the method of \cite{Friedman1}  (although  we could also use the arguments of \cite{KawamataNamikawa}). 

\begin{theorem} Let $(Y_0, D)$ be a $d$-semistable Type III anticanonical pair. Then the pair $(Y_0, D)$ is smoothable. More precisely, there is a unique smoothing component $(X,0)$ of the locally semi-universal deformation of the pair $(Y_0, D)$, and moreover:
\begin{enumerate}
\item[\rm(i)] $X$ is smooth and the discriminant locus in $X$ is a smooth hypersurface.
\item[\rm(ii)] $\dim X = \dim \mathbb{T}^1_{Y_0}(-\log D)-f+1$. 
\item[\rm(iii)] Given a  germ of a  map of the disk $(\Delta,0)$ to $(X,0)$, the pulled back family $(\mathcal{Y}, \mathcal{D}) \to \Delta$ is a Type III degeneration of anticanonical pairs, i.e.\ the total space is smooth $\iff$ the map $\Delta \to X$ is transverse to the discriminant locus.
\end{enumerate}
\end{theorem}
\begin{proof} By the arguments of \cite[(5.10)]{Friedman1}, there is a germ of a smoothing component $(\widehat{X},0)$ of the surface $Y_0$. Here $0$ corresponds to the surface $Y_0$,   $\widehat{X}$ is smooth at $0$ and its  tangent space at $0$ is given by the exact sequence
$$0\to \C  \to H^1(Y_0; S_{Y_0} ) \to T_{\widehat{X},0} \to H^0(Y_0; T^1_{Y_0})\cong \C \to 0.$$
Let $\widehat{\mathcal{X}} \to \widehat{X}$ be the corresponding family.
For each component $D_i$ of $D$, consider the component $\mathcal{Z}_i$  of the relative Hilbert-Douady scheme of curves in $\widehat{\mathcal{X}}$ containing $D_i$. We have the following exact sequence for the normal bundle of $D_i$ in $\widehat{\mathcal{X}}$:
$$0 \to N_{D_i/Y_0} \to N_{D_i/\widehat{\mathcal{X}}} \to \scrO_{D_i}\otimes_\C T_{\widehat{X},0} \to 0.$$
 Hence, taking  Euler characteristics, we have
 $$\chi(Y_0; N_{D_i/\widehat{\mathcal{X}}}) = \chi(Y_0; N_{D_i/Y_0}) + \dim \widehat{X} = \dim \widehat{X} -d_i + 1.$$
 The dimension of $\mathcal{Z}_i$ at the point corresponding to component $D_i$ is at least $\chi(Y_0; N_{D_i/\widehat{\mathcal{X}}})$.  Since $D_i^2< 0$, the corresponding Hilbert scheme over a point $x$ of $\widehat{X}$ is either empty or smooth of dimension zero. It follows that (as germs) $\mathcal{Z}_i$ can be identified with its image $Z_i \subseteq \widehat{X}$, and that $Z_i$ is a smooth submanifold of $\widehat{X}$ of codimension at most $d_i -1$. Setting $X = \bigcap _iZ_i$, $X$ is a subspace of $\widehat{X}$ and every component of $X$ has codimension at most $\sum_i(d_i-1)$. On the other hand, clearly the Zariski tangent space of $X$ is contained in the subspace $T'$ of $T_{\widehat{X},0}$ where the normal crossings divisor $D =\sum_iD_i$ deforms to first order. There is an exact sequence
 $$0\to \C  \to H^1(Y_0; S_{Y_0}(-\log D )) \to T' \to H^0(Y_0; T^1_{Y_0})\cong \C .$$
Comparing the above sequence with that in the first paragraph of the proof for $T_{\widehat{X},0}$ and using  Lemma~\ref{evenmorecalcs}, we see that the codimension of $T'$ in $T_{\widehat{X},0}$ is either $\sum_i(d_i-1)$ or $\sum_i(d_i-1)+1$, and the first case arises $\iff$ the map $T' \to H^1(Y_0; T^1_{Y_0})$ is surjective. By the discussion on $X$ above, we see that $X$ is smooth, $\dim X = \dim \widehat{X} - \sum_i(d_i-1)$, $T_{X,0}=T'$, and the homomorphism $T_{X,0} \to H^0(Y_0; T^1_{Y_0})$ is surjective.   The proofs of the remaining statements of the theorem are then straightforward. 
\end{proof}

\section{Monodromy}

\subsection{A formula for the monodromy}

We keep the notation of the previous section: $\pi \colon  (\mathcal{Y}, \mathcal{D}) \to \Delta$ is a Type III degeneration of anticanonical pairs. Let $(Y,D) = (Y_t, D_t)$ be a general fiber of $\pi$. Then the monodromy diffeomorphism  of the family acts on $H_2(Y; \Z)$, $H_2(Y-D;\Z)$, and on $\overline{H}_2(Y-D, \partial)$. Our goal is to analyze this action, primarily on $H_2(Y-D;\Z)$.

We begin with a discussion of the topology of $\mathcal{Y}$. Let $c\colon Y=Y_t \to Y_0$ be the Clemens collapsing map. For each triple point $p\in Y_0$, $c^{-1}(p)=\tau_p$ is a $2$-torus in $Y$. The argument of \cite[Lemma (1.9)]{FriedmanScattone} shows that, for every $p$ and $q$, the $2$-tori $\tau_p$ and $\tau_q$ are homologous in $Y-D$ (up to sign). Denote this common homology class (well-defined up to sign) by $\tau$. Next, let $\alpha=\alpha_{ij}^k$ be a (real) simple closed curve in a double curve $C_{ij}$ simply enclosing the triple point $p=p_{ijk}$; we can assume that $\alpha$ is contained in a small disk around $p$. An examination of the local form for $c$ (see for example \cite[\S2.2]{Persson}) shows that $c^{-1}(\alpha) =\tau_p$ is homologous to $\tau$.

Explicitly, there exist local analytic coordinates centered in a ball of radius $s$ at $p$ such that, for $|t|\ll s$, $\pi$ is given by $xyz=t$. We may model the collapsing map $c$ as follows: Given $(x,y,z)$ such that $xyz=t$, let $(a,b,c)\in \R_{\geq 0}^3$ be the unique point such that $abc=0$ and $(a,b,c)=(|x|,|y|,|z|)-(k,k,k)$ for some $k\in \R$. Then the Clemens collapse is $$c(x,y,z)=\left(\frac{ax}{|x|},\frac{by}{|y|},\frac{cz}{|z|}\right).$$ Thus, $$\tau_p=c^{-1}(0,0,0)=\{(x,y,z)\colon |x|=|y|=|z|=|t|^{1/3}\}.$$ Restrict $(x,y,z)$ to lie in the polydisk $|x|,|y|,|z|\leq s$. There is a topologically trivial two-torus fibration $(x,y,z)\mapsto (|x|,|y|,|z|)$ whose image is a two-simplex $\sigma_p$ and whose fibers are homologous to $\tau_p$. The curve $\alpha$ can be taken to be $(0, 0, re^{i\theta})$ and thus there is some $r'>0$ such that $$c^{-1}(\alpha)=\{(x,y,z)\colon|x|=|y|=r', |z|=|t|/(r')^2\}.$$ Hence $c^{-1}(\alpha)$ and $\tau_p$ are homologous.

Next suppose that $\alpha_{0j} \subseteq C_{0j}$ is a loop around a triple point contained in $V_0$. Then $c^{-1}(\alpha_{0j}) = c^{-1}(\tilde\alpha_{0j})$, where $\tilde \alpha_{0j} \subseteq V_0$ is the tube over $\alpha_{0j}$ in $V_0 -C_0 -D$. By the explicit description of the topology of $V_0$ (see \cite[\S4]{Inoue}), $\tilde \alpha_{0j}$ is equal to $\gamma$, where $\gamma\subseteq V_0 -C_0 -D$ is the tube over a simple closed curve in $D_i$ enclosing a double point as in Section 1. Identifying the $2$-torus $\gamma$ in $V_0$ with $c^{-1}(\gamma)$ in $Y$, we see that, in the above notation, $\tau$ and $c^{-1}(\alpha_{0j})$ are homologous to $\gamma$. In particular, they are homologous to $0$ in $H_2(Y;\Z)$, but are nontrivial in $H_2(Y-D;\Z)$. 

Similarly, for $i$ arbitrary (possibly $0$), if $\alpha_{ij}$ is a loop in a double curve $C_{ij}$ simply enclosing the triple point $p_{ijk}$, then we can define the class $\tau  \in H_2(V_i - C_i; \Z)$. 

\begin{theorem} With notation as in Section 1,
\begin{enumerate}
\item[\rm(i)] Let $T$ be the action of monodromy on $H_2(Y; \Z)$, $H_2(Y-D;\Z)$, or $\overline{H}_2(Y-D, \partial)$. Then $T$ is unipotent and, if $N = T-\Id$, then $N^2 =0$.
\item[\rm(ii)] For the action of $T$ on $H_2(Y; \Z)$, $T=\Id$ and $N=0$.
\item[\rm(iii)] There is a unique class $\lambda \in \Lambda\spcheck$ such that, for all $x\in \overline{H}_2(Y-D;\Z)$, 
$$N(x) = -(x\bullet \lambda)\gamma.$$
Moreover, $\lambda^2 = v$, where $v$ is the number of triple points.
\end{enumerate}
\end{theorem}
\begin{proof} Since $\pi\colon \mathcal{Y}\to \Delta$ is a degeneration with reduced normal crossings, general results (for example \cite{Clemens}) show that $T$ is unipotent on $H_2(Y; \Z)$. Since $T([D_i]) = [D_i]$, $T$ induces a unipotent automorphism of $\Lambda$ as well. For $H_2(Y-D;\Z)$, $T$ preserves the filtration 
$$0 \to \Z\gamma \to  H_2(Y-D;\Z) \to \Lambda \to 0,$$
and is unipotent on the graded pieces (in fact it is $\Id$ on $\Z\gamma$) and hence is unipotent on $H_2(Y-D;\Z)$. A similar argument handles $\overline{H}_2(Y-D, \partial)$. We will deal with the index of unipotency shortly.

We can then apply \cite[Theorem 5.6]{Clemens} to conclude that $N=0$ on $H^2(Y;\Z)$. In that notation,
$s=0$ since the $\tau_p$ are homologous to $0$ in $H^2(Y;\Z)$, and the classes denoted by $\psi(\gamma')$ there are all of the form $c^{-1}(\alpha)$ for $\alpha$  a simple closed curve in a double curve $C_{ij}$ simply enclosing a triple point $p$, hence as we have seen they are all homologous to $0$ in our case. Thus $N=0$.

It follows that the induced automorphism $T$ of $\Lambda$ is $\Id$ as well. Thus, for the action of $T$ on $H_2(Y-D;\Z)$, if $N= T-\Id$, then $\im N \subseteq \Z\gamma$ and thus $N$ is necessarily of the form $N(x) = -(x\bullet \lambda)\gamma$ for a unique class $\lambda \in \Lambda\spcheck$ (with the understanding that $N(x)$ is defined for all $x\in H_2(Y-D; \Z)$ via the homomorphism $H_2(Y-D; \Z) \to \Lambda$), and hence $N^2 =0$. 

We finally prove that $\lambda^2 = v$. The idea of the proof is similar to the argument of \cite[Proposition (1.10)]{FriedmanScattone}, but is complicated by the fact that we cannot simply consider $N^2$ because $N^2=0$. Instead, we will use  the dual action of monodromy  $\hat{T}$   on $\overline{H}_2(Y-D, \partial)$. Let $\hat{N} = \hat{T}-\Id$.  We have the exact sequence $$0 \to \Lambda\spcheck \to \overline{H}_2(Y-D, \partial) \to \Z \to 0.$$
Let $\hat{\gamma}$ denote an element of $\overline{H}_2(Y-D, \partial)$ which maps onto a generator of the quotient $\Z$ above, i.e.\ $(\gamma \bullet \hat{\gamma}) = 1$.
By duality $\hat{N}(\Lambda\spcheck) = 0$ and $\hat{N}(\hat{\gamma}) \in \Lambda\spcheck$, so that $\hat{N}^2 =0$. 

\begin{lemma} $\hat{N}(\hat{\gamma}) = \lambda \bmod \Z\gamma$.
\end{lemma}
\begin{proof} Since monodromy is a diffeomorphism, for all $x\in H_2(Y-D)$ and for all $y\in \overline{H}_2(Y-D, \partial)$, we have 
$$(Tx\bullet \hat{T}y) = (x\bullet y).$$
Using $T = \Id + N$ and $\hat{T} = \Id + \hat{N}$, we have
$$(x\bullet y) = (Tx\bullet \hat{T}y) = ((\Id + N)x\bullet (\Id + \hat{N})y),$$
and thus 
$$(Nx\bullet y) + (x \bullet \hat{N}y) =  -(Nx\bullet \hat{N}y) = 0,$$
since $Nx\in \Z\gamma$ is orthogonal to the image of $\hat{N}$. In particular, taking  $y =\hat{\gamma}$, we see that, for all $x\in   H_2(Y-D)$,
$$(Nx\bullet \hat{\gamma}) = - (x\bullet \lambda) = -(x \bullet \hat{N}\hat{\gamma}),$$
and hence $\hat{N}(\hat{\gamma}) = \lambda \bmod \Z\gamma$.
\end{proof}

\begin{corollary}\label{Nsquared1} If $\overline{N}$ denotes the map $\Lambda_\Q \to \Q\gamma$ induced by $N$, then
$$\overline{N}\circ \hat{N}(\hat{\gamma}) = -(\lambda^2)\gamma.$$
\end{corollary} 
\begin{proof} This is clear since, by definition, 
 $\overline{N}\circ \hat{N}(\hat{\gamma}) = -(\lambda \bullet \lambda) \gamma$. 
\end{proof}

To relate this to the number $v$ of triple points, the essentially local arguments used in the proof of \cite[Theorem 4.4]{Clemens}  show:

\begin{lemma}\label{Nsquared2} For all $y\in \overline{H}_2(Y-D, \partial)$,
$$\overline{N}\circ \hat{N}(y) = -v(\gamma \bullet y)\gamma.\qed$$
\end{lemma}

Combining Corollary~\ref{Nsquared1} with Lemma~\ref{Nsquared2}, we see that 
 $$\overline{N}\circ \hat{N}(\hat{\gamma}) = -(\lambda^2)\gamma = -v(\gamma \bullet \hat{\gamma})\gamma= -v\gamma,$$
 and hence $\lambda^2 = v$ as claimed.
\end{proof}


\subsection{Further properties of $\lambda$}

To say more about the class $\lambda$, we describe a geometric representative $\Sigma$ for $\lambda$. For each double curve $C_{ij}$, choose a simple closed curve $\ell_{ij}$ connecting the two triple points $p_{ijk}$, $p_{ij\ell}$ lying in $C_{ij}$. Let $C_{ij}^0$ be the complement in $C_{ij}$ of two small open disks enclosing the two triple points on $C_{ij}$ and let $\ell_{ij}^0=\ell_{ij}\cap C_{ij}^0$ be the result of truncating $\ell_{ij}$ near each of the triple points. Then $T_{ij}=c^{-1}(\ell_{ij}^0)=\ell_{ij}^0 \times f_{ij}$ is a cylinder, where $f_{ij}$ is the fiber of the Clemens collapsing map over $C_{ij}^0$. Near the triple point $p_{ijk}$ we have the local model $xyz=t$ and a torus fibration $(x,y,z)\rightarrow (|x|,|y|,|z|)$. As above, we choose an appropriate neighborhood of $p_{ijk}$ so that the image of this torus fibration is a $2$-simplex $\sigma_p$. The sum $[f_{ij}]+[f_{jk}]+[f_{ki}]=0$ is zero in the homology of $\sigma_p\times \tau_p$---for instance, \begin{align*} f_{ij}=c^{-1}(\epsilon,0,0) & = \{(x,y,z)\colon x=\epsilon',\,|y|=|z|\} \\  f_{ki}=c^{-1}(0,\epsilon,0) & = \{(x,y,z)\colon y=\epsilon',\,|x|=|z|\} \\  f_{jk}=c^{-1}(0,0,\epsilon) & = \{(x,y,z)\colon z=\epsilon',\,|x|=|y|\} \end{align*} are three representatives which sum to zero in $H_1(\sigma_p\times \tau_p;\Z)$. Thus, there is a pair of pants $S_{ijk}$ whose boundary is $\pm (f_{ij}+f_{jk}+ f_{ki})$, depending on the choice of orientation. We may patch together the cylinders $T_{ij}$ and the pairs of pants $S_{ijk}$ to form a closed surface $\Sigma$ in $Y_t - D_t$. 

\begin{lemma}\label{orientable} The surface $\Sigma$ is orientable.
\end{lemma}
\begin{proof} Fix a triple point $p=p_{ijk}$ and fix once and for all an orientation on $\tau_p$. This orients the classes $\tau \in H_2(V_i-C_i)$, and similarly for $H_2(V_j-C_j)$ and $H_2(V_k-C_k)$. Using the fact that the dual complex $\Gamma(Y_0)$ is simply connected, it is easy to see that this determines a consistent  orientation on $\tau_q$ for every triple point $q$. If $\alpha_{ij}^k$ is the boundary of the small disk on $C_{ij}$ enclosing $p_{ijk}$ counterclockwise (in the complex orientation on $C_{ij}$), then orient the fiber $f_{ij}^k$ over $\alpha_{ij}^k\cap \ell_{ij}$ so that $\alpha_{ij}^k\times f_{ij}^k$ gives the natural orientation on $\tau_p$. Choose the orientation on the surface $S_{ijk}$ such that  $\partial S_{ijk} = -(f_{ij}^k + f_{jk}^i + f_{ki}^j)$. For the other triple point $p_{ij\ell}$ on $C_{ij}$, note that $\alpha_{ij}^\ell=-\alpha_{ij}^k$ as homology classes in $H_1(C_{ij}^0)$, and hence $f_{ij}^\ell = -f_{ij}^k$. If $T_{ij}$ is oriented so that $\partial T_{ij} = f_{ij}^k  +f_{ij}^\ell$, then the orientations on the $T_{ij}$ and $S_{ijk}$ give a consistent orientation on $\Sigma$.
\end{proof} 

The homology class $[\Sigma]$ as described above is well-defined up to sign. But,
if we choose a different pair of pants $S_{ijk}$ whose boundary is $f_{ij}\cup f_{jk}\cup f_{ki}$ then $[\Sigma]$ changes by adding a multiple of $[\tau_p] =\gamma$. Hence $[\Sigma]$ is only naturally defined mod $\Z\gamma$.   

As oriented $2$-manifolds $T_{ij} = \ell_{ij}^0\times f_{ij}^k$, where $\ell_{ij}^0$ is oriented so that the endpoint is $\alpha_{ij}^k\cap \ell_{ij}$ and the starting point is $\alpha_{ij}^\ell\cap \ell_{ij}$. With these choices of orientation, $\ell_{ij}^0 \bullet \alpha_{ij} =-1$ for the pairing $H_1(C_{ij}^0 , \partial C_{ij}^0) \otimes H_1(C_{ij}^0) \to \Z$. This is independent of the superscript $k$, in the sense that if we replace $p_{ijk}$ by the other triple point $p_{ij\ell}$ contained in $C_{ij}$, then both $\ell_{ij}^0$ and $\alpha_{ij}$ change orientation, so that $\ell_{ij}^0 \bullet \alpha_{ij}$ is unchanged.

\begin{lemma}\label{intersection} $[\Sigma]\cdot [\Sigma]=\lambda\cdot \lambda=v$ and $[\Sigma]\cap D_t=0$. In particular, $[\Sigma]\neq 0$ in $H_2(Y_t-D_t;\Z)$. \end{lemma}

\begin{proof} First note that $\Sigma$ is supported in a neighborhood of the singular locus of $Y_0$, and thus $\Sigma\cap D_t=0$ since $D_0$ does not intersect the singular locus of $Y_0$. To compute $[\Sigma]\cdot [\Sigma]$, we sum local contributions to the intersection number. Each cylinder $T_{ij}$ may be perturbed to a cylinder $T_{ij}'$ by perturbing the path $\ell_{ij}$ to a path $\ell_{ij}'$ which only intersects $\ell_{ij}$ at the two triple points. Then $T_{ij}\cap T_{ij}'=0$. 

In the model of the Clemens collapse above, $c^{-1}(\alpha_{ij}^k)$, $c^{-1}(\alpha_{jk}^i)$, and $c^{-1}(\alpha_{ki}^j)$ are three fibers of $\sigma_p\times \tau_p\rightarrow \sigma_p$ which map to the three center points of the edges of $\sigma_p$. Then $f_{ij}'$, $f_{jk}'$, and $f_{ki}'$ are three circles in each of these three fibers, homologous to $f_{ij}$, $f_{jk}$, and $f_{ki}$. Let $q$ and $q'$ be two points in the interior of $\sigma_p$.

Let $A_{ij}$, $A_{jk}$, and $A_{ki}$ be three cylinders which fiber over segments connecting $q$ and the midpoints of the edges of $\sigma_p$, one boundary component of each cylinder being $f_{ij}$, $f_{jk}$, and $f_{ki}$. Then the union of the boundary components of $A_{ij}$, $A_{jk}$, and $A_{ki}$ which are contained in the fiber over $q$ is null-homologous in this fiber. Thus, there is a chain $B_{ijk}$ supported in the fiber over $q$ such that $\partial(B_{ijk}\cup A_{ij}\cup A_{jk}\cup A_{ki})=-(f_{ij}+f_{jk}+f_{ki})$. In other words, we can replace the pair of pants $S_{ijk}$ by  $B_{ijk}\cup A_{ij}\cup A_{jk}\cup A_{ki}$.

Define $A_{ij}'$, $A_{jk}'$, $A_{ki}'$, and $B_{ijk}'$ analogously to replace a pair of pants $S_{ijk}'$ whose boundary is $-(f_{ij}'+f_{jk}'+f_{ki}')$. By choosing paths from $q$ and $q'$ to the midpoints appropriately, we can ensure that there is only one torus fiber in which the representatives of $S_{ijk}$ and $S_{ijk}'$ can intersect. Furthermore, we may assume that the two cylinders passing though this fiber are $A_{ij}$ and $A_{jk}'$. Then, our two representatives of $S_{ijk}$ and $S_{ijk}'$ have exactly one positively oriented intersection. Hence, each pair of pants $S_{ijk}$ contributes $1$ to the self-intersection of $[\Sigma]$, and each cylinder $T_{ij}$ contributes $0$. Thus $[\Sigma]\cdot [\Sigma]=v$.\end{proof}

\begin{proposition}\label{geomrep} $[\Sigma] \bmod \Z \gamma =\pm \lambda$. \end{proposition}

\begin{proof} Note that $[\Sigma]\in H_2(\mathcal{Y}-\mathcal{D};\Z)$ can be constructed for arbitrarily small $t$. Thus, $[\Sigma]$ is supported in an arbitrarily small tubular neighborhood of $\bigcup_{ij}\ell_{ij}$. But $\bigcup_{ij}\ell_{ij}$ has real dimension one. Hence the class $[\Sigma]$ in $H_2(\mathcal{Y}-\mathcal{D};\Z) = H_2(Y_0-D_0; \Z)$ is zero. That is, $[\Sigma]$ is in the kernel of the specialization map $H_2(Y_t-D_t; \Z) \to H_2(Y_0-D_0; \Z)$.

Any element of $\im(H^2(Y_0-D_0;\Q)\to H^2(Y_t-D_t;\Q))$ pairs with $[\Sigma]$ to zero, since the specialization map preserves the pairing between homology and cohomology. Furthermore $[\Sigma]\neq 0$ in $H_2(Y_t-D_t,\Q)$ by Lemma \ref{intersection}. By the equivalence of (iii)$'$ and  (iv)$'$ of Corollary \ref{Cokerd}, $\im(H^2(Y_0-D_0;\Q)\to H^2(Y_t-D_t;\Q))$ has codimension one in $H^2(Y_t-D_t;\Q)$, and hence 
 $$\im(H^2(Y_0-D_0;\Q)\to H^2(Y_t-D_t;\Q)) =\{\alpha\colon\alpha([\Sigma])=0\} .$$

By definition, the $T$-invariant classes in $x\in \overline{H}_2(Y_t-D_t;\Z)$ are those satisfying $x\bullet \lambda=0$. Since the image of the specialization map on cohomology is contained in the set of $T$-invariant classes, we conclude that $\lambda=c[\Sigma]\mod \Z\gamma$ for some $c\in \Q^*$. Since  $\lambda\cdot \lambda = [\Sigma]\cdot [\Sigma]\neq 0$, we must have $c=\pm 1$. \end{proof}

 A priori $\lambda\in \Lambda^\vee$, but we may now conclude: 

\begin{corollary}  The class $\lambda\in\Lambda$.
\end{corollary} 
\begin{proof} Visibly, $\lambda =\pm[\Sigma]$ is the image of a  class in $H_2(Y_t-D_t;\Z)$.
\end{proof}

\begin{proposition}\label{lambdaprop} Let $\operatorname{sp}\colon H^2(Y_0; \Z) \to H^2(Y_t; \Z)$ be the specialization map in cohomology. 
\begin{enumerate}
\item[\rm(i)]  $\Coker \operatorname{sp}\cong \Z$.   
 \item[\rm(ii)] In the notation of Lemma~\ref{Lemma3}, $\Coker d =0$.
\item[\rm(iii)] An element $\alpha \in H^2(Y_t;\Z)$ lifts to an element of $H^2(Y_0;\Z) \cong \Pic Y_0\cong \Pic \mathcal{Y}$ $\iff$ $\alpha(\lambda) = 0$ $\iff$ $\alpha \cdot \lambda = 0$, where in the last equality we identify $H_2(Y_t;\Z)$  with $H^2(Y_t;\Z)$ via Poincar\'e duality. 
\end{enumerate}
\end{proposition} 
\begin{proof} (i) First we claim that $\Coker \operatorname{sp}$ is torsion free. Consider the Gysin spectral sequence for   $\mathcal{Y}-Y_0 = \mathcal{Y}^*$, i.e.\ the Leray spectral sequence for the inclusion $i\colon \mathcal{Y}^*\to\mathcal{Y}$, with $E_2^{p,q} = H^p (\mathcal{Y}; R^qi_*\Z)$,  converging to $H^*(\mathcal{Y}^*; \Z)$. Its $E_2$ page is (all cohomology is with $\Z$-coefficients):
\begin{center}
\begin{tabular}{|c|c|c|c|c}
$\!\!\bigoplus_{i< j< k}H^0(p_{ijk})\!\!$ &{} &{} &{} &{}\\ \hline
$\bigoplus_{i< j}H^0(C_{ij})$ & $\!\!\bigoplus_{i< j}H^1(C_{ij})\!\!$ & $\!\!\bigoplus_{i< j}H^2(C_{ij})\!\!$ &{} &{}\\ \hline
$\bigoplus_iH^0(V_i)$ & $\bigoplus_iH^1(V_i)$ & $\bigoplus_iH^2(V_i)$ & $\!\!\bigoplus_iH^3(V_i)\!\!$ & $\!\!\bigoplus_iH^4(V_i)\!\!$\\ \hline
$H^0(\mathcal{Y})$ & $H^1(\mathcal{Y})$ & $H^2(\mathcal{Y})$ & $H^3(\mathcal{Y})$ &$H^4(\mathcal{Y})$ \\ \hline
\end{tabular}
\end{center}  This yields an exact sequence
$$\textstyle \bigoplus _iH^0(V_i; \Z)\cong \Z^f \to H^2(\mathcal{Y}; \Z) \to H^2(\mathcal{Y}^*; \Z) $$ and $H^2(\mathcal{Y}^*; \Z)/\im(H^2(\mathcal{Y}; \Z))$ has a filtration with one quotient equal to 
$\Ker (\bigoplus_iH^1(V_i;\Z) \to H^3(\mathcal{Y}; \Z))$ and the second contained in 
$$\textstyle \Ker \big{(}\bigoplus_{i,j}H^0(C_{ij}; \Z) \to \bigoplus _i H^2(V_i;\Z)\big{)}.$$
Both $\bigoplus_iH^1(V_i;\Z) = H^1(V_0; \Z)\cong \Z$ and $\bigoplus_{ij}H^0(C_{ij}; \Z)$  are torsion free. Thus $H^2(\mathcal{Y}^*; \Z)/\im(H^2(\mathcal{Y}; \Z))$ is torsion free.
By the Wang sequence, since $N=T-I = 0$ on $H^2(Y_t; \Z)$,   restriction to a fiber induces an isomorphism $H^2(\mathcal{Y}^*; \Z) \cong H^2(Y_t; \Z)$.  As $H^2(\mathcal{Y}; \Z) \cong H^2(Y_0;\Z)$,  the map $H^2(\mathcal{Y}; \Z) \to H^2(\mathcal{Y}^*; \Z) $ is identified with the specialization map $H^2(Y_0; \Z) \to H^2(Y_t; \Z)$. Thus $\Coker \operatorname{sp}\cong H^2(\mathcal{Y}^*; \Z)/\im(H^2(\mathcal{Y}; \Z))$  is torsion free.
 
The specialization map in cohomology is dual to the map $H_2(Y_t; \Z) \to H_2(Y_0; \Z)$. Since $\lambda \in \Ker (\operatorname{sp}\colon H_2(Y_t; \Q) \to H_2(Y_0; \Q))$ and $\lambda \neq 0$ since $\lambda^2 > 0$, we see that the kernel of $H_2(Y_t; \Z) \to H_2(Y_0; \Z)$ has rank at least one. Dually,  the rank of $\Coker \operatorname{sp}$ is at least one. Then   $\Coker \operatorname{sp}$ has rank exactly one by Corollary~\ref{Cokerd}, and hence $\Coker \operatorname{sp} \cong \Z$.

\smallskip
\noindent (ii) This follows immediately from Corollary~\ref{Cokerd}.

\smallskip
\noindent (iii) By (i), the cokernel of $H^2(Y_0; \Z) \to H^2(Y_t; \Z)$ is $\Z$, and in particular the image of $H^2(Y_0; \Z)$ in $H^2(Y_t; \Z)$ is a primitive sublattice of corank one containing  
$$\{\alpha \in H^2(Y_t;\Z):\alpha(\lambda) = 0\}.$$
Since $\{\alpha \in H^2(Y_t;\Z):\alpha(\lambda) = 0\}$ is also primitive of corank one, it is equal to 
the image of $H^2(Y_0; \Z)$ in $H^2(Y_t; \Z)$, which is a restatement of (iii).
\end{proof}

\begin{remark} The above shows that the local invariant cycle theorem holds for $H^2(Y_t-D_t)$, in fact over $\Z$: A class $x\in H^2(Y_t-D_t;\Z)$ is in the image of $H^2(Y_0-D_0; \Z)$ $\iff$ $N(x) = 0$. Note that this result fails for $H^2(Y_t)$ (even with $\Q$-coefficients).
\end{remark}

The class $\lambda$ has the following additional property, in the  notation of Definition~\ref{defB}:

\begin{proposition}\label{extendsections} Viewing $\lambda$ as an element of $\Lambda_\R$, for a unique choice of sign, $\pm \lambda\in  \mathcal{B}_{\text{\rm{gen}}}$.
\end{proposition}
\begin{proof} Using Lemma~\ref{altcharA}, it suffices to show that (i) $\lambda \cdot \alpha \neq 0$ for all effective numerical exceptional curves $\alpha$ and that (ii) if $\alpha_1$ and $\alpha_2$ are two different effective numerical exceptional curves, then $\lambda \cdot [\alpha_1]$ and $\lambda \cdot [\alpha_2]$ have the same sign. To see (i), suppose that $\alpha$ is an effective numerical exceptional curve and that $\lambda \cdot \alpha =0$. Note that the condition that $\alpha$ is effective is independent of $t\in \Delta^*$. By Proposition~\ref{lambdaprop}, there exists a line bundle $\mathcal{L}$ on $\mathcal{Y}$ such that $c_1(\mathcal{L}\big{|}_{Y_t}) =\alpha$ for $t\neq 0$. Thus $h^0(Y_t; \mathcal{L}\big{|}_{Y_t}) > 0$ for all $t\neq 0$, and hence by semicontinuity $h^0(Y_0; \mathcal{L}\big{|}_{Y_0}) > 0$ and there is a nonzero section $s$ of $\mathcal{L}$. Choose a such a nonzero section $s$. After  removing all possible  components of $Y_0$ from the divisor $(s)$ of zeroes of $s$ on $\mathcal{Y}$, we can assume that $(s)$ is flat over $\Delta$, i.e.\ all fibers of $\pi\big{|}_{(s)}$ have dimension one.  Restricting to $Y_0$, there exists an effective Cartier divisor $C$ on $Y_0$ such that $C\cdot D = 1$. But every such divisor $C$ is a sum of the form $\sum_in_iD_i$, where $n_i \geq 0$,  plus components disjoint from $D$. Since $D\cdot D_i \leq 0$ for every $i$, this is impossible.

The proof of (ii) is similar: If  $\alpha_1$ and $\alpha_2$ are two different effective numerical exceptional curves such that $\lambda \cdot [\alpha_1]$ and $\lambda \cdot [\alpha_2]$ have opposite signs, then there exist positive integers $n$ and $m$ such that $\lambda \cdot (n\alpha_1 + m\alpha_2) = 0$. Arguing as in the proof of (i), there exists an effective Cartier divisor $C$ on $Y_0$ such that $C\cdot D = n+m>0$, and we reach a contradiction as before.
\end{proof}

\begin{remark} In the above proof, it is  crucial that $H^1(Y_0; \scrO_{Y_0}) = 0$ so that $\Pic Y_0 \cong \Pic \mathcal{Y} \cong H^2(Y_0;\Z)$.
\end{remark}

From now on, by convention, we fix the orientation on $D$ so that $\lambda\in \mathcal{B}_{\text{\rm{gen}}}$. 

\begin{definition} If $\pi\colon (\mathcal{Y}, \mathcal{D})\to \Delta$ is a Type III degeneration, we define the \textsl{monodromy invariant of $\mathcal{Y}$} to be the class $\lambda\in \mathcal{B}_{\text{\rm{gen}}}\cap \Lambda$ modulo the action of $\Gamma(Y,D)$, the group of admissible isometries. The monodromy invariant only depends on the central fiber $(Y_0,D)$, and so we shall also use the above to define the monodromy invariant of a $d$-semistable Type III anticanonical pair. 

A Type III degeneration is \textsl{of $\Q$-type $\lambda$} if its  monodromy invariant is of the form $r\lambda$, $r\in \Q^+$.
The $\Q$-type only depends on the family $(\mathcal{Y}^*, \mathcal{D}^*) \to \Delta^*$ and is preserved by base change. In particular, let $\overline{\pi}\colon (\overline{\mathcal{Y}}, \mathcal{D}) \to \Delta$  a one parameter family  whose central fiber is the contraction of the cusp $D'$ in $V_0$ and whose general fiber has only RDP singularities. By the argument of \cite[Theorem 2.5]{FriedmanMiranda}, possibly after a base change, the family $\overline{\pi}\colon (\overline{\mathcal{Y}}, \mathcal{D}) \to \Delta$ is birational to a Type III degeneration $\pi\colon (\mathcal{Y}, \mathcal{D})\to \Delta$ via  a birational isomorphism which is a minimal resolution on the fibers away from $0$. Thus, it makes sense to say $\overline{\pi}\colon \overline{\mathcal{Y}} \to \Delta$ is of $\Q$-type $\lambda$.
\end{definition}

\subsection{A symplectic description of monodromy}

In this section we describe the monodromy of $\mathcal{Y}\rightarrow \Delta$ in terms of the Lagrangian torus fibration over $\Gamma(Y_0)\backslash \{v_0\}$. 

\begin{proposition}\label{symp} Suppose that $Y_0$ is generic as in Definition \ref{generic}. Let $$\mu\colon(X,D,\omega)\rightarrow \Gamma(Y_0)$$ be the map constructed in Proposition \ref{extend}. There is a diffeomorphism $\phi\colon (X,D)\to (Y_t,D_t)$ such that:

\begin{enumerate} \item[\rm(i)] $\phi(D_i)=D_i$, 
\item[\rm(ii)] $\phi_*([\mu^{-1}(p)])=\gamma$, and
\item[\rm(iii)] $\phi_*([\omega])=\pm\lambda\mod \Z\gamma$.
\end{enumerate}
\end{proposition}

\begin{proof}

We remark that (ii) is vacuous unless we restrict $\phi$ to $X-D$. First, we construct the diffeomorphism $\phi$. Let $B=\Gamma(Y_0) - \{v_0\}$ and let $B_0$ denote the non-singular locus of $B$ where the fibers of $\mu$ are smooth. Then $\mu^{-1}(B_0)$ is the quotient of the symplectic manifold $T^*B_0$ by a lattice $T_\Z^*B_0\subset T^*B_0$ of Lagrangian sections of $T^*B_0\rightarrow B_0$: $$\mu^{-1}(B_0)\cong T_\Z^*B_0 \backslash T^*B_0.$$ 

Consider the intersection complex $\Gamma(Y_0)^{\vee}$ of $Y_0$, which has a face for each surface $V_i$, an edge for each double curve $C_{ij}$, and a trivalent vertex for each triple point $p_{ijk}$. There is diffeomorphism from $\Gamma(Y_0)^\vee$ into the dual complex $\Gamma(Y_0)$ such that:

\begin{enumerate}
\item The image of a vertex of $\Gamma(Y_0)^\vee$ lies in the associated triangle of $\Gamma(Y_0)$.
\item The image of an edge of $\Gamma(Y_0)^\vee$ crosses the associated edge of $\Gamma(Y_0)$.
\item The image of a face of $\Gamma(Y_0)^\vee$ contains the associated vertex of $\Gamma(Y_0)$.
\end{enumerate}

\noindent See Figure \ref{example} for example. There is a canonical isomorphism $$H_1(X_b;\Z)\cong T_{\Z,b}^*B_0\subset T_b^*B_0.$$ Let $e_{ij}^\vee$ be an edge in the one-skeleton of $\Gamma(Y_0)^\vee$. It intersects the associated edge $e_{ij}$ in the one-skeleton of $\Gamma(Y_0)$ at one point $b\in B_0$. Up to sign, there is a unique primitive co-vector $\eta_{ij}\in T_{\Z,b}^*B_0\subset T_b^*B_0$ which vanishes on the edge $e_{ij}$. The three cycles $\eta_{ij}$, $\eta_{jk}$, and $\eta_{ki}$ mutually form bases of $H_1(X_b,\Z)$ because $f_{ijk}$ is a basis triangle. Let $C_{ij}$, $C_{ik}$, and $C_{i\ell}$ be three adjacent double curves on $V_i$. Then by construction of the integral-affine structure on $\Gamma(Y_0)$, we have $\eta_{ij}+\eta_{i\ell}=d_{ik}\eta_{ik},$ where $d_{ik}=-C_{ik}^2$. For $i\neq 0$ let $$\mu_i\colon (V_i,C_i,\omega_i)\rightarrow B_i$$ be an arbitrary almost toric fibration of $(V_i,C_i)$ and for $i=0$ let $\mu_0$ denote the composition of the two-torus fibration of $V_0$ over an annulus with the map which collapses the image of $D$ to a point. Then for all $i$, the components of $C_i$ fiber over the components of the boundary of $B_i$.

Let $F_i$ be the interior of a face of $\Gamma(Y_0)^\vee$ and consider $$\mu\big{|}_{F_i}\colon \mu^{-1}(F_i)\rightarrow F_i.$$ Suppose $i\neq 0$. We claim there is a fiber-preserving diffeomorphism between $\mu\big{|}_{F_i}$ and $\mu_i\big{|}_{\text{int}(B_i)}$. First note that the vanishing cycles $\eta_{ij}'$ associated to a component $C_{ij}\subset C_i$ in an almost toric fibration of $(V_i,C_i)$ also satisfy the equation $$\eta_{ij}'+\eta_{i\ell}'=d_{ik}\eta_{ik}'.$$ Thus, the monodromy of $\mu_i$ restricted to a neighborhood of $\partial B_i$ is the same as the monodromy of $\mu$ restricted to a neighborhood of $\partial F_i$. Thus, there is a diffeomorphism between $\mu\big{|}_{F_i}$ and $\mu_i\big{|}_{\text{int}(B_i)}$ in a neighborhood of their boundaries. We may furthermore assume that the boundary components of $F_i$ and $B_i$ are identified and that in a neighborhood of a boundary component, $\eta_{ij}$ and $\eta_{ij}'$ are identified by the diffeomorphism.

As there is either zero or one irreducible nodal fiber of both $\mu_i$ and $\mu\big{|}_{F_i}$, depending on whether $Q(V_i,C_i)=0$ or $1$, we can extend this diffeomorphism to one between all of $\mu\big{|}_{F_i}$ and $\mu_i\big{|}_{\text{int}(B_i)}$. Note that these fibrations are smoothly equivalent, but not necessarily  equivalent as Lagrangian torus fibrations. Also, note that $\mu\big{|}_{F_0}$ and $\mu_0$ are diffeomorphic as they both give the $2$-torus fibration on $V_0-C_0$, and the cycle $D$ is sent to itself.

We now construct a torus fibration in a neighborhood $U\subset Y_t$ of $(Y_0)_{\text{sing}}$ as follows---in a neighborhood of a triple point $p\in Y_0$, we use the model $\sigma_p\times \tau_p\to \tau_p$ defined above. Near the degeneration to a double curve of $Y_0$, we have a local model defined by $xy=t$. Furthermore, we may choose a third coordinate $z$ which restricts on the double locus $xy=0$ to a global coordinate on $\mathbb{P}^1$. Then $$(x,y,z)\mapsto (|x|,|y|,|z|)$$ defines a $2$-torus fibration on $xy=t$ whose fibers near $z=0$ and $z=\infty$ are homologous to the fibers of the torus fibration near a triple point. Let $\pi\,:\,U\rightarrow V$ be this two-torus fibration. Then $V$ is diffeomorphic to a tubular neighborhood of the $1$-skeleton of $\Gamma(Y_0)^{\vee}$, and we may further assume that $c^{-1}((Y_0)_{\text{sing}})$ fibers over the $1$-skeleton of $\Gamma(Y_0)^{\vee}$. The class $\eta_{ij}'$ defined as the vanishing cycle of $C_{ij}$ in the almost toric fibration $\mu_i$ is identified in $c^{-1}(V_i)\cap U$ with the class of $c^{-1}(p_{ij})$ for some generic point $p_{ij}\in C_{ij}$. 

Thus, we see that $Y_t$ is the fiber connect sum of the $\mu_i\big{|}_{\text{int}(B_i)}$, in which $\eta_{ij}'$ is identified with $\eta_{ji}'$. Since $\mu\colon (X,D,\omega)\rightarrow \Gamma(Y_0)$ is the fiber connect sum of the $\mu\big{|}_{F_i}$ in which $\eta_{ij}$ and $\eta_{ji}$ are identified, and we have fiber-preserving diffeomorphisms from $\mu_i\big{|}_{\text{int}(B_i)}$ to $\mu\big{|}_{F_i}$ which identify $\eta_{ij}$ with $\eta_{ij}'$, we conclude that there is a diffeomorphism $\phi\colon (X,D)\rightarrow (Y_t,D_t)$. The cycle $D$ is sent to $D_t$ by construction, thus property 1 is satisfied. Property 2 also follows immediately from construction.

Finally, we show $\phi_*([\omega]):=(\phi^{-1})^*([\omega])=\pm\lambda\bmod \Z\gamma$, in the sense that $$\phi_*([\omega])(x)=\lambda\bullet x$$ for any class $x\in H_2(Y_t-D_t;\Z)$. First, we must have $\phi_*([\omega])(\gamma)=0$. This is clear because the fibers of $\mu$ are Lagrangian. We henceforth elide $\phi_*$ and consider $(X,D)$ and $(Y_t,D_t)$ identified.

Proposition \ref{lambdaprop} implies that an element $x\in H_2(Y_0-D,\partial)$ may be represented by a union of surfaces $S_i\subset V_i$, closed for $i\neq 0$, meeting the $C_{ij}$ transversally and such that $[S_i]\cdot [C_{ij}]=[S_j]\cdot [C_{ji}]:=n_{ij}$, and such that $\partial S_0\subset \partial$. Choose $S_i$ such that $S_i\cap C_{ij}=S_j\cap C_{ji}$ and no $S_i$ contains a triple point of $Y_0$. Letting $S=\bigcup S_i$ we have that $c^{-1}(S)$ is a surface whose class in $H_2(Y_t-D_t,\partial)$ specializes to $x$. Let $A_{ij}\subset (X,D,\omega)$ denote a Lagrangian cylinder lying over the edge $e_{ij}\subset\Gamma(Y_0)^{[1]}$ in the one-skeleton of the dual complex. Denote its two boundary components by $U_{ij}\subset X_{v_i}$ and $U_{ji}\subset X_{v_j}$. Note that when $X_{v_i}$ is singular, we may have $U_{ij}=\emptyset$, because $U_{ij}$ could be the vanishing cycle of the node of $X_{v_i}$. We need the following lemma:

\begin{lemma} For $i\neq 0$ there is a (unique) class $[S_i]\in H_2(V_i;\Z)$ satisfying $[S_i]\cdot [C_{ij}]=n_{ij}$ if and only if $$\sum_j n_{ij}[U_{ij}]=0\in H_1(X_{v_i};\Z).$$ For $i=0$, there are no linear conditions on the $n_{0j}$. \end{lemma}

\begin{proof} When $v_i$ is nonsingular, $(V_i,C_i)$ is toric, and there is an exact sequence $$0\to H_2(V_i;\Z)\to \bigoplus_j \Z C_{ij}^*\to \Z^2\to 0$$ where the first map sends $x\mapsto (n_{ij})$ if $x\cdot C_{ij}=n_{ij}$ and the second map sends $(n_{ij})\mapsto \sum_j n_{ij}v_{ij}$ where $v_{ij}=e_{ij}$ are the primitive integral vectors spanning the one-dimensional rays of the fan of $(V_i,C_i)$ corresponding to the components $C_{ij}$. By the description of the Lagrangian cylinders $A_{ij}$ given in Remark \ref{fibration}, the condition $\sum n_{ij}v_{ij}=0$ exactly implies that the sum of the boundaries of $n_{ij}A_{ij}$ lying over $v_i$ are null-homologous. 

When $v_i$ is singular and $i\neq 0$, we have $Q(V_i,C_i)=1$. Let $$\pi\colon (V_i,C_i)\to(\overline{V_i},\overline{C}_i)$$ be an internal blow-down to a toric pair, if such a blowdown exists (the general case is easy to deduce from this case). The exceptional curve $E$ meets a component $C_{ij_0}$ such that $e_{ij_0}$ lies in the monodromy-invariant line through $v_i$. Then every element of $H_2(V_i;\Z)$ is of the form $x=nE+\pi^*\overline{x}$ for some $\overline{x}\in H_2(\overline{V_i};\Z)$. Let $v_{ij}$ be the primitive integral vectors in the fan of $\overline{V_i}$. Then there is a class in $x\in H_2(V_i;\Z)$ such that $x\cdot C_{ij}=n_{ij}$ if and only if $\sum_j n_{ij}v_{ij}$ lies in the line spanned by $v_{ij_0}$. Equivalently, if we deformed $A_{ij}$ slightly so that their boundaries lay in  a fiber near the fiber over $v_i$, we would have that $\sum n_{ij}A_{ij}$ is homologous to the circle which when transported along the monodromy-invariant line  produces a Lagrangian cylinder. But this circle is exactly the vanishing cycle of the node in the fiber over $v_i$, and so $\sum_j n_{ij}[U_{ij}]$ is null-homologous in $X_{v_i}$.

Possibly scaling all the $n_{ij}$ by a large integer, we may construct a class $S_0$ satisfying $[S_0]\cdot C_{0j}=n_{0j}$ because the components of $C_0=D'$ span a negative-definite lattice. \end{proof}

Thus, for $i\neq 0$ the boundary of $\sum n_{ij}A_{ij}$ is the boundary of a chain $Z$ supported in $\bigcup X_{v_i}$. Note that $X_{v_0}=D$, so the boundary of $A_{0i}$ lying over $v_0$ is contained in $\partial$. Then under the identification defined by $\phi$, the cycle $-Z+\sum n_{ij} A_{ij}$ represents $[c^{-1}(S)]$ because each cylinder $A_{ij}$ gets pinched by the Clemens collapse to give an intersection point with $C_{ij}$ and a class in $H_2(Y_0-D,\partial)$ is uniquely determined by the numbers $n_{ij}$. Since the fibers and $A_{ij}$ are Lagrangian, we conclude that $$\int_{c^{-1}(S)}\omega =0.$$ Consider the exact sequence $$H_2(\partial)\to H_2(Y_t-D_t)\to H_2(Y_t-D_t,\partial)$$ from Section 1.1. While there is no pairing between $H^2(Y_t-D_t)$, which contains $[\omega]$, and $H_2(Y_t-D_t,\partial)$, which contains $[c^{-1}(S)]$, we may conclude that the monodromy-invariant classes in $H_2(Y_t-D_t)$ pair with $[\omega]$ to give zero, since these map to the monodromy-invariant classes in $H_2(Y_t-D_t,\partial)$. We remark that that $[\omega]$ pairs with $\Z\gamma$, the image of $H_2(\partial)$, to be zero.  

Since the monodromy-fixed classes in $H_2(Y_t-D_t)$ are those that pair with $\lambda$ to zero, we conclude $[\omega]=c\lambda\mod \gamma$ in the sense above, for some constant $c\in \R$. Finally, we note that $$[\omega]\cdot [\omega]=2\cdot \textrm{Area}(\Gamma(Y_0))=\#(\textrm{triple points of }Y_0)=\lambda\cdot \lambda.$$ Hence $c=\pm 1$. \end{proof}

\begin{remark} The above proposition has an interpretation in terms of mirror symmetry. Taking the connected sums of the almost toric fibrations $\mu_i$ gives a topological model for the SYZ fibration on a general fiber of the algebraic degeneration $\mathcal{Y}$. This fibration is, topologically, the SYZ dual of the fibration $(X,D,\omega)\rightarrow B$, and thus, the fact that $(Y_t,D_t)$ and $(X,D)$ are diffeomorphic is essentially a fluke due to working in dimension two.

If we choose $\omega_i$ in Proposition \ref{symp} such that $[\omega_i]\cdot D_{ij} = [\omega_j]\cdot D_{ji}$, then the bases $B_i$ glue together to produce an integral-affine structure on $\Gamma(Y_0)^{\vee}\backslash\{v_0\}$, see for instance the first author's thesis, Example 6.3.5. This integral-affine structure, when ``seen from a great distance" in the language of Gross-Hacking-Keel, is the main integral-affine manifold of interest in \cite{GHK2}, cf.\ Sections 0.3 and 1.2. It is essentially the pseudo-fan $\mathfrak{F}(Y,D)$, but is constructed by gluing copies of $\R^2$ together, rather than lattice triangles. The ``fine detail" lost by viewing the SYZ fibration from a great distance is exactly the information of the SNC resolution of the smoothing.

Proposition \ref{symp} has appeared in various forms throughout the mirror symmetry literature. Given a non-singular affine surface $B$, its ``radiance obstruction" is the class in $H^1(B,T_B^{\textrm{flat}})$ which in local affine coordinates is represented by $$(\partial/\partial x) \otimes dx + (\partial/\partial y) \otimes dy.$$ The evaluation of this $1$-cocycle on a $1$-chain in $B$ coincides with the evaluation of the symplectic form $\omega$ on the associated cylinder, see e.g. Section 3.1.1 of \cite{kontsoib}. The radiance obstruction measures the obstruction to the existence of a global flat section of the SYZ fibration. Furthermore, the monodromy about a large complex structure limit point of Calabi-Yau manifolds is expected to be a higher dimensional analogue of a Dehn twist given by translation by a section of the SYZ fibration, see Conjecture 3.7 of \cite{grossspecial}. Thus the monodromy action on cohomology is given by cup product with the radiance obstruction, as shown in Theorem 5.1 of \cite{gs2}. In fact, it is likely that the methods of proof there could verify Proposition \ref{symp} in this noncompact situation.
\end{remark}

We may now strengthen Looijenga's conjecture to give a lower bound for the number of smoothing components of a cusp singularity:

\begin{theorem}\label{deftypes} The number of smoothing components of the cusp singularity $p'$ with minimal resolution $D'$ is greater than or equal to the number of deformation families of anticanonical pairs $(X,D)$. \end{theorem}

\begin{proof} The construction of \cite{Engel} inputs an anticanonical pair $(X,D)$ and outputs some Type III anticanonical pair $Y_0$ which is generic such that $$\mu\,:\,(X,D,\omega)\rightarrow \Gamma(Y_0)$$ is the extension of a Lagrangian torus fibration as in Proposition \ref{extend}, though $\omega$ is not kept track of in \cite{Engel}. By Proposition \ref{symp}, there is a diffeomorphism $(X,D)\rightarrow (Y_t,D_t)$ such that $\phi(D_i)=D_i$. By Theorem 5.14 of \cite{Friedman2}, the existence of a diffeomorphism between two anticanonical pairs preserving the classes $[D_i]$ implies that the pairs are deformation equivalent. Thus $(Y_t,D_t)$ is deformation-equivalent to $(X,D)$. Hence, we have constructed a semistable model of a smoothing of $p'$ whose generic fiber is deformation-equivalent to $(X,D)$. Furthermore, if two one-parameter smoothings have non-deformation equivalent generic fibers, they lie on different smoothing components. The theorem follows. \end{proof}

The Type III anticanonical pair $Y_0$ in Theorem \ref{deftypes} is in fact reconstructed from $\mu$ by triangulating the base and inverting the procedure defined in Section 2.2 which constructs $\Gamma(Y_0)$ from $Y_0$. Essentially, Proposition \ref{symp} and the reconstruction procedure in \cite{Engel} reduce the problem of finding a Type III degeneration $\mathcal{Y}\rightarrow \Delta$ whose monodromy invariant is any given element of $\mathcal{B}_{\text{gen}}\cap\Lambda$ to finding an almost toric fibration $\mu:(X,D,\omega)\rightarrow S^2$ whose symplectic form has any given class in $\mathcal{B}_{\text{gen}}\cap\Lambda$. This requires a more delicate construction than in \cite{Engel}, and is the purpose of Section 5.

\section{Birational modification and base change of Type III degenerations}

In this section we collect various ways to modify a Type III degeneration $\mathcal{Y}\rightarrow \Delta$, and then describe effect of these modifications on the dual complex of the central fiber $\Gamma(Y_0)$, as an integral-affine surface.

Let $C\cong \mathbb{P}^1$ be a smooth rational curve in a smooth threefold $\mathcal{Y}$ whose normal bundle is $\mathcal{O}_{C}(-1)\oplus \mathcal{O}_{C}(-1)$. Then the exceptional locus of the blowup $\operatorname{Bl}_C\mathcal{Y}$ is isomorphic to $\mathbb{P}^1\times \mathbb{P}^1$ and we may blow down by contracting along the other ruling of $\mathbb{P}^1\times \mathbb{P}^1$. Such a birational modification is an example of a \textsl{flop}.  Flops centered on the central fiber of $\mathcal{Y}$ are (rather confusingly) called Type 0, I, and II birational modifications depending on how they meet the double locus of $Y_0$:

\begin{definition} Let $\mathcal{Y}\rightarrow \Delta$ be a Type III degeneration. A flop centered on $C\subset Y_0$ is called:

\begin{enumerate}
\item[(i)] A Type 0 modification if $C$ does not intersect $(Y_0)_{\text{\rm{sing}}}$. In this case, $C$ is a smooth $-2$-curve on a component of $Y_0$ and $C$ does not deform to the general fiber. (However, somewhat more general modifications are also allowed.) 
\item[(ii)] A Type I modification if $C$ intersects $(Y_0)_{\text{\rm{sing}}}$ but is not contained in it. In this case, $C$ is an exceptional curve on a component $V_i$ of $Y_0$.
\item[(iii)] A Type II modification if $C$ is contained in $(Y_0)_{\text{\rm{sing}}}$. In this case, $C=C_{ij}$ is an exceptional curve on both components $V_i$ and $V_j$ which contain it.

\end{enumerate}
\end{definition}

\begin{proposition}\label{flop} Let $\mathcal{Y}\rightarrow \Delta$ be a Type III degeneration, and let $C\subset Y_0$ be the center of a flop.

\begin{enumerate}
\item[\rm(i)] A Type 0 modification does not change the isomorphism type of $Y_0$.
\item[\rm(ii)] A Type I modification blows down $C$ on the component $V_i$ containing it and blows up the point $C\cap V_j$ on $V_j$, where $V_j$ is the unique component of $Y_0$ such that $C\cdot C_{ij}=1$. See Figure \ref{typeone}.
\item[\rm(iii)] A Type II modification blows down $C=C_{ij}$ on the two components $V_i$ and $V_j$ containing it, and blows up the points $C\cap V_k$ and $C\cap V_{\ell}$ on the two remaining components $V_k$ and $V_{\ell}$ which intersect $C$. See Figure \ref{typetwo}.
\end{enumerate}
 \end{proposition}

\begin{proof} Since the components of $Y_0$ are locally smooth divisors in $\mathcal{Y}$, the blow-up and blow-down commute with restriction to the components of $Y_0$. The proposition follows. Note that the Type I modification is an internal blow-down on $V_i$ and an internal blow-up on $V_j$ while the Type II modification is a corner blow-down on $V_i$ and $V_j$ and a corner blow-up on $V_k$ and $V_{\ell}$. \end{proof}

\begin{figure}[h]
\begin{centering}
\includegraphics[width=3.25in]{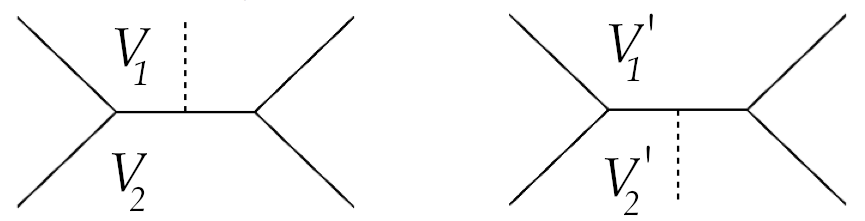} 
\caption{Two components of the central fiber before and after a Type I modification.}
\label{typeone}
\end{centering}
\end{figure}

\begin{figure}
\begin{centering}
\includegraphics[width=3in]{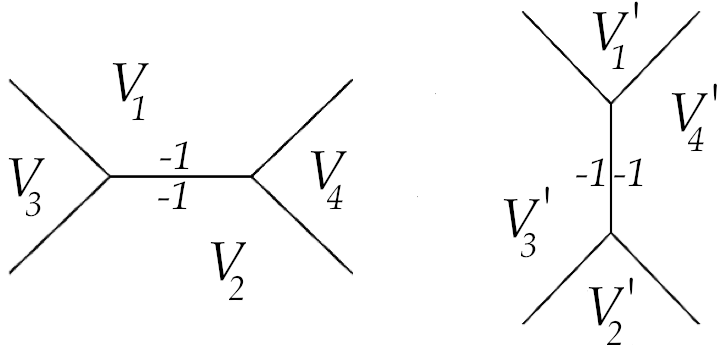} 
\caption{Four components of the central fiber before and after a Type II modification.}
\label{typetwo}
\end{centering}
\end{figure}

\begin{figure}
\begin{centering}
\includegraphics[width=4.5in]{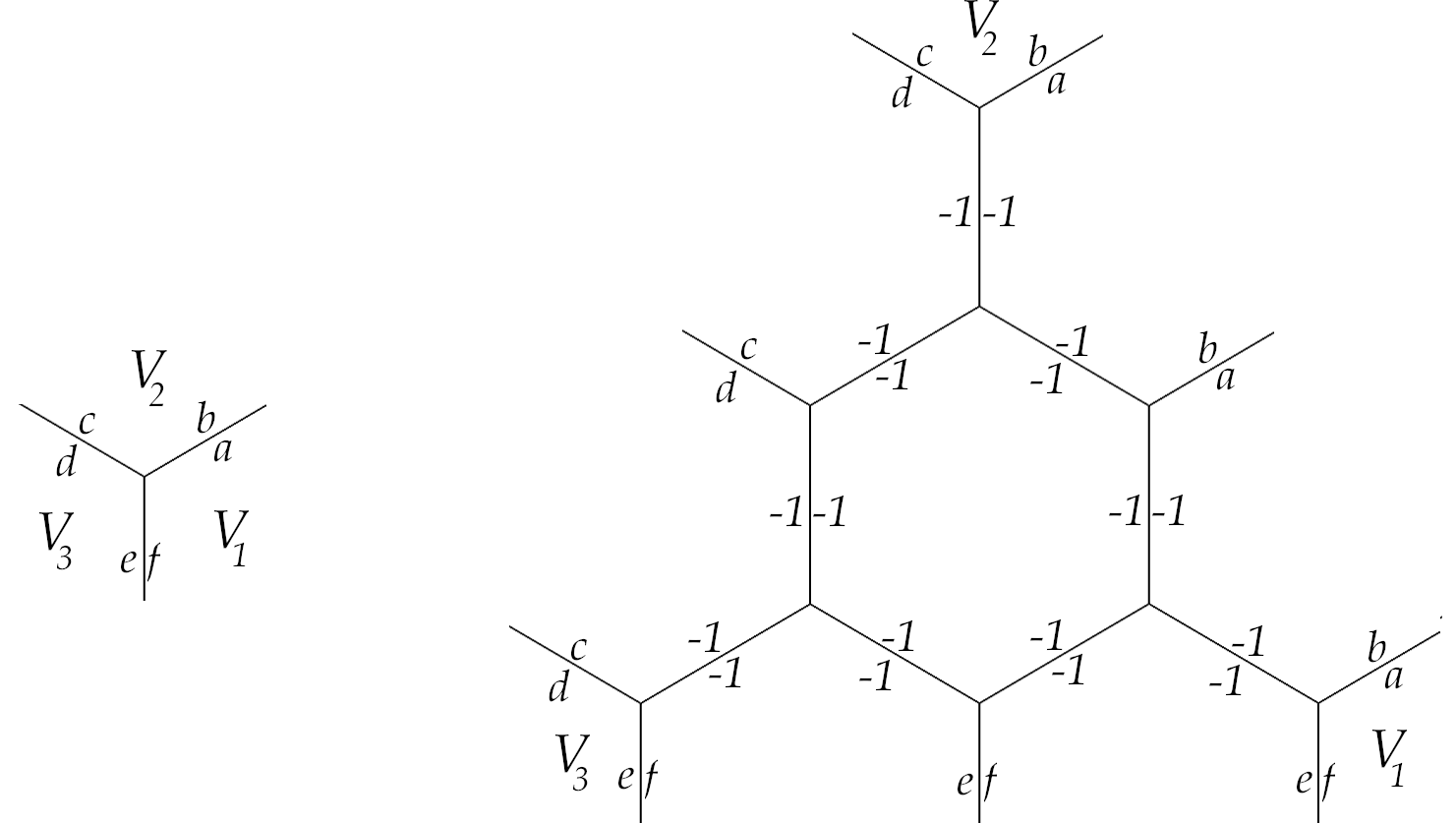} 
\caption{On the right, the components of the central fiber $Y_0[3]$ of the standard resolution $\mathcal{Y}[3]\rightarrow \Delta$ an order $3$ base change.}
\label{basechange}
\end{centering}
\end{figure}

Next, we describe a standard resolution of an order $k$ base change of a Type III degeneration which produces another Type III degeneration.

\begin{proposition}\label{base} Let $\mathcal{Y}\rightarrow \Delta$ be a Type III degeneration. Consider the order $k$ base change: $$\begin{CD} \mathcal{Y}\times_k\Delta @>>> \mathcal{Y} 
\\
 @VVV   @VVV \\
\Delta @>{t\mapsto t^k}>> \Delta
\end{CD}$$ There is a standard resolution $\mathcal{Y}[k]\rightarrow \mathcal{Y}\times_k\Delta$, see Figure \ref{basechange}, such that $\mathcal{Y}[k]\rightarrow \Delta$ is a Type III degeneration. As an integral-affine manifold the dual complex of the central fiber $\Gamma(Y_0[k])$ is the subdivision of each triangle of $\Gamma(Y_0)$ into $k^2$ triangles in the standard manner, where each of the $k^2$ sub-triangles has the structure of a lattice triangle of area $\frac{1}{2}$. \end{proposition}

\begin{proof} See \cite{FriedmanBaseChange} for proof of the first statement. The second claim then follows directly from the definition of the integral-affine structure on the dual complex, see Section 2.2. \end{proof}

\begin{remark}\label{birationallambda} Note that Type 0, I, and II modifications do not change the monodromy invariant of the degeneration $\mathcal{Y}\rightarrow \Delta$ because the punctured family $\mathcal{Y}^*\rightarrow \Delta^*$ is unaffected by the modifications. On the other hand, taking the base change $\mathcal{Y}[k]\rightarrow \Delta$ multiplies the monodromy invariant by $k$, because the base change $t\mapsto t^k$ is a $k$-fold covering of the punctured family. \end{remark}

We now examine the effect of the birational modifications on the dual complex. The Type II modification is simplest:

\begin{proposition}\label{two} Let $\mathcal{Y}\dashrightarrow \mathcal{Y}'$ be a Type II modification along the double curve $C_{ij}=V_i\cap V_j$. The edge $e_{ij}$ of $\Gamma(Y_0)$ associated to $C_{ij}$ is the diagonal of a quadrilateral $(v_i,v_k,v_j,v_{\ell})$ integral-affine equivalent to a unit square. Furthermore $\Gamma(Y_0)$ and $\Gamma(Y_0')$ are isomorphic as integral-affine manifolds, but their triangulations differ: $\Gamma(Y_0')$ contains the other diagonal of this quadrilateral. \end{proposition}

\begin{proof} Since $C_{ij}^2=-1$, the directed edges of $\Gamma(Y_0)$ satisfy the equation \begin{equation} e_{ik}+e_{i\ell}=e_{ij}.\end{equation} Eq.\ (1) implies the first statement because the integral-affine structure on $(v_i,v_k,v_j,v_{\ell})$ is uniquely determined by $C_{ij}^2$ and the unit square satisfies Eq (1). The second statement follows from (iii) of Proposition \ref{flop}: Switching the diagonal of the quadrilateral corresponds exactly to the changes in the self-intersection sequence of the four anticanonical pairs $V_i$, $V_j$, $V_k$, and $V_{\ell}$ involved in the Type II modification. \end{proof}

Before we describe the effect of a Type I modification on $\Gamma(Y_0)$, we must introduce a surgery on integral-affine surfaces:

\begin{definition} Let $S_0$ be an integral-affine surface and let $p\in S_0$ be an $A_1$ singularity. Parameterize a monodromy-invariant ray originating at $p$ by a segment $$\gamma\colon [0,\epsilon]\rightarrow S_0$$ where $\gamma(0)=p$. There is a family of integral-affine manifolds $S_t$ for all $t\in [0,\epsilon]$ which are isomorphic to $S_0$ in the complement of $\gamma([0,t])$ and which have an $A_1$ singularity at $\gamma(t)$. We say that $S_t$ is a \textsl{nodal slide} of $S_0$. See \cite{Symington}, Section 6. \end{definition}

\begin{remark} It is shown in \cite{Symington} that if the surfaces $S_t$ are bases of almost toric fibrations $(X_t,\omega_t)\rightarrow S_t$, then $(X_t,\omega_t)$ are symplectomorphic. In particular, the class $[\omega_t]$ is constant under the Gauss-Manin connection. \end{remark}

To quantify the ``amount" of nodal slide, we need a natural measure of the length of a straight line segment in an integral-affine manifold:

\begin{definition} Let $L\subset S$ be a straight line segment with rational slope in an integral-affine surface $S$. The \textsl{lattice length} of $L$ is the unique positive multiple $r$ of a primitive integral vector $v$ such that $rv=L$ in some chart. \end{definition}

We may now describe the effect of a Type I modification on $\Gamma(Y_0)$:

\begin{proposition}\label{one} Let $\mathcal{Y}\dashrightarrow \mathcal{Y}'$ be a Type I modification along a curve $C$ such that $C\cdot C_{ij}=1$. Then $\Gamma(Y_0)$ and $\Gamma(Y_0')$ are isomorphic as simplicial complexes but the integral-affine structure on $\Gamma(Y_0')$ is the result of a nodal slide of lattice length $1$ along the directed edge $e_{ij}$ of $\Gamma(Y_0)$. \end{proposition}

\begin{proof} The proof follows from Proposition \ref{flop}(ii) and Proposition 2.10 of \cite{Engel}, which describes the effect of an internal blow-up on the pseudo-fan of an anticanonical pair. \end{proof}

\begin{remark}\label{factor1} Let $S_t$ for $t\in [0,1]$ denote the family of integral-affine surfaces in Proposition \ref{one}, so that $S_0=\Gamma(Y_0)$ and $S_1=\Gamma(Y_0')$. Then $S_t$ has an $A_1$ singularity at $t$ for all $t\in(0,1)$ but at $t=0,1$, this singularity may collide with an already existing singularity. In particular, $p_i$ is singular for $t\neq 0$ if $Q(V_i',C_i')\geq 1$ and $p_j$ is singular for $t\neq 1$ if $Q(V_j,C_j)\geq 1$. In these cases, the singularity at $p_i$ or $p_j$ \textsl{factors} into an $A_1$ singularity and the remaining singularity as one performs the nodal slide which extracts the $A_1$ singularity. \end{remark} 

In light of Remark \ref{factor1}, we introduce the following definition:

\begin{definition}\label{factor2} Let $p$ be an integral-affine surface singularity. A \textsl{factorization of $p$ into $A_1$ singularities} is a collection of nodal slides of $A_1$ singularities which collide to form $p$. \end{definition}

\begin{proposition}\label{factor3} Let $\mathfrak{F}(V,D)$ be the pseudo-fan of an anticanonical pair $(V,D)$. Let $\pi\colon (V,D)\rightarrow (\overline{V},\overline{D})$ be a toric model, that is, a blowdown of $Q(V,D)$ disjoint internal exceptional curves $\{E_i\}$. Then $(\overline{V},\overline{D})$ is necessarily toric. Let $n_j:=\#\{E_i\,\big{|}\,E_i\cdot D_j=1\}$. Then $\mathfrak{F}(V,D)$ admits a factorization into $A_1$ singularities where $n_j$ of the singularities have monodromy-invariant lines along the edge $e_j\subset \mathfrak{F}(V,D)$ associated to the component $D_j$.  \end{proposition}

\begin{proof} Proposition 2.10 of \cite{Engel} gives a construction of $\mathfrak{F}(V,D)$ from $\mathfrak{F}(\overline{V},\overline{D})$ which corresponds exactly to colliding an $A_1$ singularities with monodromy-invariant lines along an edges of $\mathfrak{F}(\overline{V},\overline{D})$ associated to components of $\overline{D}$ receiving an internal blow-up. The Proposition follows. \end{proof}

\begin{example} The converse to Proposition \ref{factor3} is false. That is, even if $\mathfrak{F}(V,D)$ has a factorization into $A_1$ singularities with monodromy-invariant lines along a collection of edges $e_j$, the pair $(V,D)$ may not have a toric model with exceptional curves meeting the associated components $D_j$. For instance, keeping track of the labeling, there are two deformation types of anticanonical pairs $(V,D_1+\cdots+D_5)$ with $$(D_i^2)=(1, -1-2,-2, -1).$$ One has disjoint exceptional curves meeting $D_2$ and $D_4$ and the other has disjoint exceptional curves meeting $D_3$ and $D_5$, but these two possibilities are mutually exclusive. For either deformation type, $\mathfrak{F}(V,D)$ is the same. Thus $\mathfrak{F}(V,D)$ admits two factorizations into $A_1$ singularities, only one of which comes from a toric model of $(V,D)$. \end{example}

\section{Degenerations with given monodromy invariant}

\subsection{Construction of degenerations}

We have most of the necessary tools to prove that any $\lambda\in\mathcal{B}_{\text{\rm{gen}}}\cap \Lambda$ is the monodromy invariant of some Type III degeneration, but first we must introduce a surgery on integral-affine surfaces corresponding to an internal blow-up. Let $(X,D,\omega)\rightarrow B$ be an almost toric fibration. Then $P=\partial B$ is locally polygonal, and the components $D_i\subset D$ fiber over the edges $P_i\subset P$.

\begin{definition} An \textsl{almost toric blow-up}, see \cite{Symington} Section 5.4 or \cite{Engel} Definition 3.3, is a surgery on $B$ which produces a new integral-affine surface with singularities $\tilde{B}$ and an almost toric fibration $$(\tilde{X},\tilde{D},\tilde{\omega})\rightarrow \tilde{B}$$ whose base is $\tilde{B}$. To construct $\tilde{B}$, we remove a triangle, called the \textsl{surgery triangle}, in $B$ resting on an edge $P_i$ and glue the two remaining edges together by a matrix conjugate to $$\twobytwo{1}{1}{0}{1},$$ at the expense of introducing an $A_1$ singularity at the interior vertex of the triangle, whose monodromy-invariant line is parallel to $P_i$. Then $$\pi\colon(\tilde{X},\tilde{D})\rightarrow (X,D)$$ is an interior blow-up at a point on $D_i$. The \textsl{size} $b$ of the surgery is the lattice length of the base of the surgery triangle, and is the area of the exceptional curve with respect to $\tilde{\omega}$. In fact $[\tilde{\omega}]=\pi^*[\omega]-bE$ where $E$ is the class of the exceptional curve.\end{definition}

A sphere in $\tilde{X}$ representing the exceptional curve can be constructed which fibers over the path in $\tilde{B}$ which is the identified pair of edges of the surgery triangle. This path connects $\tilde{P}_i$ to the $A_1$ singularity introduced. One forms the sphere from a family of circles homologous to the vanishing cycle of the nodal fiber over the singularity. These circles collapse to points over the two endpoints of the path.

\begin{remark} Even when we begin with a symplectic anticanonical pair $(X,D,\omega)$ which has a compatible complex structure, such as a symplectic toric surface, Symington's blow-up operation is a priori purely symplectic-topological. We never use a complex structure on $(X,D,\omega)$. But, it is possible when the size $b$ is small to upgrade the blow-up operation for Lagrangian fibrations with respect to a K\"ahler form, see e.g. Section 4 of \cite{abouzaid}. In addition, it is possible to retain a Hamiltonian $S^1$ action in a neighborhood of each exceptional curve. \end{remark}

\begin{definition} Let $L$ be a straight line segment in a non-singular integral-affine surface $S$ and let $v$ be a primitive integral vector pointing along $L$. Suppose that $\gamma\colon [0,1]\rightarrow S$ is a path in $S$ such that $\gamma(0)\in L$. The \textsl{lattice distance} (along $\gamma$) from $L\subset B$ to $\gamma(1)$ is the integral $$\int_0^1 \gamma'(t)\times v \,dt.$$ Note that the cross product is preserved by integral-affine transformations, and thus the integral is well-defined. In fact, the lattice length is invariant if we deform the path within the nonsingular locus of $B$, so long as we fix $\gamma(1)$ and keep the condition $\gamma(0)\in L$. The lattice distance from the interior vertex of a surgery triangle to its base is necessarily equal to the lattice length of the base. \end{definition}

We begin with a theorem regarding the class of an \emph{ample} divisor:

\begin{theorem}\label{anylambda} Let $(X,D)$ be an anticanonical pair. Suppose $\lambda\in \mathcal{A}_{\text{\rm{gen}}}$. There exists an almost toric fibration $(X,D,\omega)\rightarrow B$ such that $[\omega]=\lambda$. \end{theorem}

\begin{proof} Suppose that $r(D)\geq 3$ and that $\lambda\in\mathcal{A}_{\text{\rm{gen}}}\cap H^2(X;\Z)$. By Proposition \ref{effdecomp}, we may express $\lambda$ in the form $$\lambda=\sum a_j[D_j]+\sum b_{ij}[E_{ij}]$$ with $E_{ij}$ disjoint exceptional curves meeting $D_j$. Consider the blowdown $$(X,D)\rightarrow (X_0,D_0)$$ of the curves $\{E_{ij}\}$. By Proposition 1.3 of \cite{GHK}, after some corner blow-ups $(X_1,D_1)\rightarrow (X_0,D_0)$ there is a toric model $(X_1,D_1)\rightarrow (\overline{X},\overline{D})$. Since $D\rightarrow D_0$ is an isomorphism, we may perform the same sequence of corner blow-ups $\nu\colon(\tilde{X},\tilde{D})\rightarrow (X,D)$ to produce a toric model $$\pi\colon(\tilde{X},\tilde{D})\rightarrow (\overline{X},\overline{D})$$ which blows down the $E_{ij}$ and possibly other internal exceptional curves.

Pulling back the expression for $\lambda$ gives $$\nu^*\lambda = \sum c_j[\tilde{D}_j]+\sum b_{ij}[E_{ij}]$$ for some constants $c_j$. For simplicity of notation, we index the components of $\tilde{D}$ such that $\nu(\tilde{D}_j)=D_j$. Then $c_j=a_j$ when $\tilde{D}_j$ is not one of the corner blow-ups of $\nu$. Furthermore, we identify $\nu^{-1}(E_{ij})$ with $E_{ij}$. Because $\lambda$ is nef, we have $\lambda\cdot [E_{ij}]=c_j - b_{ij}\geq 0$. Now re-write the above expression as $$\nu^*\lambda=\sum c_j\pi^*([\overline{D}_j])+\sum (b_{ij}-c_{j})[F_{ij}]$$ where $\{F_{ij}\}\supset \{E_{ij}\}$ is the set of exceptional curves blown down by $\pi$. We have $b_{ij}=0$ exactly for those exceptional curves in $\{F_{ij}\}- \{E_{ij}\}$. The coefficients $b_{ij}-c_j$ are non-positive, and have absolute value less than or equal to $c_j$ with equality if and only if the exceptional curve  lies in $\{F_{ij}\}- \{E_{ij}\}$. We then have the following lemma, which proves Theorem~\ref{anylambda} under a technical assumption which we will later remove. 

\begin{lemma}\label{nooverlap} Let $\lambda\in \mathcal{A}_{\text{\rm{gen}}}\cap H^2(X;\Z)$ and suppose $r(D)\geq 3$. In the notation above, suppose that $b_{ij}\neq 0$ for all $F_{ij}$. Then the conclusion of Theorem~\ref{anylambda} holds for $\lambda$. \end{lemma}

\begin{proof} Linearizing the action of $(\C^*)^2$ on the space of sections $$H^0(\overline{X};\mathcal{O}_{\overline{X}}(\pi_*(\nu^*\lambda)))$$ gives a polytope $\overline{B}$ for $(\overline{X},\overline{D})$ whose integral points correspond to torus equivariant sections. Note that $\sum c_j\overline{D}_j$ is a torus invariant divisor in the linear system of $\pi_*(\nu^*\lambda)$ because it is a linear combination of boundary components, so it corresponds to a lattice point $p\in\overline{B}$.

The moment map is a Lagrangian torus fibration $$(\overline{X},\overline{D})\rightarrow \overline{B}$$ such that the components $\overline{D}_j$ fiber over the components $\overline{P}_j$. We will produce via almost toric blowups on $\overline{B}$ an almost toric fibration $(X,D,\omega)\rightarrow B$ such that $[\omega]=\lambda$. We must therefore find for all $F_{ij}$ disjoint surgery triangles of size $m_{ij}:=c_{j}-b_{ij}$ resting on the edge $\overline{P}_j$. To see why we can perform the necessary surgeries, we need the following facts:

\begin{enumerate}

\item[] Fact 1: The point $p$ has lattice distance $c_j$ from the edge $\overline{P}_j$. This is a well-known formula in toric geometry.
\item[] Fact 2: The lattice length of $\overline{P}_j$ is at least the sum of the lengths of the bases of all surgery triangles which must rest on $\overline{P}_j$.

\end{enumerate}

Fact 2 follows from the ampleness of $\lambda$---we have the equation $$\textrm{lattice length}(\overline{P}_j)=\pi_*(\nu^*\lambda)\cdot [\overline{D}_j]=\nu^*\lambda\cdot \tilde{D}_j+\sum_j m_{ij}$$ and thus the lattice length of $\overline{P}_j$ is at least $\sum_j m_{ij}$. By Fact 1, $p$ is at least as distant as is necessary to perform the surgeries associated to $\{F_{ij}\}_i$ within the triangle formed by $\overline{P}_j$ and $p$---each surgery has size $m_{ij}$. By Fact 2, the length of $\overline{P}_j$ is large enough to accommodate all the surgeries.

Figure \ref{surgery} illustrates a general method for finding the necessary triangles, though the figure is deceptive in two ways: (1) Two $F_{ij}$ satisfy $b_{ij}=0$ and correspondingly there are two surgery triangles which involve $p$, whereas by assumption of the lemma $b_{ij}\neq 0$ and (2) the edges of $B$ will have nonzero lattice length after the surgeries because $\lambda$ is ample. We will ultimately relax both the ampleness condition and the assumption that $b_{ij}\neq 0$.

Let $P=\partial B$. After surgeries, the lattice length of the edge $P_j\subset P$ over which $D_j$ fibers is $\nu^*\lambda\cdot [\tilde{D}_j]$, which is positive, since $\lambda\in \mathcal{A}_{\text{\rm{gen}}}$. Since each edge of $B$ corresponding to one of the corner blow-ups in $$\nu:(\tilde{X},\tilde{D})\rightarrow (X,D)$$ has length zero, we have that $B$ is in fact the base of an almost toric fibration $(X,D,\omega)\rightarrow B$ such that $[\omega]=\lambda$. \end{proof}

\begin{figure}
\includegraphics[width=4.5in]{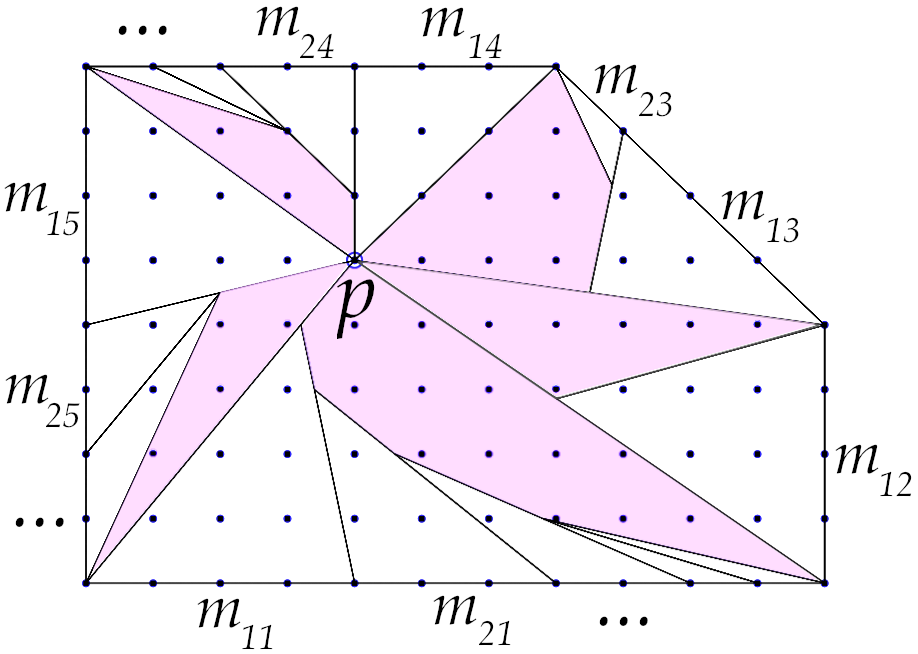} 
\caption{Surgery triangles in $\overline{B}$ of size $m_{ij}:=c_j-b_{ij}$.}
\label{surgery}
\end{figure}

We must now remove the assumption that $b_{ij}\neq 0$. When some $b_{ij}=0$, the proof of Lemma \ref{nooverlap} fails because multiple surgery triangles may overlap at the vertex $p$. See for instance Figure \ref{surgery} (though as before, the figure is still deceptive because the components $P_j$ have positive lattice length when $\lambda$ is ample). We may still construct $B$ as an integral-affine surface, but $B$ is not the base of an almost toric fibration, because $p$ need not be an $A_k$ singularity. Thus we show:

\begin{lemma}\label{slidefactor} There is a family of integral-affine surfaces $B_t$ for $t\in[0,\epsilon]$ with $B_0=B$ which factorizes of $p$ into $A_1$ singularities along nodal slides. Furthermore, $B_t$ for $t\neq 0$ is the base of an almost toric fibration $$(X_t,D_t,\omega_t)\rightarrow B_t$$ such that $[\omega_t]=\lambda$. \end{lemma}

\begin{proof} Each surgery triangle resting on $\overline{P}_j$ introduces an $A_1$ singularity with monodromy-invariant line parallel to $P_j$. We may therefore factor $p$ into a collection of $A_1$ singularities $\{q_{ij}\}$, one for each exceptional curve $F_{ij}$ such that $b_{ij}=0$. To construct $B_t$, we simultaneously perform nodal slides along segments which are cyclically ordered about $p$ in the same ordering as the edges of $P$. Consider the almost toric fibration $(X_t,D_t,\omega_t)\rightarrow B_t$. We must show that $(X_t,D_t)$ is diffeomorphic to $(X,D)$ and that $[\omega_t]=\lambda$.

There is a collection of disjoint paths $$\gamma_{ij}\colon [0,1]\to B_t$$ satisfying $\gamma_{ij}(0)\in P_j$ and $\gamma_{ij}(1)=q_{ij}$. Furthermore, the monodromy-invariant line of $q_{ij}$ is parallel to $P_j$ under the trivialization of integral-affine structure along $\gamma_{ij}$. Thus, as is the case for the usual almost toric blow-up, there are smooth $(-1)$-spheres $F_{ij}$ which fiber over $\gamma_{ij}$ and intersect $D_j$. So $(X_t,D_t)$ and $(X,D)$ have the same toric model, and are thus diffeomorphic. To prove $[\omega_t]=\lambda$, we must show $[\omega_t]\cdot F_{ij}= m_{ij}$. But $\int_{F_{ij}}\omega_t$ is the lattice distance from $P_j$ to $q_{ij}$. This lattice distance is $m_{ij}$, because the lattice distance from $P_j$ to $p$ is equal to $m_{ij}$ and nodal slides parallel to $P_j$ keep the lattice distance constant.  \end{proof}

Hence we have proven Theorem \ref{anylambda} for any $\lambda\in \mathcal{A}_{\text{\rm{gen}}}\cap H^2(X;\Z)$ such that $r(D)\geq 3$. We may cheaply deduce the result for $\lambda\in \mathcal{A}_{\text{\rm{gen}}}\cap H^2(X;\Q)$ by scaling the symplectic form. To prove Theorem \ref{anylambda} for real classes, we must show that the conclusion of Proposition \ref{effdecomp} holds with real coefficients for any $\lambda\in \mathcal{A}_{\text{\rm{gen}}}$. We avoid going into details, since we will not use this result in the rest of the paper, but the proof follows from a continuity argument---the polyhedra in $\mathcal{A}_{\text{\rm{gen}}}$ defined by $$\left\{\sum a_jD_j+\sum b_iE_i\,\big{|}\, a_j,b_i\in \R^{\geq 0}\right\},$$ 
where the $E_i$ are disjoint exceptional curves, are locally finite, and thus their union is closed. Furthermore, their union contains $\mathcal{A}_{\text{\rm{gen}}}\cap H^2(X;\Q)$. Thus, any $\lambda\in \mathcal{A}_{\text{\rm{gen}}}$ may be expressed as $\sum a_jD_j+\sum b_iE_i$ for some $a_j,b_i\in \R^+$ and the proof of Theorem \ref{anylambda} goes through as before.

Finally, when $r(D)\leq 2$ we perform some corner blow-ups until $r(D)\geq 3$. Then we may apply Proposition \ref{effdecomp} and the proof applies, with the caveat that we must blow down these initial corner blow-ups (as the symplectic form will be degenerate on them) to produce a non-degenerate almost toric fibration on $(X,D)$. \end{proof}

\begin{remark}\label{degenerate} We can weaken the assumption of Theorem \ref{anylambda} by assuming only that $\lambda$ is nef and big. In this case, we can apply the method of Theorem \ref{anylambda}, but there are two ways the fibration $\mu$ can degenerate.

First, the length of a boundary component $P_j\subset P$ has length zero if $\lambda\cdot [D_j]=0$. This results in a fibration with $\mu(D_j)$   a point, in fact a vertex of $B$. Furthermore, the symplectic form is only non-degenerate  on the complement of $D_j$. If $\lambda\cdot [D_j]=0$ for all $j$, that is $\lambda\in \Lambda$, then the whole boundary $P$ consists of a single point, and thus the base of $\mu$ is an integral-affine sphere and the symplectic form is only non-degenerate on $X-D$. This is the case shown in Figure \ref{surgery}.

Second, if $\lambda\cdot E_{ij}=0$ then the surgery associated to $E_{ij}$ will have size zero, that is, $\mu(E_{ij})$ will be some point on the interior of the edge $P_j$ and $\omega$ will be the pull-back of a symplectic form from the blowdown of $E_{ij}$. \end{remark}

\begin{proposition}\label{makeint} Let $\lambda\in H^2(X;\Z)$ be big and nef. Let $(X,D,\omega)\rightarrow B$ be the almost toric fibration constructed in Theorem \ref{anylambda}, which may be degenerate in the sense of Remark \ref{degenerate}. There are nodal slides on $B$ until every singularity lies at an integral point. \end{proposition}
 
\begin{proof} First, undo the nodal slides of Lemma \ref{slidefactor} which factor $p$ into $A_1$ singularities, as $p$ is an integral point. Even so, as shown in Figure \ref{surgery}, the surgery associated to an exceptional curve $E_{ij}$ may introduce a singularity at a rational but non-integral point, and thus we must modify the construction of Lemma \ref{nooverlap} slightly. Within the triangle whose base is $\overline{P}_j$ and whose third vertex is $p$, we fit triangles $\Delta_{ij}$ whose base is a subsegment of $\overline{P}_j$ of lattice length $m_{ij}$ and whose third vertex is $p$, see Figure \ref{integral}.

We would like to show that there is a \emph{lattice point} $p_{ij}$ of lattice distance $m_{ij}$ from $\overline{P}_j$ within the triangle $\Delta_{ij}$. Let $L_{ij}$ denote the segment inside $\Delta_{ij}$ of points whose lattice distance is $m_{ij}$ from $\overline{P}_j$. Then the lattice length of $L_{ij}$ is equal to $$m_{ij}\frac{c_j-m_{ij}}{c_j}.$$ If this length is at least $1$, then $L_{ij}$ contains an integral point. Note $$m_{ij}\frac{c_j-m_{ij}}{c_j}\geq 1\iff m_{ij}(c_j-m_{ij})\geq c_j.$$ Thus $L_{ij}$ contains an integral point $p_{ij}$ whenever $2\leq m_{ij}\leq c_j-2$.

The remaining cases are when $m_{ij}=0$, $1$, $c_j-1$, or $c_j$. The cases $m_{ij}=0$ and $m_{ij}=c_j$ are trivial---in the former case, there is no surgery to make and in the latter case, we must choose $p_{ij}=p$. Consider the cases $m_{ij}=1$ and $m_{ij}=c_j-1$. Then $L_{ij}$ has lattice length $\frac{c_j-1}{c_j}$ and its endpoints are rational points with denominator $c_j$. Thus it contains a lattice point $p_{ij}$. Choosing the surgery triangle associated to $F_{ij}$ to be the triangle whose base is the segment of length $m_{ij}$ and whose third vertex is $p_{ij}$ produces an integral-affine surface with singularities at integral points which can be connected to $B$ by nodal slides. \end{proof}
 
\begin{figure}
\includegraphics[width=4.5in]{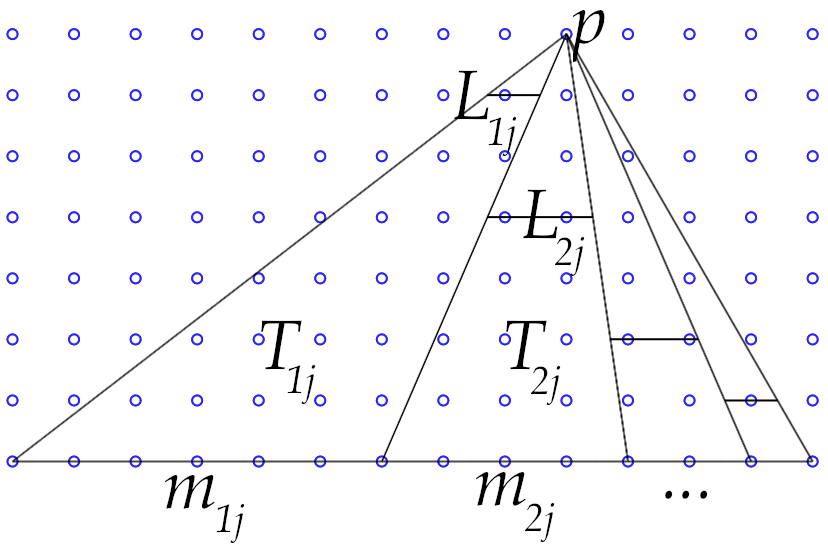} 
\caption{Finding integral points for the surgery triangles resting on the edge $\overline{P}_j$.}
\label{integral}
\end{figure}

\begin{theorem}\label{lambdas} For all $\lambda\in \mathcal{B}_{\text{\rm{gen}}}\cap \Lambda$ there is a Type III degeneration $\mathcal{Y}\rightarrow \Delta$ of anticanonical pairs such that the monodromy invariant of $\mathcal{Y}$ is $\lambda$. \end{theorem}

\begin{proof} By Theorem \ref{anylambda} and Remark \ref{degenerate} we may construct a degenerate almost toric fibration $(X,D,\omega)\rightarrow B'$ with $[\omega]=\lambda$. Then $B'$ is a sphere, and $D$ maps to a point $v_0\in B'$. Furthermore, we may assume that the only other singularities of $B'$ are type $A_1$. By Proposition \ref{makeint}, we may assume that there is an integral-affine surface $B$ with singularities at integral points, connected to $B'$ by small nodal slides of lattice lengths in $\frac{1}{n}\Z$. By Section 4 of \cite{Engel}, after triangulation, $B$ is the dual complex of the central fiber of a Type III degeneration $\mathcal{Y}\rightarrow \Delta$ of anticanonical pairs.  But because $\Gamma(Y_0)$ is not generic, we cannot apply Proposition \ref{symp} directly. 

Let $B'[n]$ denote the integral-affine manifold which refines the lattice in $B'$ to order $n$. Equivalently, we post-compose the charts on $B'$ with multiplication by $n$. Then $B'[n]$ is the base of an almost toric fibration on $(X,D,n\omega)$. Note that $B'[n]$ has singularities at integral points, and thus after triangulation, $B'[n]$ is the dual complex of the central fiber $Y_0'[n]$ of a Type III degeneration $$\mathcal{Y}'[n]\rightarrow \Delta.$$ Furthermore, $Y_0'[n]$ is generic, and thus by Proposition \ref{symp}, there is a diffeomorphism $\phi\colon X\rightarrow Y'_t[n]$ such that $n\phi_*([\omega])$ is the monodromy invariant of $\mathcal{Y}'[n]$.

Let $\mathcal{Y}[n]\rightarrow \Delta$ be the standard resolution as in Proposition \ref{base} of the base change of $\mathcal{Y}\rightarrow \Delta$ of order $n$. Then $\Gamma(Y_0[n])=B[n]$ and the triangulation is the standard order $n$ refinement of the triangulation of $B$. Note that $B'[n]$ is given by a set of nodal slides on $B[n]$ all of integer lattice length. Applying Proposition \ref{two}, we may perform Type II modifications on $\mathcal{Y}[n]$ to flip edges of the triangulation of $B[n]$ along diagonals of quadrilaterals. Choosing $n$ large enough, we may ensure by \cite{lawson} that there is a sequence of edge flips on $B[n]$ such that the segments along which one performs nodal slides to get from $B[n]$ to $B'[n]$ are edges of the resulting re-triangulation. Let $\mathcal{Y}[n]\dashrightarrow\mathcal{Y}''[n]$ be this sequence of Type II modifications. Then by Proposition \ref{one}, there is a series of Type I modifications $\mathcal{Y}''[n]\dashrightarrow \mathcal{Y}'[n]$ such that $\Gamma(\mathcal{Y}'_0[n])=B'[n]$, assuming we choose the appropriate toric model for the component of $Y_0''[n]$ corresponding to the point $p$.

Thus, $\mathcal{Y}[n]$ and $\mathcal{Y}'[n]$ are birational. By Remark \ref{birationallambda}, $n[\omega]=n\lambda$ is the monodromy invariant of $\mathcal{Y}[n]$. Again by Remark \ref{birationallambda}, the base change of order $n$ multiplies the monodromy invariant by $n$, and thus the monodromy invariant of $\mathcal{Y}\rightarrow \Delta$ is $\lambda$ under the identification of $H^2(Y_t;\Z)$ with $H^2(X;\Z)$. 

We have only proven existence up to diffeomorphism in the sense that we have constructed, for any $\lambda\in \mathcal{B}_{\text{\rm{gen}}}(X)\cap\Lambda$, a Type III degeneration $\mathcal{Y}\rightarrow \Delta$ and a diffeomorphism $\phi\colon X\rightarrow  Y_t$ such that $\phi_*[D_i]$ is the class of the corresponding component of $D_t$ and the monodromy invariant of $\mathcal{Y}$ is the push-forward of $\lambda$. By Theorem 5.14 of \cite{Friedman2}, we can conclude that $(X,D)$ and $(Y_t,D_t)$ are deformation equivalent. Hence, up to self-diffeomorphism of $Y_t$ preserving the components of $D_t$, every monodromy invariant can be attained. Finally, by Theorem 5.14 of \cite{Friedman2}, the subgroup of the self-diffeomorphism group of $Y_t$ which preserves the components of $D_t$ acts on cohomology by the group $\Gamma =\Gamma(Y_t,D_t)$ of admissible isometries, and thus, there is no difference between the monodromy invariant up to such diffeomorphisms and up to the action of $\Gamma$. \end{proof}

\begin{remark} Theorem \ref{lambdas} is to be expected from the perspective of mirror symmetry. The deformation space of the partially contracted Inoue surface $\overline{V}_0$ represents the ``complex moduli space" for mirror symmetry, and thus should be isomorphic to the ``K\"ahler moduli space" of the mirror $(X,D,\omega)$, cf. \cite{morrison}. When $\lambda\in \mathcal{A}_{\text{\rm{gen}}}- \mathcal{B}_{\text{\rm{gen}}}$ the linear system associated to $\lambda$ does not contract all of $D$ on the pair $(Y,D)$, and thus, the mirror should be a Landau-Ginzburg model, whereas when $\lambda\in \mathcal{B}_{\text{\rm{gen}}}$, the situation is more similar to the Calabi-Yau case, as opposed to the log Calabi-Yau  case.

In particular, given any class $\lambda\in\mathcal{B}_{\text{\rm{gen}}}$, there is a large symplectic structure limit on $(Y,D)$ given by scaling a K\"ahler form $\omega$ satisfying $[\omega]=\lambda$ in the complexified K\"ahler cone. Here $\omega$ is only non-degenerate on $Y-D$ and trivial on $D$. The corresponding deformation in the complex moduli space is a family of anticanonical pairs with monodromy invariant $\lambda$ over a punctured disk. This family should be fillable at the origin of the disc by the large complex structure limit, i.e. $\overline{V}_0$. Thus, one expects that every $\lambda\in\mathcal{B}_{\text{\rm{gen}}}$   arises as a monodromy invariant, and conversely that every monodromy invariant lies in $\mathcal{B}_{\text{\rm{gen}}}$, as shown in Proposition \ref{extendsections}. \end{remark}

\subsection{An application to symplectic geometry}

Theorem \ref{anylambda} has a number of interesting consequences for the symplectic geometry of anticanonical pairs. But first, we cite the following theorem about the symplectic geometry of rational surfaces, due to Li and Liu \cite{LiLiu}:

\begin{theorem}\label{rationalsymp} Let $X$ be a smooth $4$-manifold with $b_2^+(X)=1$ and let $C_{X,K}$ denote the cone of classes of symplectic forms on $X$ with symplectic canonical class $K$. Then \begin{align*} C_{X,K}=\{x\in \mathcal{C}^+\colon &x\cdot [E]>0\textrm{ for all contractible} \\ &2\textrm{-spheres }E\textrm{ s.t. }[E]\cdot K=-1\}\end{align*} \end{theorem}

Let us further restrict to the cone $C_{(X,D)}$ of classes of symplectic forms on $X$ along with a normal crossings symplectic anticanonical divisor $D$. By Theorem \ref{rationalsymp}, we can conclude that \begin{align*}C_{(X,D)}\subset \{x\in \mathcal{C}^+\colon &x\cdot [E]>0\textrm{ for all }E \\ &\textrm{exceptional and }x\cdot [D_j]>0\}= \mathcal{A}_{\text{\rm{gen}}}.\end{align*} Conversely, any element $x\in \mathcal{A}_{\text{\rm{gen}}}$ is in the K\"ahler cone of a generic $(X,D)$ because every ample class is K\"ahler. Thus we in fact have an equality $C_{(X,D)}=\mathcal{A}_{\text{\rm{gen}}}$. Let $\mathcal{F}$ denote the \textsl{set} of classes of symplectic forms of almost toric fibrations $(X,D,\omega)\rightarrow B.$ Clearly we have $\mathcal{F}\subset\mathcal{C}_{(X,D)}$, but by Theorem \ref{anylambda}, we also have $\mathcal{A}_{\text{\rm{gen}}}\subset\mathcal{F}$. Thus, all three are equal: 

\begin{corollary}\label{fun} $\mathcal{F}=C_{(X,D)}=\mathcal{A}_{\text{\rm{gen}}}.$ \end{corollary} In particular, $\mathcal{F}$ is a convex cone. Furthermore, every symplectic anticanonical pair $(X,D,\omega)$ has some other symplectic form in the class $[\omega]$ which admits a Lagrangian torus fibration. By the main result of Li and Mak \cite{LiMak}, we conclude that when $X$ is rational, any symplectic anticanonical pair $(X,D)$ is symplectic-isotopic (through symplectic anticanonical pairs with cycle $D$) to an almost toric fibration.

\subsection{A conjecture on Type III degenerations}

Before we state  the conjecture, let us recall the main results from \cite{FriedmanScattone}. Let $X$ be a $K3$ surface. The analogue of the group $\Gamma$ of admissible isometries is the group $O^+(\Lambda_{K3})$ of integral isometries of spinor norm $1$ of the $K3$ lattice $\Lambda_{K3}$. Let  $\mathcal{X}\to \Delta$ be a (polarized or weakly K\"ahler) Type III degeneration of $K3$ surfaces. If $N$ is the logarithm of the monodromy of the degeneration $\mathcal{X}\to \Delta$ acting on $H^2(X)$, then the saturation of $\im N^2$ is of the form $\Z \gamma$, where $\gamma$ is a primitive element of $\Lambda_{K3}$, and there exists an element $\lambda\in \gamma^\perp=W_0$ such that, for all $x\in H^2(X)$,
$$N(x) = (x\cdot\gamma)\lambda -  (x\cdot \lambda)\gamma.$$
It is easy to see that $\lambda$ is well-defined mod $\gamma$, i.e.\ as an element of $\gamma^\perp/\Z \gamma \cong U^2 \oplus (-E_8)^2$. Furthermore $\lambda^2 >0$.
 The analogue of the cone $\mathcal{B}_{\text{\rm{gen}}}$ for an anticanonical pair $(Y,D)$ is an appropriate component $\mathcal{C}^+_{K3, \gamma}$ of the positive cone  in $(\gamma^\perp/\Z \gamma)\otimes _\Z\R$.   Given an element $\lambda \in (\gamma^\perp/\Z \gamma)\cap \mathcal{C}^+_{K3, \gamma}$, we can associate to $\lambda$ the pair of positive integers $(k, t)$, where $t = \lambda^2$ and $k$ is the largest positive integer such that $\lambda = k\lambda_0\bmod \Z \gamma$ for some $\lambda_0 \in \Lambda_{K3}$. It is easy to see that the pair $(k,t)$ is a complete set of invariants for the pair $(\gamma, \lambda)$, in the sense that, given two pairs $(\gamma, \lambda)$ and   $(\gamma', \lambda')$ as above, the pairs have the same associated pair $(k,t)$ $\iff$ there exists a $\psi \in O^+(\Lambda_{K3})$ such that $\psi(\gamma) = \gamma'$ and, via the induced isometry $\gamma^\perp/\Z \gamma\to (\gamma')^\perp/\Z \gamma'$, $\psi(\lambda) = \lambda'$. The main result of \cite{FriedmanScattone} is then:
 
 \begin{theorem}\label{mainthmFS} {\rm (i)} For every pair of positive integers $(k,t)$, there exists a Type III degeneration of $K3$ surfaces with invariants $(k,t)$.
 
 \smallskip
 \noindent {\rm (ii)} Given two Type III degenerations $\mathcal{X}\to \Delta$ and $\mathcal{X}'\to \Delta$ of $K3$ surfaces, $\mathcal{X}\to \Delta$ and $\mathcal{X}'\to \Delta$ have the same invariants $\iff$ there exists a locally trivial deformation from $X_0$ to a Type III $K3$ surface $X_0''$ over a smooth connected base, with all fibers $d$-semistable, and a sequence of Type I and II birational modifications $X_0'' \dasharrow X_0'$. 
 \end{theorem} 
 
Theorem~\ref{lambdas} is then the analogue for Type III anticanonical pairs of the existence part (i) 
of Theorem~\ref{mainthmFS}. It is natural to conjecture that there is an analogue  of (ii) 
of Theorem~\ref{mainthmFS} as well:
 
\begin{conjecture}\label{lambdaclassifies} Let $(\mathcal{Y}, \mathcal{D})\to \Delta$ and $(\mathcal{Y}', \mathcal{D}')\to \Delta$ be two Type III degenerations of anticanonical pairs with monodromy invariants $\lambda$ and $\lambda'$, respectively. Assume that $D$ and $D'$ have the same self-intersection sequence and choose compatible labelings on them. Then, there exists an admissible isometry \cite[Definition 5.7]{Friedman2} $\psi \colon H^2( Y_t; \Z) \to H^2(Y'_t; \Z)$, $t\neq 0$, such that $\psi(\lambda) = \lambda'$ $\iff$ there exists a locally trivial deformation from $( Y_0, D)$ to a Type III anticanonical pair $( Y''_0, D)$ over a smooth connected base, with all fibers $d$-semistable, and a sequence of Type I and II birational modifications $(Y''_0, D) \dasharrow (Y'_0, D)$. 
\end{conjecture}

In other words, the class $\lambda \bmod \Gamma(Y,D)$ is a complete invariant of the combinatorial type of the central fiber. Note that the existence of an admissible isometry implies that the pairs $( Y_t,D)$ and $( Y'_t, D)$ are deformation equivalent, by \cite[Theorem 5.13]{Friedman2}.

\begin{example} Let $D'$ be the cusp with self-intersection sequence $(-12, -2)$. Consider the possible triples $(Y,D, \lambda)$, where $D$ is the dual cusp, $\lambda \in \Lambda\cap \mathcal{B}_{\text{gen}}$, and $\lambda^2 = 2$. There are two different Type III fibers $Y_0 = V_0 \cup V_1 \cup V_2$ with $2$ triple points. Here $V_0$ is the Inoue surface, $V_1 =\F_6$ and the self-intersection sequence is $(10, -6)$, and $V_2$ has self-intersection sequence $(0,4)$. Thus $V_2$ can be either $\F _0/\F _2$ (and these are deformation equivalent) or $\F _1$ (see Figure \ref{twodeftypes}). These two Type III fibers deform into two different deformation types of anticanonical pair $(Y,D)$ with dual self-intersection sequence $(-4, -2, \dots, -2)$. The first is discussed in \cite[Example 4.4]{Friedman1.5}. Here $\Lambda$ is spanned by $G_1, G_2$ with $G_1^2=10$, $G_2^2=-2$, and $G_1\cdot G_2 = 0$. It follows that the $D_i$ span a primitive sublattice of $\Pic Y$. There is just one $\lambda\in \Lambda$ with $\lambda^2= 2$ in $\mathcal{B}_{\text{gen}}$. There are no roots (i.e.\ $R = \emptyset$), $\mathcal{B}_{\text{gen}}$ is an angular sector not meeting the boundary of $\mathcal{C}^+$ away from $0$, and $\Gamma(Y,D) =\{1\}$. This deformation type of $(Y,D)$ corresponds to the case where $V_2=\F_1$.

\begin{figure}
\begin{centering}
\includegraphics[width=2.3in]{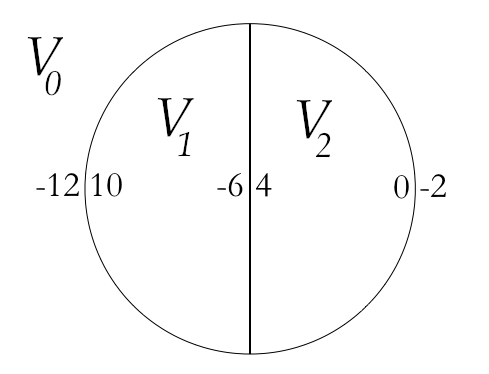} 
\caption{The combinatorial type of two Type III anticanonical pairs.}
\label{twodeftypes}
\end{centering}
\end{figure}

To find the second embedding of the dual $D$ in a rational surface, index the curves in the dual cycle by $D_0, \dots, D_9$ with $D_0^2=-4$, $D_i^2 = -2$, $i>0$, and $D_i$ meets $D_{i+1}$, $D_9$ meets $D_0$. Note generally that if there is a $-1$ curve meeting $D_1$, then contracting it and then $D_1, \dots, D_9$ leads to a surface $\overline{Y}$ and a nodal curve of square $8$ on $\overline{Y}$. Thus $\overline{Y}$ must be either $\F _0/\F _2$ or $\F _1$. The case of $\overline{Y}=\F _1$ is the case considered above. In case  $\overline{Y}=\F _0/\F _2$, there is obviously   a  root for $(Y,D)$ (the proper transform of the $-2$ curve on $\F _2$). Also, one checks that the $D_i$ span an index two sublattice of $\Pic Y$ and that $\Lambda$ is spanned by $G_1$, $G_2$ with $G_1^2 = G_2^2=-2$ and $G_1\cdot G_2 = 3$. One can show that in fact both $G_1$ and $G_2$ are  roots. It follows that there are infinitely many roots, and hence $\Gamma(Y,D)$ is infinite as well. The set $\mathcal{B}_{\text{gen}}$ is   equal to  $\mathcal{C}^+$. It is straightforward to check that $\Gamma = \mathsf{W}(R)$, the reflection group generated by the roots, and all of the $\lambda \in \Lambda\cap \mathcal{B}_{\text{gen}}$ with $\lambda^2 = 2$ are conjugate under $\Gamma$. This case then corresponds to the case where $V_2$ is either $\F_0$ or $\F_2$.  

To sum up, we see that within each deformation type, there is exactly one combinatorial possibility for $(Y_0, D)$ and exactly one element of square $2$ in $\mathcal{B}_{\text{\rm{gen}}}\cap \Lambda\mod \Gamma$. So Conjecture \ref{lambdaclassifies} is true in this case for classes of square $2$. \end{example}

\section{Asymptotic behavior of the period map}

Let $\pi\colon \mathcal{Y} \to \Delta$ be a Type III degeneration of anticanonical pairs. Fix a coordinate $t$ on $\Delta$ and let $\displaystyle z =\frac{\log t}{2\pi \sqrt{-1}}$ be the corresponding coordinate on the universal cover $\nu\colon \widetilde{\Delta}^*\to \Delta^*$. Thus $t = e^{2\pi \sqrt{-1} z}$ and $\displaystyle\im z = -\frac{\log |t|}{2\pi}$ is well-defined on $\Delta^*$. As usual, let $\mathcal{Y}^* = \mathcal{Y}\big{|}_{\Delta^*}$ and let $\rho\colon \mathcal{Y}^*-\mathcal{D} \to \Delta^*$ be the restriction of $\pi$. Then $R^2\rho_*\Z/(\text{torsion})$ is a local system with unipotent monodromy $T$, which we denote by $\mathcal{H}$.  Let $N =\log T$. Note that the pullback of $\mathcal{H}$ is trivialized by the choice of a point $t_0\in \Delta^*$:  $\nu^*\mathcal{H}$ is isomorphic to the constant sheaf with fiber $\overline{H}^2(Y_{t_0} -D; \Z)$.

The associated flat vector bundle $\mathcal{H}\otimes_\Z\scrO_{\Delta^*}$ has the canonical extension $\overline{\mathcal{H}}$ of Deligne \cite{Deligne}: $\overline{\mathcal{H}}$ can be taken to be the trivial holomorphic vector bundle  $\overline{H}^2(Y_{t_0} -D; \Z)\otimes _\Z\scrO_{\Delta}= H^2(Y_{t_0} -D; \C)\otimes _\C\scrO_{\Delta}$, with the connection $\displaystyle D = -\frac{N}{2\pi \sqrt{-1}}\frac{dt}{t}$. A (local) flat section over $\Delta^*$ is then a section locally of the form $e^{zN}v$, for $v\in H^2(Y_{t_0}-D; \C)$. In particular, we see that:

\begin{proposition}\label{rsp} Let $s$ be a  holomorphic section $s$ of $\mathcal{H}\otimes_\Z\scrO_{\Delta^*}$. Then $s$ extends to a holomorphic section of $\overline{\mathcal{H}}$ if and only if,  writing $\nu^*s$ as a holomorphic function on $\widetilde{\Delta}^*$ with values in $H^2(Y_{t_0} -D; \C)$, the function $e^{-zN}\nu^*s$, viewed as a function on $\Delta^*$,  extends to a (single-valued) holomorphic function on $\Delta$ with values in $H^2(Y_{t_0} -D; \C)$. \qed
\end{proposition}

On the other hand, by  \cite{Steenbrink} (see also \cite{Deligne}, \cite{PetersSteenbrink}), the canonical extension $\overline{\mathcal{H}}$ is also given by $\mathbb{R}^2\pi_*\Lambda^*_{\mathcal{Y}/\Delta}(\log \mathcal{D})$. By Corollary~\ref{subbundles}, the subsheaf $R^0\pi_*\omega_{\mathcal{Y}/\Delta}(\log \mathcal{D})$ is a rank one subbundle of $\mathbb{R}^2\pi_*\Lambda^*_{\mathcal{Y}/\Delta}(\log \mathcal{D})$. An easy argument shows that, for $\phi_t \in H^0(Y_t; \omega_{Y_t}(D_t))$, if $\phi_t\neq 0$, then $\displaystyle \int _\gamma \phi_t \neq 0$ (including the case $t=0$). We uniquely define an everywhere generating holomorphic section $\omega(t)$ of $R^0\pi_*\omega_{\mathcal{Y}/\Delta}(\log \mathcal{D})$ by the requirement that
$$\int_\gamma\omega(t) = 1 \text{ for all $t\in \Delta$ }.$$

Now fix an integral splitting of the exact sequence 
$$0 \to \Lambda\spcheck \to \overline{H}^2(Y_{t_0}-D;\Z) \to \Z \to 0,$$
so that every element of $\overline{H}^2(Y_{t_0}-D;\C)$ is uniquely written as $a\hat\gamma + \beta$, where $\hat\gamma \in \overline{H}^2(Y_{t_0}-D;\Z)$ maps to the (oriented) generator of $\Z$ and $\beta \in \Lambda\spcheck_\C =\Lambda_\C$. Note that, for $\overline{H}^2(Y_{t_0}-D;\Z)$, $N(\xi) =  \xi(\gamma) \lambda$, and in particular  $N(\hat\gamma) =  \lambda$, where we identify $\lambda$ with its Poincar\'e dual. With our choice of $\omega(t)$, we can write
$$\omega(t) = \hat\gamma + \beta(t),$$
where $\beta(t)$ is a multi-valued section of $\Lambda_\C$.
By Proposition~\ref{rsp}, the section  $e^{-zN}\omega(t)$ extends to a single-valued holomorphic function on $\Delta$ with values in $H^2(Y_{t_0} -D; \C)$. As 
$$e^{-zN}(\xi) = \xi -z\xi(\gamma) \lambda,$$
we see that 
$$e^{-zN}(\omega(t)) =  \hat\gamma + \beta(t) -z \lambda  =  f(t),$$
where $f(t)$ is a single valued holomorphic function on $\Delta^*$ with values in $H^2(Y_{t_0}-D; \C)$ which extends to a holomorphic function on $\Delta$.

Thus $\omega(t) = z\lambda + f(t)$, or equivalently:

\begin{proposition}\label{asympbehav} $\displaystyle \omega(t) = \frac{\log t}{2\pi \sqrt{-1}}\lambda + f(t)$, where $f(t) \in H^2(Y_{t_0}-D; \C)$ and $f$ is holomorphic at $0$. \qed
\end{proposition}

We then have the following corollary, which is a partial strengthening of \cite[III(2.12)]{Looij}:

\begin{corollary} The real, single-valued $C^\infty$ function $\im \omega(t)=\im\beta(t)$ with values in $\Lambda_\R$ satisfies: for $|t|\ll 1$, $\im\omega(t) \in \mathcal{C}^+$, and in fact $\im\omega(t)$ is in $\mathcal{B}_{\text{\rm{gen}}}$.
\end{corollary}
\begin{proof} We have
$$\im\omega(t) = -\frac{\log|t|}{2\pi}\lambda + g(t) =\left(-\frac{\log|t|}{2\pi}\right)\left(\lambda - \frac{2\pi g(t)}{\log|t|}\right),$$
where $g\colon \Delta \to H^2(Y_{t_0}-D; \R)$ is a $C^\infty$ function. The function $-\log|t|/2\pi$ is positive and tends to $+\infty$ when $t\to 0$, and hence  $2\pi g(t)/\log|t|$ extends to a continuous function on $\Delta$ with $\lim_{t\to 0}2\pi g(t)/\log|t|= 0$. Thus $(\im\omega(t))^2 > 0$ for all $|t|\ll 1$. With our sign conventions, it follows that $\im\omega(t) \in \mathcal{C}^+$. Moreover, $\lim _{t\to 0}(\lambda -  2\pi g(t)/\log|t|) =\lambda$. Since $\mathcal{B}_{\text{\rm{gen}}}$ is open in $\Lambda_\R$, for  $|t|\ll |t_0|$, $\lambda -  2\pi g(t)/\log|t|\in \mathcal{B}_{\text{\rm{gen}}}$. Thus, for $|t|\ll |t_0|$, $\im\omega(t)\in  \mathcal{B}_{\text{\rm{gen}}}$.
\end{proof}

\begin{remark} There is a several variable version of the above. The poldisk $\Delta^k$ has $k$ boundary components, defined by $t_i =0$, where $t_i$ is the coordinate on the $i^{\text{th}}$ factor. Given $\pi \colon \mathcal{Y} \to \Delta^k$, such that the generic fiber over each boundary component is  $d$-semistable, there is an associated monodromy transformation $T_i$ with $N_i =\log T_i$ and $N_i(x) = -(x\bullet \lambda_i)\gamma$. Then the period $\omega(t_1, \dots, t_k)$  has the form
$$\omega(t_1, \dots, t_k) = \sum_{i=1}^k\frac{\log t_i}{2\pi \sqrt{-1}}\lambda_i + f(t_1, \dots, t_k),$$
where $f\colon \Delta^k \to H^2(Y_{t_0}-D; \C)$ is holomorphic.  

By applying semistable reduction to a one parameter family over a disk $\Delta \subseteq \Delta^k$ which has order of contact $n_i$ with the $i^{\text{th}}$ component of the boundary, it follows that, for all $n_i\in \Z$, $n_i > 0$, we have  ($\sum_in_i\lambda_i)^2 > 0$. This easily implies that, for $i\neq j$, $\lambda_i\cdot \lambda_j > 0$, and hence either   $\lambda_i\in \mathcal{C}^+$ for all $i$ or  $-\lambda_i\in \mathcal{C}^+$ for all $i$. Choosing orientations so that $\lambda_i\in \mathcal{C}^+$ for all $i$, it follows that, if $|t_i| \ll 1$ for every $i$, then $\im\omega(t_1, \dots, t_k)\in\mathcal{B}_{\text{\rm{gen}}}$.
\end{remark} 

\section{The differential of the period map} 

We begin by discussing the smooth case, using \cite{Friedman2} as a general reference.  Let $(Y,D)$ be an anticanonical pair (with a fixed orientation of $D$). For notational simplicity, we assume that each component of $D$ is smooth, i.e.\ that $r >1$; minor modifications handle the case $r=1$. The \textsl{period homomorphism} is the homomorphism $\varphi_Y\colon \Lambda \to \C^*$ defined by: if $\alpha \in \Lambda$ and $L_\alpha$ is the corresponding line bundle, then $\varphi_Y(\alpha)=L_\alpha\big{|}_D\in \Pic^0D\cong \C^*$, and the \textsl{period map} is the map $Y\mapsto \varphi_Y\in \Hom (\Lambda, \C^*)$. The period map is holomorphic: for a fixed $\alpha \in \Lambda$, given a family $(\mathcal{F}, \mathcal{D})$ of pairs over $S$, after shrinking $S$ we can assume that $L_\alpha$ extends to a holomorphic line bundle $\mathcal{L}_\alpha$ over $\mathcal{F}$ and that $\mathcal{D}\cong S\times D$. The line bundle $\mathcal{L}_\alpha\big{|}_\mathcal{D} \cong S\times D$ then defines a holomorphic map from $S$ to $\Pic^0D$, and fitting these together for each $\alpha$ defines the period map as a holomorphic map $S\to \Hom (\Lambda, \C^*)$. The relationship between this period map and that considered in the previous section is as follows \cite[(3.12)]{Friedman2}: 
$$\varphi_Y(\alpha) = \exp\left(2\pi \sqrt{-1}\int_\alpha\omega\right).$$

The differential of the period map (for the semi-universal family) is then a map 
\begin{align*} \psi\colon H^1(Y; T_Y(-\log D))\to &\Hom(\Lambda, H^1(D; \scrO_D))= \\ &\Hom_\C(\Lambda_\C, H^1(D; \scrO_D)).\end{align*}
 We would like to describe the map $\psi$. A local calculation shows:

\begin{lemma}\label{locexactseq} Let $\nu\colon \widetilde D=\coprod_iD_i \to Y$ be the composition of normalization and inclusion.  Then there is an exact sequence
$$0 \to \Omega^1_Y(\log D)(-D) \to \Omega^1_Y \to \nu_*\Omega^1_{\widetilde{D}} \to 0.$$ 
\end{lemma}

Given a line bundle $L$ on $Y$ such that $\deg (L\big{|}_{D_i}) = 0$ for every $i$, its Chern class $c_1(L)\in H^1(Y;\Omega^1_Y)$ and its image in $\bigoplus _iH^1(D_i; \Omega^1_{D_i})$ is zero. Thus, from the exact sequence
\begin{align*} 0=&\textstyle \bigoplus _iH^0(D_i; \Omega^1_{D_i}) \to H^1(Y;\Omega^1_Y(\log D)(-D)) \to H^1(Y; \Omega^1_Y) \to \\  & \textstyle  \bigoplus _iH^1(D_i; \Omega^1_{D_i}),\end{align*} 
we see that $c_1(L)$ lifts to a unique element of $H^1(Y;\Omega^1_Y(\log D)(-D))$, which we denote by $\hat{c}_1(L)$.  Note that $H^1(Y;\Omega^1_Y(\log D)(-D)) \cong \Lambda_\C$ and we can think of $\hat{c}_1(L)$ as a Chern class with values in $\Lambda_\C$. The differential $\psi$ of the period map is then described as follows:

\begin{theorem}\label{diffcomp} Let $\partial \colon H^1(D; \scrO_D) \to H^2(Y; \scrO_Y(-D))$ be the coboundary map arising from the exact sequence
$$0 \to \scrO_Y(-D) \to \scrO_Y \to \scrO_D \to 0,$$
which is an isomorphism since $H^1(Y; \scrO_Y) =  H^2(Y; \scrO_Y) = 0$.  Then
$$\partial \circ \psi (\theta)(\alpha) = \theta \smallsmile \hat{c}_1(L_\alpha),$$
where $\smallsmile$ denotes the cup product 
$$H^1(Y; T_Y(-\log D))\otimes H^1(Y;\Omega^1_Y(\log D)(-D)) \to H^2(Y; \scrO_Y(-D)).$$
\end{theorem}

\begin{corollary} $\psi$ is an isomorphism.
\end{corollary}
\begin{proof} By Theorem~\ref{diffcomp}, $\partial \circ \psi$ is the cup product map
$$H^1(Y; T_Y(-\log D))\to \Hom (H^1(Y;\Omega^1_Y(\log D)(-D)), H^2(Y; \scrO_Y(-D))).$$
This map is an isomorphism since $\scrO_Y(-D) = \omega_D$ and by Serre duality.
\end{proof}

\begin{remark} The vector space $H^1(Y;\Omega^1_Y(\log D)(-D))$ is dual to $H^1(Y;\Omega^1_Y(\log D)) =\Lambda_\C\spcheck$. Furthermore, $\Hom_\C(\Lambda_\C, H^2(Y; \scrO_Y(-D)))$ is canonically isomorphic to $\Hom_\C(H^0(\Omega^2_Y(\log D)), H^1(Y;\Omega^1_Y(\log D)))$. This gives the usual form of the differential of the period map, via cup product and contraction:
$$H^1(Y; T_Y(-\log D))\otimes H^0(Y;\Omega^2_Y(\log D))\to H^1(Y;\Omega^1_Y(\log D)).$$
\end{remark}

\begin{proof}[Proof of Theorem~\ref{diffcomp}] The class $\theta\in H^1(Y; T_Y(-\log D))$ corresponds to a first order deformation $(\mathcal{Z}, \mathcal{D}) \to \Spec \C[\varepsilon]$. Explicitly, the correspondence is as follows: For a sufficiently small open cover $\{U_i\}$ of $Y$, there exists an isomorphism $\phi_i\colon \scrO_{\mathcal{Z}}\big{|}_{U_i} \cong \scrO_{U_i}\oplus \scrO_{U_i}\cdot \varepsilon$. Moreover, on $U_i\cap U_j$, $\phi_j\circ \phi_i^{-1} -\Id = \theta_{ij}$, where $\theta_{ij}$ is a $1$-cocycle. The condition that we keep the normal crossings divisor $D$ means that, in local coordinates, 
$$\theta_{ij} = az_1\frac{\partial}{\partial z_1} + bz_2\frac{\partial}{\partial z_2}$$
for some holomorphic functions $a$, $b$, i.e.\ that $\theta_{ij}$ is a section of $T_Y(-\log D)$ over $U_i\cap U_j$.  However, in our situation we can do a little better:

\begin{lemma} After possibly shrinking the cover $\{U_i\}$ and modifying $\{\theta_{ij}\}$ by a coboundary, we can assume that $\theta_{ij}\big{|}_D = 0$, i.e.\ that $\theta_{ij}$ has the local form $a'z_1z_2\frac{\partial}{\partial z_1} + b'z_1z_2\frac{\partial}{\partial z_2}$ for some holomorphic functions $a'$, $b'$.
\end{lemma}
\begin{proof} We have a commutative diagram
$$\xymatrix{
0\, \ar[r] & T_Y(-\log D)  \ar[r]  \ar[d] & \,\,T_Y\,\, \ar[r] \ar[d]  &  \bigoplus_i N_{D_i/Y} \ar[r] \ar[d] & \, 0\\
0\, \ar[r]&  \,\,T_D\,\, \ar[r] & \,T_Y\big{|}_D\, \ar[r] &  \,\,N_{D/Y}\,\, \ar[r] & \,\,T^1_D \,\,\ar[r] & \, 0} $$

A diagram chase shows that  $T_Y(-\log D)\to T_D$ is surjective. Hence there is an exact sequence
$$0 \to T_Y(-D) \to T_Y(-\log D) \to T_D \to 0.$$
But an easy calculation gives $H^1(D; T_D) =0$ (i.e.\  every locally trivial deformation of $D$ is a product). Thus the natural map $H^1(Y; T_Y(-D)) \to H^1(Y; T_Y(-\log D))$ is surjective, which is the statement of the lemma. 
\end{proof}

Returning to the proof of Theorem~\ref{diffcomp}, fix once and for all a representative $\{\theta_{ij}\}$ for $\theta$ satisfying the conclusion of the lemma. We have the usual exact sequence 
$$0\to \scrO_Y \to \scrO_{\mathcal{Z}}^* \to \scrO_Y^* \to 0,$$
and hence  $\Pic \mathcal{Z} \cong \Pic Y$ since $H^1(Y; \scrO_Y) = H^2(Y; \scrO_Y) =0$. Hence, given a line bundle $L$ on $Y$, $L$ lifts uniquely to a line bundle $\mathcal{L}$ on $\mathcal{Z}$, which we see explicitly as follows: Every lifting $\tilde{\lambda}_{ij}$ of $\lambda_{ij}$ to an element of $\scrO_{\mathcal{Z}}^*|U_i$ is necessarily of the form $\phi_i^{-1}(\lambda_{ij} + \mu_{ij}\varepsilon)$ for some holomorphic functions $\mu_{ij}$. The condition that $\tilde{\lambda}_{ij}$ can be chosen to be  a $1$-cocycle is the equality
$$\frac{\theta_{ji}(\lambda_{jk})}{\lambda_{jk}} = \lambda_{ij}^{-1}\mu_{ij} -\lambda_{jk}^{-1}\mu_{jk} + \lambda_{ik}^{-1}\mu_{ik} =\delta(\{\rho_{ij}\}),$$
where $\rho_{ij}$ is the $1$-cochain $\lambda_{ij}^{-1}\mu_{ij}$ and $\delta$ here and for the rest of the proof denotes \v{C}ech coboundary. In other words, the $2$-cocycle $\theta\smallsmile d\log \lambda_{ij} = \theta\smallsmile c_1(L)$ is zero in $H^2(Y;\scrO_Y)$. In this case, the transition functions 
$$\tilde{\lambda}_{ij}\phi_i^{-1}(\lambda_{ij}(1+  \rho_{ij}\varepsilon))$$ define the lift of $L$ to $\mathcal{L}$. Moreover, $\rho_{ij}\big{|}_D$ is then a $1$-cocycle since 
$$\delta(\{\rho_{ij}\}\big{|}_D)= \frac{\theta_{ji}(\lambda_{jk})}{\lambda_{jk}}\big{|}_D = 0,$$
and, in case $L=L_\alpha$ for $\alpha \in \Lambda$,  then $\rho_{ij}\big{|}_D$ represents the element $\psi(\theta)(\alpha)\in H^1(D;\scrO_D)$.

We compute $\partial \circ \psi(\theta)(\alpha)$ by lifting to a $1$-cochain with values in $\scrO_Y$ and taking its \v{C}ech coboundary. Clearly, $\rho_{ij}$ is such a lift, and thus $\partial \circ \psi(\theta)(\alpha)$ is represented by the $2$-cocycle  $\theta_{ji}(\lambda_{jk})/\lambda_{jk}$. 

Finally we compute the term $\theta \smallsmile \hat{c}_1(L)$. We can write
$$\frac{d\lambda_{ij}}{\lambda_{ij}} = \eta_{ij} + (\delta \{\sigma_i\})_{ij},$$
where $\eta_{ij}$ is a $1$-cocycle with values in $\Omega^1_Y(\log D)(-D)$ representing $\hat{c}_1(L)$ and $\sigma_i$ is a $0$-cochain with values in $\Omega^1_Y$. Thus  $\theta \smallsmile \hat{c}_1(L)$ is represented by the  $1$-cocycle
$$\frac{\theta_{ji}(\lambda_{jk})}{\lambda_{jk}} - \theta_{ji}(\delta \{\sigma_i\})_{jk}.$$
But, since $\delta \theta =0$, 
$$\theta_{ji}(\delta \{\sigma_i\})_{jk} = \theta\smallsmile \delta \{\sigma_i\} =\pm \delta (\theta\smallsmile \{\sigma_i\}).$$
Since $\theta_{ij}$ vanishes on $D$, $\delta (\theta\smallsmile \{\sigma_i\})$ is a $1$-coboundary with coefficients in $\scrO_Y(-D)$. Thus, in cohomology, $\partial \circ \psi (\theta)(\alpha) = \theta \smallsmile \hat{c}_1(L_\alpha)$ as claimed.
\end{proof}

We turn now to the case of a Type III degeneration $Y_0$. 
We define $\Lambda_0$ and $\Lambda_0\spcheck$ by analogy with the smooth case:
\begin{align*} 
\Lambda_0 &=\{L\in  \Pic Y_0: \deg (L\big{|}_{D_i}) =0 \text{ for all $i$}\};\\
\Lambda_0\spcheck &= \textstyle \Pic Y_0/\bigoplus _i\Z[D_i].
\end{align*}

Note that, by Lemma~\ref{Lemma1.9}, 
$$(\Lambda_0)_\C\spcheck  \cong H^1(Y_0; \Omega_{Y_0}^1(\log D)/\tau_{Y_0}^1).$$

\begin{lemma} $(\Lambda_0)_\C  \cong H^1(Y_0; \Omega_{Y_0}^1(\log D)/\tau_{Y_0}^1(-D))$.
\end{lemma}
\begin{proof} We have the analogue of the exact sequence of Lemma~\ref{locexactseq}:
$$0 \to \Omega^1_{Y_0}(\log D)/\tau_{Y_0}^1(-D) \to \Omega_{Y_0}^1/\tau_{Y_0}^1 \to \nu_*\Omega^1_{\widetilde{D}} \to 0.$$
Thus there is an exact sequence
$$\textstyle 0\to H^1(Y_0; \Omega_{Y_0}^1(\log D)/\tau_{Y_0}^1(-D)) \to H^1(Y_0; \Omega_{Y_0}^1/\tau_{Y_0}^1) \to \bigoplus_iH^1(D_i; \Omega^1_{D_i}).$$
This identifies $H^1(Y_0; \Omega_{Y_0}^1(\log D)/\tau_{Y_0}^1(-D))$ with the kernel of 
$$\textstyle H^1(Y_0; \Omega_{Y_0}^1/\tau_{Y_0}^1) \cong H^2(Y_0;\C) \to \bigoplus_iH^1(D_i; \Omega^1_{D_i}) \cong \bigoplus _iH^2(D_i; \C),$$
which is $(\Lambda_0)_\C$.
\end{proof}

There is a period homomorphism $\varphi_{Y_0}\in \Hom(\Lambda_0, \C^*)$ and it extends to a holomorphic map on germs of deformations of $Y_0$. In particular, there is a well-defined differential 
$$\psi\colon  \mathbb{T}^1_{Y_0}(-\log D) \to \Hom_\C((\Lambda_0)_\C, H^1(\scrO_D)).$$
To compute it, note first that the analogue of Lemma~\ref{locexactseq} holds, so that given $L\in \Lambda_0$, we have $\hat{c}_1(L) \in H^1(Y_0;\Omega^1_{Y_0}(\log D)(-D))$ and its image $\bar{c}_1(L) \in H^1(Y_0;\Omega^1_{Y_0}(\log D)/\tau_{Y_0}^1(-D))$.  Then we have

\begin{theorem}\label{diffcomp2} Let $\partial \colon H^1(D; \scrO_D) \to H^2(Y_0; \scrO_{Y_0}(-D))$ be the coboundary map arising from the exact sequence
$$0 \to \scrO_{Y_0}(-D) \to \scrO_{Y_0} \to \scrO_D \to 0,$$
which is an isomorphism since $H^1(Y_0; \scrO_{Y_0}) =  H^2(Y_0; \scrO_{Y_0}) = 0$.  Then
$$\partial \circ \psi (\theta)(L) = \langle \theta ,  \hat{c}_1(L)\rangle,$$
where $\langle \cdot, \cdot\rangle$ denotes Yoneda pairing 
$$\operatorname{Ext}^1(\Omega^1_{Y_0}(\log D), \scrO_{Y_0})\otimes H^1(Y_0;\Omega^1_{Y_0}(\log D)(-D)) \to H^2(Y_0; \scrO_{Y_0}(-D)).$$
\end{theorem}
\begin{proof} On the hyperplane $H^1(Y_0; T_{Y_0}(-\log D))$ of $\mathbb{T}^1_{Y_0}(-\log D)$, the Yoneda product is just cup product, and the proof is exactly the same as the proof for Theorem~\ref{diffcomp}. By linearity, it  suffices to prove the formula of Theorem~\ref{diffcomp2} for a single class $\theta$ projecting onto a nonzero element of $H^0(Y_0; T^1_{Y_0})\cong \C$. Fix a one parameter smoothing $\pi\colon (\mathcal{Y},\mathcal{D})\to \Delta$. There are relative versions $\mathbb{T}^1_{\mathcal{Y}/\Delta}(-\log \mathcal{D})$, $R^1\pi_*\Omega^1_{\mathcal{Y}/\Delta}(\log \mathcal{D})(-\mathcal{D})$, $R^1\pi_*\scrO_{\mathcal{D}}$, $R^2\pi_*\scrO_{\mathcal{Y}}(-\mathcal{D})$. We also have the relative Kodaira-Spencer map $\kappa \colon T_\Delta \to \mathbb{T}^1_{\mathcal{Y}/\Delta}(-\log \mathcal{D})$, and $\kappa(\partial/\partial t)$ is a section which restricts to an element $\theta$ with nonzero image in  $H^0(Y_0; T^1_{Y_0})$. Given $L\in \Pic Y_0$, we can lift $L$ to a line bundle $\mathcal{L}\in \Pic \mathcal{Y}$, and there is a relative (modified) Chern class $\hat{c}_1(\mathcal{L})\in R^1\pi_*\Omega^1_{\mathcal{Y}/\Delta}(\log \mathcal{D})(-\mathcal{D})$. We want to establish the equality
$$\partial \circ \psi (\kappa(\partial/\partial t))(\mathcal{L}) = \langle \kappa(\partial/\partial t),  \hat{c}_1(\mathcal{L})\rangle.$$
This holds for a general $t$ by Theorem~\ref{diffcomp}, and thus follows by continuity for $t=0$, noting that $R^2\pi_*\scrO_{\mathcal{Y}}(-\mathcal{D})\cong \scrO_{\Delta}$ is torsion free (see also Proposition~\ref{extpermapholom} below). Restricting to $t=0$ gives the result for $\theta$.
\end{proof}

\begin{remark} We will only need this result for $H^1(Y_0; T_{Y_0}(-\log D))$, the hyperplane  tangent to the equisingular deformations. 
\end{remark}

\begin{corollary}\label{diffcomp2cor} The restriction of the differential of the period map to $H^1(Y_0; T_{Y_0}(-\log D))$ is an isomorphism and the differential of the period map on $\mathbb{T}^1_{Y_0}(-\log D)$ is surjective.
\end{corollary}
\begin{proof} Note that $T_{Y_0} =Hom(\Omega^1_{Y_0}, \scrO_{Y_0}) = Hom(\Omega^1_{Y_0}/\tau^1_{Y_0}, \scrO_{Y_0})$, and similarly for $T_{Y_0}(-\log D)$. Thus there is the cup product map
$$H^1(Y_0 ; T_{Y_0}(-\log D)) \otimes H^1(Y_0;\Omega^1_{Y_0}(\log D)/\tau^1_{Y_0}(-D)) \to H^2(Y_0; \scrO_{Y_0}(-D)),$$
which is identified via $\partial^{-1}$ with 
$$H^1(Y_0 ; T_{Y_0}(-\log D)) \otimes (\Lambda_0)_\C \to H^1(D; \scrO_{D}),$$
the differential of the period map. By \cite[(2.10)]{Friedman1}, the vector spaces
 $H^1(Y_0;\Omega^1_{Y_0}(\log D)/\tau^1_{Y_0}(-D))$ and $H^1(Y_0 ; T_{Y_0}(-\log D))$ are dual under cup product, hence
  the above cup product  is nondegenerate by Serre duality. Moreover by compatibility of cup product,
  $$\langle \theta , \hat{c}_1(L)\rangle = \langle \theta , \bar{c}_1(L)\rangle,$$
where $\bar{c}_1(L) \in H^1(Y_0;\Omega^1_{Y_0}(\log D)/\tau_{Y_0}^1(-D))$ is the image of $\hat{c}_1(L) \in H^1(Y_0;\Omega^1_{Y_0}(\log D)(-D))$.  This proves the first statement and the second then follows.
\end{proof}

\section{Surjectivity of the period map}

In this section, our goal is to use Theorem~\ref{diffcomp2} to prove results about the surjectivity of the period map for families of $d$-semistable  Type III anticanonical pairs, and then to apply this to construct Type III degenerations $\mathcal{Y}\to \Delta$ such that the kernel of the period map on a general fiber $Y_t$ is of a special type. 

Let $(Y_0, D_0)$ be a Type III anticanonical pair as defined in Definition~\ref{defTypeIII}. If $Y_0$ is the central fiber of a Type III degeneration $\pi\colon \mathcal{Y}\to \Delta$ of anticanonical pairs, which is the case if and only if $Y_0$ is $d$-semistable, then there are $f$ line bundles on $Y_0$, namely $\scrO_{\mathcal{Y}}(V_i)\big{|}_{Y_0}$, $i=0, \dots, f-1$.  By Lemma~\ref{notorsinH1},  
$$\Pic Y_0 = \textstyle H^2(Y_0; \Z) \cong \Ker \left(\bigoplus_iH^2(V_i;\Z) \to \bigoplus_{i< j}H^2(C_{ij}; \Z)\right).$$
Define a class $\xi_\ell\in \bigoplus _iH^2(V_i;\Z)$ by the rule:
$$\xi_\ell \big{|}_{V_i} = \begin{cases} [C_{i\ell}], &\text{if $i\neq \ell$;}\\
-\sum_{j\neq \ell}[C_{j\ell}]= -[C_\ell], &\text{if $i= \ell$.}
\end{cases}$$
By the triple point formula, $\xi_\ell \in \Ker  (\bigoplus_iH^2(V_i;\Z) \to \bigoplus_{i< j}H^2(C_{ij}; \Z) )$, and hence $\xi_0, \dots, \xi_{f-1}\in H^2(Y_0;\Z)\cong \Pic Y_0$. It is easy to check that $\sum_i\xi_i =0$. Clearly $\xi_i \cdot [D_j] =0$ for every component $D_j$ of $D$.

\begin{lemma}\label{cokeristorsionfree} The homomorphism $\psi\colon \Z^f \to H^2(Y_0;\Z)$ defined by 
$$(n_0, \dots, n_{f-1}) \mapsto \sum_in_i\xi_i$$ has kernel $\{(n, \dots, n)\in \Z^f: n\in \Z\}$ and thus defines an injective homomorphism $\Z^f/\Z \to H^2(Y_0;\Z)$. Moreover, the cokernel of $\psi$ is torsion free. 
\end{lemma}
\begin{proof} Suppose that $(n_0,\dots,n_{f-1})$ is in the kernel. After adding the element $(-n_0, \dots, -n_0)$, we may assume that $n_0 = 0$ and must show that $n_i=0$ for all $i$. We have $\sum_{V_i\cap V_0 \neq \emptyset}n_i[C_{0i}] =0$ in $H^2(V_0; \Z)$. As the classes $[C_{0i}]$ are linearly independent in $H^2(V_0; \Z)$,   $n_i = 0$ for all $i$ such that $V_i\cap V_0 \neq \emptyset$. Fix one such component $V_i$. Then there exists a component  $V_j$ such that $V_0\cap V_i\cap V_j\neq \emptyset$, corresponding to a  triple point  of $Y_0$ lying on $C_{0i}$. By assumption, $\sum_{V_i\cap V_\ell \neq \emptyset}n_\ell  [C_{i\ell}] =0$ in $H^2(V_i; \Z)$, and $n_0 = n_i = n_j=0$. An easy argument (cf.\ \cite[2.8]{Friedman2}) shows that, if $(Y,D)$ is an anticanonical pair and $D_1, \dots, D_{r-2}$ is a chain of length $r-2$, i.e.\ $D_i\cdot D_{i+1}=1$, $1\leq i\leq r-3$, then the classes $[D_1], \dots, [D_{r-2}]$ are linearly independent in $H^2(Y; \Z)$. Applying this fact to the pair $(V_i, C_i)$, it follows that $n_\ell =0$ for all $\ell$ such that $V_\ell\cap V_i \neq \emptyset$. Continuing in this way and using the fact that the dual complex of $Y_0$ is connected, a straightforward argument then shows that $n_i=0$ for all $i$.

Finally, we must show that $\Coker \psi$ is torsion free. This statement is invariant under a locally trivial deformation, so we can assume that $Y_0$ is $d$-semistable by \cite[(2.16)]{FriedmanMiranda}, and that $\pi\colon \mathcal{Y}\to \Delta$ is a smoothing. Thus $\xi_i =  \scrO_{\mathcal{Y}}(V_i)\big{|}_{Y_0}$ as described above. The proof of (i) of Proposition~\ref{lambdaprop} shows that there is an inclusion of $\Coker \psi$ in $H^2(Y_t; \Z)$ for a general fiber $Y_t$. Hence $\Coker \psi$ is torsion free.
\end{proof}

 We next define the period map for $d$-semistable Type III anticanonical pairs:

\begin{definition} Let $\overline{\Lambda}_0$ be the  lattice  defined by
$$\overline{\Lambda}_0 = \{\alpha \in \Pic Y_0/\operatorname{span}\{\xi_0, \dots, \xi_{f-1}\}: \alpha \cdot [D_i] =0 \text{ for all $i$}\}.$$
Note that the condition $\alpha \cdot [D_i] =0$ is well-defined since $\xi_\ell \cdot [D_i]=0$ for all $\ell$ and $i$. By the proof of Lemma~\ref{cokeristorsionfree}, if $Y_0$ is $d$-semistable and $\pi\colon \mathcal{Y}\to \Delta$ is a smoothing, then there is an inclusion of $\overline{\Lambda}_0$ in $\Lambda= \Lambda(Y_t,D)$ for a general fiber $Y_t$. 

For  a Type III  anticanonical pair $(Y_0, D)$, with $D$ oriented, we have the period map $\varphi_{Y_0}\colon \Lambda_0 \to \C^*$ defined by $L\in \Pic Y_0 \mapsto L\big{|}_D$. If $Y_0$ is $d$-semistable, then clearly $\varphi_{Y_0}(\xi_i) = 1$ for every $i$. We define $\bar\varphi_{Y_0}\colon \overline{\Lambda}_0 \to \C^*$ to be the homomorphism induced by $\varphi_{Y_0}$. Given  a locally trivial deformation of the pair $(Y_0,D)$ with trivial monodromy, over a smooth connected base $S$, with all fibers $d$-semistable, there exists a holomorphic map $\operatorname{per}\colon S \to \Hom(\overline{\Lambda}_0, \C^*)$ in the usual way.
\end{definition}
 
 The following is then the main result of this section:
 
 \begin{theorem}\label{surj1} Let $\bar\varphi\colon \overline{\Lambda}_0 \to \C^*$ be a homomorphism.  Then there exists a locally trivial deformation of $Y_0$ with trivial monodromy, over a smooth connected base $S$ which we can take to be of the form $(\C^*)^N$, such that all fibers are $d$-semistable, and a point $s\in S$ such that, if $(Y_s,D)$ is the corresponding Type III anticanonical pair, then $\varphi_{Y_s}(\xi_i) = 1$ for all $i$ and the induced homomorphism $\overline{\Lambda}_0 \to \C^*$ is $\bar\varphi$.  
\end{theorem}
\begin{proof} There are three ways that we can deform $Y_0$ in a locally trivial way (cf.\ \cite[(4.3)]{FriedmanScattone} for more details):
 \begin{enumerate}
  \item We change the gluings $C_{ij}\subseteq V_i\cong C_{ji}\subseteq V_j$ by multiplying by $\lambda\in \C^*$.
 \item If $V_i$ is obtained by blowing up a point $p\in C_{ij}^{\text{int}}\cong \C^*$, we can deform the point $p$ to a point $p'\in C_{ij}^{\text{int}}$.  (Note that a blowup at a singular point of the anticanonical divisor has no moduli.) 
 \item In most cases, $V_i$ is obtained via iterated blowups from a minimal pair with no moduli (\textsl{taut}, in the terminology of \cite{Friedman2}). In a very small number of exceptional cases ($V_i = \F _0$ or $\F _2$ is already minimal and the double curve on $V_i$ is either irreducible nodal or has two components of self-intersection $2$), the pair $(V_i, C_i)$ has moduli $\cong \C^*$. 
  \end{enumerate}
There is then a ``universal"  deformation $(\mathcal{U}, \mathcal{D})$ of $Y_0$, parametrized by $T= (\C^*)^M$, corresponding to the different possibilities (1), (2), (3) above. Universal in this sense simply means that every locally trivial deformation of $(Y_0, D)$ appears as a fiber. After a translation  on $T$, we can assume  that $1\in T$ corresponds to $Y_0$. Furthermore, the monodromy of the family is trivial. Finally, we can assume that there is a unique component of $\mathcal{U}$ isomorphic to $V_0\times T$. 

Let $\Theta = (Y_0)_{\text{sing}} = \bigcup_{i<j}C_{ij}$ be the double curve on $Y_0$. Then it is easy to check that $\Theta$ is rigid under locally trivial deformations and that $\Pic^0\Theta \cong (\C^*)^{f-1}$. Furthermore, if $\Aut^0\Theta $ denotes the neutral component of $\Aut \Theta$, then $\Aut^0\Theta \cong (\C^*)^e$.  Since $H^1_{\text{\'et}}(T, \mathbb{G}_m) = 0$, $H^1_{\text{\'et}}(T,\Aut^0\Theta) =0$ as well, and thus we can choose an isomorphism of schemes  
$$(\mathcal{U})_{\text{sing}} \cong \Theta \times T.$$

Since the monodromy of the family $(\mathcal{U}, \mathcal{D}) \to T$ is trivial, there is a period map 
$\operatorname{per}\colon T \to \Hom( \Lambda_0, \C^*)$.

\begin{lemma}\label{peraffine} Let $(\mathcal{U}, \mathcal{D}) \to T$ be the deformation constructed above. Then the  period map $\operatorname{per}\colon T \to \Hom(\Lambda_0, \C^*)$ is an affine map of tori.
\end{lemma}
\begin{proof} Very generally, if $\rho\colon (\C^*)^{N_1} \to (\C^*)^{N_2}$ is a morphism of algebraic varieties, then by a theorem of Rosenlicht $\rho$ is an affine map. We will give an essentially  direct argument that $\operatorname{per}$ is an affine map. Thus, given a line bundle $L$ on $Y_0$, we must analyze how $L|_D$ varies as we vary the gluings. 

 Let $D' = \bigcup_{i=1}^{r'}D_i'$ be the dual cycle to $D$ on $V_0$. As previously noted,  $\pi_1(D', *)\cong \pi_1(V_0, *) \cong \pi_1(D, *)$, and hence 
\begin{align*}
H^1(D'; \Z) &\cong H^1(V_0; \Z) \cong H^1(D ; \Z);\\
H^1(D'; \scrO_{D'}) &\cong H^1(V_0; \scrO_{V_0}) \cong H^1(D ; \scrO_D). 
\end{align*}
Thus, an orientation on $D$ induces a compatible orientation on $D'$, and  
$$\Pic^0D'\cong \Pic ^0D \cong \Pic^0V_0 =H^1(V_0; \scrO_{V_0}^*)/H^1(V_0; \Z).$$
In particular, if $L''$ is a line bundle on $V_0$ which restricts to a line bundle of multidegree $0$ on $D$, then the map $\lambda \in \Pic^0V_0 \mapsto (L''\otimes \lambda)\big{|}_D$ is an affine isomorphism from $\Pic^0V_0 \cong \C^*$ to $\Pic^0D$. 

Choose an ordering on the components $V_i$ of $Y_0$ so that $D_i' = V_i\cap V_0$, $1\leq i \leq r'$, i.e.\ $D_i' = C_{0i}$, where the ordering of the components of $D'$ is compatible with the orientation. Let $Y_0' =\bigcup_{i=1}^{f-1}V_i$. Then Mayer-Vietoris applied to the union $Y_0 = Y_0' \cup V_0$ implies that there is an exact sequence
$$\{1\} \to \Pic Y_0 \to \Pic Y_0' \times \Pic V_0 \to  \Pic D' \to \{1\},$$
i.e.\ a line bundle $L$ on $Y_0$ is determined by line bundles $L'$ on $Y_0'$ and $L''$ on $V_0$, such that $L''\big{|}_{D'} \cong \varphi^*(L'\big{|}_{D'})$, where $\varphi$ is the isomorphism $D'\subseteq  V_0\cong D'\subseteq Y_0'$ given by $Y_0$. Thus, if $s\in T$, and the gluing $\varphi$ changes by $\varphi\circ \psi(s)$, where $\psi(s)\in \Aut^0D'$ is given by a surjective homomorphism $\psi$ from $T$ to a subtorus of $\Aut^0D'$, then the line bundle $L$ deforms to some line bundle $L_s$ corresponding to pairs $L'_s $, the line bundle corresponding to $L'$ over the corresponding deformation of $Y_0'$ and $L_s''  = L''\otimes \lambda''(s)$, with $\lambda(s)\in \Pic^0 V_0$,  such that 
$$(L''\otimes \lambda''(s))\big{|}_{D'} \cong \psi(s)^*\varphi^*(L_s'\big{|}_{D'}).$$
If $\Aut^0D' \cong (\C^*)^{r'}$, where the $i^{\text{th}}$ factor operates in the usual way on the $i^{\text{th}}$ component $D_i'$ of $D'$, then $\Aut^0D'$ acts trivially on $\Pic^0D'$, and $(z_1, \dots, z_{r'})$ acts on $\Pic^{d_1, \dots, d_{r'}}D'$, the set of line bundles on $D'$ of multidegree $(d_1, \dots, d_{r'})$, via $z_1^{d_1}\cdots z_{r'}^{d_{r'}}$ and the usual action of $\Pic^0D'\cong \C^*$ on $\Pic^{d_1, \dots, d_{r'}}D'$. Given $L_s$ and $\psi(s)$, there is a unique line bundle $\lambda''(s) \in \Pic^0 V_0$ for which the compatibility condition holds. Thus, if we can show that $s\mapsto L_s'\big{|}_{D'}$ is an affine map, it follows that $s\mapsto \lambda''(s)$ is affine and hence so is $\operatorname{per}(s)(\alpha)=\lambda''(s)\big{|}_D$, where $\alpha\in \overline{\Lambda}_0$ is the class corresponding to the line bundle $L$.

 The line bundle $L'$ determines a multidegree $\mathbf{d} = (d_{ij})$, with $d_{ij} = \deg(L'\big{|}_{C_{ij}})$, and hence $d_i = d_{0i}$, $1\leq i\leq r'$. Restriction defines a  morphism  $\Pic^{\mathbf{d}}\Theta\to \Pic^{d_1, \dots, d_{r'}}D'$, and it is equivariant with respect to the actions of $\Pic^0\Theta\cong  (\C^*)^{f-1}$ on $\Pic^{\mathbf{d}}\Theta$ and $\Pic^0D'\cong \C^*$ on $\Pic^{d_1, \dots, d_{r'}}D'$ via the natural homomorphism $\Pic^0\Theta \to \Pic^0D'$. We now consider the various ways that the construction of Lemma~\ref{univdss} changes the line bundle $L'\big{|}_\Theta$:
 \begin{enumerate}
 \item Suppose that we change the gluings $C_{ij}\subseteq V_i\cong C_{ji}\subseteq V_j$ by $\lambda \in \Aut C_{ij}^{\text{int}}$. Here we may assume that $i,j\geq 1$ since we have already dealt with the gluings from $Y_0'$ to $V_0$. The line bundle $L'\big{|}_\Theta$ is then replaced by $\lambda^*(L'\big{|}_\Theta)$ for the automorphism $\lambda\in \Aut^0\Theta$ corresponding to multiplying by $\lambda\in \C^*$ on the component $C_{ij}$.
 \item If $V_i$ is obtained by blowing up a point $p\in C_{ij}^{\text{int}}\cong \C^*$, and we replace $p$ by $p'\in  C_{ij}^{\text{int}}$, then the line bundle $L'\big{|}_\Theta$ is then replaced by $L'\big{|}_\Theta\otimes \scrO_{\Theta}(a(p'-p))$ for an appropriate integer $a$.
\item  In this exceptional case, a direct if somewhat tedious calculation shows that the corresponding map $\C^* \to \Pic^{\mathbf{d}}\Theta$ is affine (cf.\ \cite[Lemma 4.6]{FriedmanScattone}). Alternatively, it is clear that the map $\C^* \to \Pic^{\mathbf{d}}\Theta$ is a morphism (i.e.\ algebraic) and hence it is affine by Rosenlicht's theorem.
\end{enumerate}

Combining the various possibilities, we see that we get an affine map from $S$ to $\Pic^{d_1, \dots, d_{r'}}D'$ as claimed.
\end{proof}

Consider the affine map $\rho \colon T \to (\C^*)^{f-1}$ defined by 
$$\rho(s) = (\varphi_{Y_s}(\xi_0), \dots, \varphi_{Y_s}(\xi_{f-1})),$$
where we view the image of $\rho$ as contained in the subtorus defined by: the product of the components is $1$. Let $S= (\Ker \rho)^0$ be the neutral component of $\Ker \rho$.

\begin{lemma}\label{univdss} Let $(\mathcal{U}_S, \mathcal{D}) \to S$ be the restriction of the family $(\mathcal{U}, \mathcal{D}) \to T$ to the subtorus $S$ of $T$. Then:
\begin{enumerate}
\item[\rm(i)] For all $s\in S$,  $Y_s$ is $d$-semistable.
\item[\rm(ii)] The image of the Kodaira-Spencer map $T_{S,0}\to H^1(Y_0; T_{Y_0}(-\log D))$ is
$$\{\theta \in H^1(Y_0; T_Y(-\log D)): \theta \smallsmile \hat{c}_1(\xi_0) = \cdots = \theta \smallsmile \hat{c}_1(\xi_{f-1}) = 0\}.$$
\end{enumerate}
\end{lemma}
\begin{proof}  
By arguments similar to the proof of Lemma~\ref{peraffine}, one checks that $\rho_1\colon T \to \Pic^0\Theta\cong (\C^*)^{f-1}$ defined by $t\mapsto T^1_{Y_t}=Ext^1(\Omega^1_{Y_t}, \scrO_{Y_t})$ is an affine map and hence a homomorphism. By definition, the locus of $d$-semistable deformations is $\rho_1^{-1}(1)$. 
Then both $\rho_1$ and $\rho$ are homomorphisms since they take the identity to the identity, and   $\rho_1^{-1}(1) \subseteq \rho^{-1}(1)$. The differential  of $\rho_1$ is surjective at $t=0$ by construction and \cite[(5.9) and (4.5)]{Friedman1}. Hence the codimension of $\Ker (\rho_1)_*$ is $f-1$. 
The differential of $\rho$ is also surjective by Corollary~\ref{diffcomp2cor}. Hence $\dim \Ker (\rho_1)_* = \dim \Ker (\rho)_*$. Thus the two subtori $(\Ker \rho_1)^0$ and $S=(\Ker \rho)^0$ have the same tangent space at $1$ and hence are equal. In particular, all fibers $Y_s$ are $d$-semistable. The final statement (ii) about the Kodaira-Spencer map then follows from Theorem~\ref{diffcomp2}.
\end{proof}

 To complete the proof of Theorem~\ref{surj1}, consider the family $(\mathcal{U}_S, \mathcal{D}) \to S$ described in the remarks before Lemma~\ref{univdss}. By construction, the family $(\mathcal{U}_S, \mathcal{D}) \to S$  induces a period map $\operatorname{per}_S\colon S \to \Hom(\overline{\Lambda}_0, \C^*)$. By Lemma~\ref{peraffine}, the map $\operatorname{per}_S$ is  an affine map. We claim that the differential of  $\operatorname{per}_S$ is surjective. By Lemma~\ref{univdss} and the remarks before it, we are reduced to showing the following: let 
$$\{\xi_0, \dots, \xi_{f-1}\}^\perp \subseteq H^1(Y_0; T_Y(-\log D))$$
be defined as $\{\theta \in H^1(Y_0; T_Y(-\log D)): \theta \smallsmile \hat{c}_1(\xi_0) = \cdots = \theta \smallsmile \hat{c}_1(\xi_{f-1}) = 0\}$. We can replace $\hat{c}_1(\xi_i)$ by $\bar{c}_1(\xi_i)$ in the above. Moreover, 
\begin{align*}
(H^1(Y_0;\Omega^1_{Y_0}(\log D)/\tau^1_{Y_0}(-D)))/&\operatorname{span}\{\bar{c}_1(\xi_0), \dots, \bar{c}_1(\xi_{f-1})\} \\ &\cong (\Lambda_0)_\C/\{\xi_0, \dots, \xi_{f-1}\} = (\overline{\Lambda}_0)_\C. 
\end{align*}
Then cup product induces a perfect pairing
$$\{\xi_0, \dots, \xi_{f-1}\}^\perp \otimes (\overline{\Lambda}_0)_\C \to H^2(Y_0; \scrO_{Y_0}(-D))\cong H^1(D; \scrO_D).$$
This says that the differential of $\operatorname{per}_S$ is surjective.  
Hence $\operatorname{per}_S$ is an affine map of tori with surjective differential. It follows that  $\operatorname{per}_S$ is surjective. This is exactly the statement of Theorem~\ref{surj1}.
\end{proof}

\begin{remark} Related methods give a somewhat more direct proof of Theorem 4.1 in \cite{FriedmanScattone}.
\end{remark}

We can relate the period map for a $d$-semistable Type III anticanonical pair to the period map for nearby smoothings as follows. Suppose that $\pi\colon (\mathcal{Y}, \mathcal{D})\to \Delta^k$ is a morphism from the smooth complex manifold $\mathcal{Y}$ to the polydisk, whose fibers are smooth anticanonical pairs for $t\in \Delta^{k-1}\times \Delta^*$ and such that the fibers over $\Delta^{k-1}\times\{0\}$ are $d$-semistable Type III anticanonical pairs. We will apply this in the case of the germ of the smoothing component of a $d$-semistable Type III anticanonical pair below and for one parameter families (the case $k=1$) in Section 9. In this case, the monodromy on $\Lambda= \Lambda(Y_t, D_t)$ is trivial, there is a well-defined monodromy invariant $\lambda\in \Lambda$, and there is an inclusion $\overline{\Lambda}_0 \to \Lambda$ identifying the image with $\lambda^\perp$.

\begin{proposition}\label{extpermapholom} Let $\rho\colon \Hom(\Lambda, \C^*) \to \Hom (\overline{\Lambda}_0, \C^*)$ denote restriction. The function $\Phi \colon \Delta^k \to \Hom (\overline{\Lambda}_0, \C^*)$ defined by 
$$\Phi(t) =\begin{cases}    \rho\circ \varphi_{Y_t},  &\text{if $t\in \Delta^{k-1}\times \Delta^*$}; \\
   \bar\varphi_{Y_t},  &\text{if $t\in \Delta^{k-1}\times \{0\}$}
\end{cases}$$
is   holomorphic.
\end{proposition}
\begin{proof} For a fixed $\alpha \in \overline{\Lambda}_0$, there is a corresponding line bundle $\mathcal{L}_\alpha$ on $\mathcal{Y}$, well defined up to twisting by components of the inverse image of $\Delta^{k-1}\times \{0\}^*$. Restricting $\mathcal{L}_\alpha$ to $\mathcal{D}$ and choosing a trivialization $\mathcal{D} \cong D \times \Delta^k$ (which is always possible locally) gives a holomorphic map $\Delta^k \to \Pic ^0D\cong \C^*$. By definition, this map is independent of the choice of $\mathcal{L}_\alpha$ and defines $\Phi$, evaluated on the class $\alpha$. Doing this construction for a basis of $\overline{\Lambda}_0$ shows that $\Phi$ is holomorphic.  
\end{proof}

\begin{remark}  Using the alternative formulation of the period map from the preceding section and Proposition~\ref{asympbehav} in the parametrized case, we see that, if $t_k$ is a coordinate defining the last factor of $\Delta^k$ in the notation above, and $\alpha$ is a class in $\Lambda$, then, for $t\in \Delta^{k-1}\times \Delta^*$, 
$$\varphi_{Y_t}(\alpha) = t_k^{(\alpha \cdot \lambda)}h(t),$$
where $h(t)$ extends to a  holomorphic, nowhere zero  function on $\Delta^k$. Hence, $\varphi_{Y_t}(\alpha)$ extends to a holomorphic function on $\Delta^k$ with values in $\C^*$ $\iff$ $\alpha \in \overline{\Lambda}_0$.
\end{remark}

We will then need the following, which is in a sense complementary to Theorem~\ref{surj1}:
 
\begin{proposition}\label{surj2} Suppose that $(Y_0, D)$ is a $d$-semistable Type III anticanonical pair. Let $\Upsilon$ be a sublattice of $\overline{\Lambda}_0$ such that $\bar\varphi_{Y_0}(L) = 1$ for all $L \in \Upsilon$. Then there exists a Type III degeneration $\mathcal{Y} \to \Delta$, with special fiber $Y_0$ such that, via the specialization map  
$$\operatorname{sp} \colon \Pic Y_0 \cong H^2(Y_0;\Z)\to H^2(Y_t;\Z),$$
 $\varphi_{Y_t}(\operatorname{sp}(L)) =1$ for all $t\in \Delta^*$. 
\end{proposition}

\begin{proof} Let $(\mathcal{V}, \mathcal{D})\to (X,0)$ be the germ of the subspace of the semiuniversal deformation of the pair $(Y_0, D)$ corresponding to the smoothing component and, on the  discriminant locus, to preserving the condition of $d$-semistability. Note that $\Pic \mathcal{V} \cong \Lambda_0$ and that the classes $\xi_0, \dots, \xi_{f-1}$ extend to Cartier divisors on $\mathcal{V}$ with trivial restriction to $\mathcal{D}$. By Proposition~\ref{extpermapholom}, there is a well-defined, holomorphic  period map (on the level of germs) $\Phi\colon X \to   \Hom(\overline{\Lambda}_0, \C^*)$. Via restriction, there is a surjection
$$\Hom(\overline{\Lambda}_0, \C^*) \to \Hom(\Upsilon, \C^*),$$
and in particular the differential of the surjection is also surjective.  Taking the composition with the period map, there is a (germ of a) holomorphic map $\Phi_\Upsilon\colon X \to \Hom(\Upsilon, \C^*)$. Theorem~\ref{diffcomp2}  and Corollary~\ref{diffcomp2cor} imply that $\Phi_\Upsilon$ is a submersion at $0\in X$ and that $\Phi_\Upsilon^{-1}(1)$ is transverse to the discriminant hypersurface. Taking a generic map of the disk $(\Delta, 0)$ to $\Phi_\Upsilon^{-1}(1) \subseteq (X,0)$ (and which is in particular transverse to the discriminant) gives the existence of the Type III degeneration $\mathcal{Y} \to \Delta$ above. 
\end{proof}

\section{Good sublattices and adjacent RDP singularities}

Throughout this section, we fix a deformation type of a negative definite anticanonical pair $(Y,D)$, $\Lambda= \Lambda(Y,D)$ and $R$ denotes the set of roots of $\Lambda$. We begin by describing a class of negative definite sublattices of $\Lambda$:

\begin{lemma}\label{Lemma5.1} Let $(Y,D)$ be an anticanonical pair and let $\Upsilon$ be a negative definite sublattice of $\Lambda=\Lambda(Y,D)$, not necessarily primitive. Suppose that 
\begin{enumerate}
\item[\rm(i)] $\Upsilon$ is spanned by elements of $R$. 
\item[\rm(ii)] The period homomorphism $\varphi_Y$ satisfies: $\Ker \varphi_Y\cap R = \Upsilon\cap R$.
\end{enumerate}
Then $\Upsilon \cap R$ is a root system in the real vector space $\Upsilon_\R = \Upsilon\otimes \R$ and there exist a set of simple roots $\beta_1, \dots, \beta_n\in \Upsilon$ such that
\begin{enumerate}
\item[\rm(i)] For every $i$, $\beta_i =[C_i]$, where $C_i$ is a $-2$-curve.
\item[\rm(ii)] If $\beta\in R$ and $\varphi_Y(\beta)=1$, then $\beta = \pm\sum_in_iC_i$, where the $n_i$ are nonnegative  integers. In particular, if $C$ is a $-2$-curve on $Y$, then $C=C_i$ for some $i$.
\end{enumerate}
\end{lemma}
\begin{proof}  If $\beta_1, \beta_2\in \Upsilon \cap R$, then $r_{\beta_1}(\beta_2) \in \Upsilon \cap R$ by \cite[(6.9)]{Friedman2}.  It then follows e.g.\ by Definition 1 in \cite[IV \S1]{Bour} that $\Upsilon \cap R$ is a root system in  $\Upsilon_\R$ (and in fact that every element of square $-2$ in the $\Z$-span of $\Upsilon \cap R$ is a root). 

Let $\Delta_Y$ be the set of $-2$-curves on $Y$, let $\mathsf{W}(\Delta_Y)$ be the reflection group generated by $\Delta_Y$, and let $R_Y^{\text{nod}} = \mathsf{W}(\Delta_Y)\cdot \Delta_Y$. Then,  by \cite[Theorem 6.6]{Friedman2}, $R_Y^{\text{nod}}= \Ker \varphi_Y\cap R = \Upsilon\cap R$.  Let $\beta_1, \dots, \beta_n\in \Upsilon$ be the simple roots for  a Weyl chamber in $\Upsilon_\R$ defined by an ample divisor on $Y$.  By \cite[Lemma 2.4(ii)]{Friedman1.5}, all positive roots are effective. On the other hand, every $\beta \in R_Y^{\text{nod}}$ is of the form $\sum_in_i[C_i]$, where the $C_i$ are $-2$-curves on $Y$.  A standard argument using the fact that $\Upsilon$ is negative definite shows that, if $\beta$ is effective and of the form $\sum_in_i[C_i]$, where the $C_i$ are $-2$-curves, then $n_i \geq 0$ for all $i$.  Hence the positive roots are exactly the classes of the $\beta \in  R_Y^{\text{nod}}$ of the form $\sum_in_i[C_i]$, $n_i\geq 0$, and the simple roots $\beta_i$ are exactly the classes of the $-2$-curves on $Y$.
\end{proof}

\begin{lemma}\label{existslambda} Let $\Upsilon$ be a negative definite sublattice of $\Lambda$ and suppose that $\Upsilon$ is spanned by  $R\cap \Upsilon$. Then there exists a $\lambda\in  \mathcal{B}_{\text{\rm{gen}}}$ such that $\Upsilon \subseteq (\lambda)^\perp$.
\end{lemma}
\begin{proof} Let $\Upsilon'$ be the saturation of $\Upsilon$, and let $\Upsilon''$ be the finite index sublattice of $\Upsilon'$ spanned by  $R\cap \Upsilon'$. Then there exists a $\varphi\colon \Lambda \to \C^*$ with $\Ker \varphi = \Upsilon'$. Hence $\Ker \varphi \cap R =  R\cap \Upsilon'=  R\cap \Upsilon''$, and $\Upsilon''$ is spanned by $R\cap \Upsilon''$. Using the surjectivity of the period map, there exists a $Y$ such that $\varphi=\varphi_Y$. By Lemma~\ref{Lemma5.1}, there exist $-2$-curves $C_1, \dots, C_k$ whose classes span $\Upsilon''$.  Moreover, there is a nef and big divisor $H$ on $Y$ such that $H\cdot C_j =0$ for all $j$, $H\cdot D_i =0$  for all $i$ and $H\cdot C > 0$ for every curve $C\neq C_j$ or $D_i$. If $\lambda =[H]$, then $\lambda\in \mathcal{B}_{\text{\rm{gen}}}$ and $\lambda\cdot [C_j]=0$ for all $j$, hence $\lambda \cdot x =0$ for all $x\in \Upsilon$. 
\end{proof}

We can then rephrase the conditions of Lemma~\ref{Lemma5.1} as follows:

\begin{lemma} Let $\lambda$ be a class in $\mathcal{B}_{\text{\rm{gen}}}$. Let $\beta\in \Lambda$, $\beta^2=-2$. Suppose that $\beta\cdot \lambda =0$. Then $\beta \in R$.

Hence, the condition  {\rm(i)} of Lemma~\ref{Lemma5.1} is equivalent to: there exists a $\lambda\in \mathcal{B}_{\text{\rm{gen}}}$ such that $\Upsilon \subseteq (\lambda)^\perp$ and

\smallskip
 
{\rm{(i)}${}'$} $\Upsilon$ is spanned by elements of $\Lambda$ of square $-2$.
\end{lemma}
\begin{proof} Since $\lambda$ is in the interior of $\mathcal{B}_{\text{\rm{gen}}}$, it is easy to check that $\lambda$ is not orthogonal  to any exceptional curve and hence that there is a neighborhood of $\lambda$ in $\mathcal{C}^+$ which is disjoint from $W^E$ for every exceptional curve $E$.   There exist positive $r_i$ such that $(\sum_ir_i[D_i])\cdot [D_j] < 0$ for all $j$, and hence 
$$(\lambda -\sum_ir_i[D_i]) \cdot [D_j] > 0$$ for all $j$. Multiplying the $r_i$ by a positive $t\ll 1$, we can further assume that $\lambda -\sum_ir_i[D_i]$  is not separated from $\lambda$ by a wall of the form $W^E$, where $E$ is exceptional, and hence $\lambda -\sum_ir_i[D_i] \in \mathcal{A}_{\text{\rm{gen}}}$. Thus $W^\beta$ meets the interior of $\overline{\mathcal{A}}_{\text{\rm{gen}}}$. It follows from   \cite[Theorem 6.6]{Friedman2} that $\beta \in R$.
\end{proof}

\begin{definition} Let $\Upsilon$ be a negative definite sublattice of $\Lambda$, not necessarily primitive. Then $\Upsilon$ is \textsl{good} if 
\begin{enumerate}
\item[\rm(i)] $\Upsilon$ is spanned by elements of $R$. 
\item[\rm(ii)] There exists a homomorphism $\varphi\colon \Lambda \to \C^*$ such that $\Ker \varphi\cap R = \Upsilon\cap R$.
\end{enumerate}
The lattice $\Upsilon$  determines an RDP configuration (possibly consisting of more than one singular point) by taking the appropriate type of a set of simple roots for $\Upsilon \cap R$. We say that the corresponding RDP configuration  is \textsl{of type $\Upsilon$}.

Given $\lambda\in \mathcal{B}_{\text{\rm{gen}}}$, the lattice $\Upsilon$ is \textsl{$\lambda$-good} if it is good and if $\Upsilon \subseteq (\lambda)^\perp$.
 By Lemma~\ref{existslambda}, every good sublattice is $\lambda$-good for some $\lambda$.
 \end{definition}
 
Let  $\Upsilon$ be a lattice  spanned by elements of $R$. If $\Upsilon$ is primitive, then it is good. If $\Upsilon'$ is the saturation of $\Upsilon$ and $\Upsilon'/\Upsilon$ is cyclic, then $\Upsilon$ is good. If $\Upsilon =\bigoplus \Upsilon_i$ is the decomposition of $\Upsilon$ into its irreducible summands, then $\Upsilon$ is good $\iff$ $\Upsilon_i$ is good for every $i$ (this follows easily from the fact that,  if $\beta =\sum_i\beta_i \in \Lambda$ with $\beta_i^2< 0$ for all $i$ and $\beta_i\cdot \beta_j =0$ if $i\neq j$, then $\beta^2 = -2$ $\iff$ $\beta=\beta_i$ for some $i$ and $\beta_i^2 = -2$). 

To explain the geometric meaning of good lattices, we start with the following definition:

\begin{definition} A \textsl{generalized anticanonical pair} is a pair  $(\overline{Y}, D)$, where $\overline{Y}$ is a rational surface, possibly with rational double points, and $D$ is a cycle  contained in the smooth locus of $\overline{Y}$, such that, if $\pi\colon Y \to \overline{Y}$ is the minimal resolution, then $(Y,D)$ is an anticanonical pair.
\end{definition}

We then have the following alternate characterization of RDP configurations of type $\Upsilon$ (cf.\ \cite[II (2.7)]{Looij}):

\begin{proposition}\label{meaningofgood} Suppose that $(\overline{Y}, D)$ is a generalized anticanonical pair with minimal resolution $\pi\colon Y \to \overline{Y}$ such that every $-2$-curve on $Y$ is contained in a fiber of $\pi$. Then the classes of the components of fibers of $\pi$ generate a good sublattice $\Upsilon$ of $\Lambda(Y,D)$, and every good sublattice of $\Lambda(Y,D)$ arises in this way. 

Hence the RDP configurations of type $\Upsilon$, for some good sublattice of $\Lambda(Y,D)$, are exactly the RDP configurations which appear on a generalized anticanonical pair $(\overline{Y}, D)$ whose minimal resolution has deformation type $[(Y,D)]$ and which satisfy the condition that every $-2$-curve on the minimal resolution is contained in a fiber of $\pi$.
\end{proposition}
\begin{proof} If $(\overline{Y}, D)$ is as in the statement of the proposition, let $\beta \in \Ker \varphi_Y\cap R$. Then, by \cite[6.6]{Friedman2} and its method of proof, $\pm \beta = \sum_in_i[C_i]$, where $C_i$ is a $-2$-curve on $Y$.  Thus $\beta \in \Upsilon$. It follows that $\Upsilon$ is a negative definite lattice spanned by $\Ker \varphi_Y\cap R$, so that $\Upsilon$ is good.

Conversely, suppose that $\Upsilon$ is a good sublattice. By the surjectivity  of the period map, we can choose an anticanonical pair $(Y,D)$ such that $\Ker \varphi_Y\cap R = \Upsilon \cap R$ and $\Upsilon$ is spanned by  $\Upsilon \cap R$. By Lemma~\ref{Lemma5.1}, there exists a set of simple roots $\beta_1=[C_1], \dots, \beta_k =[C_k]$, where the $C_i$ are exactly the $-2$-curves on $Y$. Since $\Upsilon$ is negative definite, there exists a morphism $\pi \colon Y \to \overline{Y}$, where $\overline{Y}$ is normal and $\pi$ exactly contracts the $C_i$. Hence $(\overline{Y}, D)$ is a generalized anticanonical pair with minimal resolution $\pi\colon Y \to \overline{Y}$ such that every $-2$-curve on $Y$ is contained in a fiber of $\pi$.
\end{proof} 
 
 We now relate the above discussion to the existence of adjacent RDP singularities.

\begin{theorem}\label{charrat} Let $\pi\colon \mathcal{Y}\to \Delta$ be a Type III degeneration with monodromy invariant $\lambda$. Let $\beta_1, \dots, \beta_k$ be the finite set of elements $\beta$ of $R$ such that $\beta \cdot \lambda = 0$. Then, for all $t\in\Delta^*$, $|t|\ll 1$ and for all $\beta \in R$, $\varphi_{Y_t}(\beta) = 1$ $\iff$ $\beta = \beta_i$ for some $i$ and the function $f(t) = \varphi_{Y_t}(\beta_i)$ is identically equal to $1$ for all $t \in \Delta^*$.
\end{theorem}
\begin{proof} $\impliedby$: Obvious. 

\noindent $\implies$: Let $U$ be a neighborhood of $\lambda$ in $\mathcal{C}^+$ such that, if $\beta \in R$ and $W^\beta$  is a wall with $W^\beta\cap U \neq \emptyset$, then $\beta=\beta_i$ for some $i$, i.e.\ $W^\beta$ is one of the finitely many walls corresponding to elements of $R$ and passing through $\lambda$.

Given $\varphi_{Y_t}\colon \Lambda \to \C^*$, the function $\log |\varphi_{Y_t}|\colon \Lambda \to \R$ is a well-defined, $C^\infty$ function on $\Delta^*$ and thus corresponds to a function $\Delta^* \to \Lambda_\R\spcheck = \Lambda_\R$, which we continue to denote by $\log |\varphi_{Y_t}|$.  By \cite[(3.12)]{Friedman2}, 
$$\log \varphi_{Y_t} = 2\pi \sqrt{-1}\omega(t)$$ 
as multi-valued functions on $\Delta^*$. Hence, by  Proposition~\ref{asympbehav}, 
$$\log \varphi_{Y_t} = 2\pi \sqrt{-1}\omega(t)   =  (\log t) \lambda + 2\pi \sqrt{-1}f(t)$$
for a single-valued holomorphic function $f$ which extends to a holomorphic function at $0$. Taking real parts, we obtain an equality of (single-valued) functions
$$\log |\varphi_{Y_t}| = \operatorname{Re} ((\log t )\lambda + 2\pi \sqrt{-1}f(t)) = \log |t|\lambda + g(t),$$
where $g(t)$ is a $C^\infty$ function at $0$. Dividing by $\log |t|$ gives
$$\frac{\log |\varphi_{Y_t}|}{\log |t|} = \lambda + \frac{g(t)}{\log |t|}.$$
Thus $\log |\varphi_{Y_t}|/\log |t|$ defines a continuous function $h(t)\colon \Delta\to \Lambda_\R$ with $h(0) = \lambda$. Thus, for $|t|\ll 1$, $h(t) \in U$. It follows that, if $t\in \Delta^*$, $|t|\ll 1$, $\beta \in R$, and $\varphi_{Y_t}(\beta) = 1$, then $h(t)$ lies on $W^\beta$ and so   $\beta=\beta_i$ for some $i$.

Given $\beta_i$ as above, either $\varphi_{Y_t}(\beta_i)$ is identically equal to $1$ for all $t \in \Delta$ or $\varphi_{Y_t}(\beta_i) = 1$ for at most finitely many $t\in \Delta^*$. Choosing $|t|$ smaller if need be, we can then assume that, for all $i$, if $\varphi_{Y_t}(\beta_i)$ is not identically equal to $1$, then $\varphi_{Y_t}(\beta_i) \neq 1$ for all $t\in \Delta^*$.
\end{proof}

Given a Type III degeneration $\mathcal{Y}\to \Delta$, by a theorem of Shepherd-Barron \cite[Theorem 2A, p.\ 157]{Shepherd-Barron}, there exists a birational morphism $f\colon \mathcal{Y} \to \overline{\mathcal{Y}}$ over $\Delta$, where $f_0=f\big{|}_{Y_0}$ contracts $\bigcup_{i>0}V_i$  and the cycle $D'$ to a point and, for $t\neq 0$, $f_t=f\big{|}_{Y_t}\colon Y_t \to \overline{Y}_t$ is the minimal resolution of a union of rational double points. We then have the following (compare \cite[III (2.12)]{Looij}):

\begin{proposition}\label{charexcepfibers} For $0<|t|\ll 1$, the irreducible components of the exceptional fibers of $f_t$ are exactly the $-2$-curves on $Y_t$.
\end{proposition}
\begin{proof} Clearly, if $C_t$ is an irreducible component  of an exceptional fiber of $f_t$, then $C_t$ is a $-2$-curve on $Y_t$. Conversely, suppose that $C_t$ is a $-2$-curve on $Y_t$ for some for $0<|t|\ll 1$ and let $[C_t] =\beta$. Then $\varphi_{Y_t}(\beta) = 1$. By Theorem~\ref{charrat},   $\beta $ is orthogonal to $\lambda$ and $\varphi_{Y_t}(\beta)$ is identically $1$. It follows from  Proposition~\ref{lambdaprop} that $\beta$ corresponds to a line bundle $\mathcal{L}$  on $\mathcal{Y}$. For all $s\in \Delta^*$, $\pm \beta \in R_{Y_s}^{\text{nod}}$, i.e.\ $\pm \beta = \sum_i[C_i]$ where the $C_i$ are $-2$-curves on $Y_s$. The condition on $s\in \Delta^*$ that $\beta$ is effective on $Y_s$ is closed (by semicontinuity) and open (because we can choose a relatively ample divisor $\mathcal{H}$ on $\mathcal{Y}$ over small open subsets of $\Delta^*$). Since $\beta$ is effective on $Y_t$, it follows that $\beta$ is effective for all $s\in \Delta^*$, and,  for  $s$ generic,  $\beta =[C_s]$ is the class of an irreducible curve. Since $h^0(Y_s; \mathcal{L}\big{|}_{Y_s}) >0$ for all $s\in \Delta^*$, arguing as in the proof of Proposition~\ref{extendsections}, after modifying $\mathcal{L}$ by $\scrO_{\mathcal{Y}}(\sum_ia_iV_i)$, we can assume that $\mathcal{L}= \scrO_{\mathcal{Y}}(\mathcal{C})$ where $\mathcal{C}$ is a surface on $\mathcal{Y}$, not containing $V_i$ for any $i$, and such that $\mathcal{C}\cap Y_t = C_t$. In particular,   there exists an effective Cartier divisor $C_0=\mathcal{C}\cap Y_0$ on $Y_0$ such that $C_0\cdot D_j = 0$ for every $j$. But every such divisor $C_0$ is a sum of the form $\sum_in_iD_i$, where $n_i \geq 0$,  plus components disjoint from $D$, and hence is contained in $\bigcup_{i>0}V_i$. Since $(D_i\cdot D_j)$ is negative definite, this is  only possible if $n_i=0$ for all $i$. 
 Thus $C_0 \subseteq \bigcup_{i>0}V_i$, so that, via the relative contraction morphism $f\colon \mathcal{Y} \to \overline{\mathcal{Y}}$ over $\Delta$, $C_0$ is contracted to a point. Since $\mathcal{C}$ fibers over $\Delta$ and $f$ contracts the fiber $C_0$ to a point, $f$ must contract all fibers of $\mathcal{C}\to \Delta$ to  points as well. Thus $f_t$ contracts $C_t$ to a point.
\end{proof}

We can now state the main result concerning possible adjacencies of the cusp to RDP singularities: 

\begin{theorem}\label{lastthm} {\rm (i)} Suppose that $\lambda\in \mathcal{B}_{\text{\rm{gen}}}$ and that $\Upsilon$ is a $\lambda$-good sublattice of $\Lambda$. Let $(Y_0,D)$  be a Type III anticanonical pair  of type $\lambda$. Then the cusp $D'$ is adjacent to an RDP  configuration of type $\Upsilon$. More precisely, after a locally trivial deformation of $Y_0$ through $d$-semistable varieties, there exists a Type III degeneration $\mathcal{Y} \to \Delta$ with central fiber $Y_0$ and a blowdown $\mathcal{Y}\to \overline{\mathcal{Y}}$ over $\Delta$, such that the central fiber of $\overline{\mathcal{Y}}$ is $V_0$ with $D'$ contracted and the singular locus of the general fiber is an
RDP configuration of type $\Upsilon$.

\smallskip
\noindent {\rm (ii)} Conversely, suppose that the cusp $D'$ is adjacent to an RDP configuration  via a one parameter family $\bar{\pi}\colon \overline{\mathcal{Y}} \to \Delta$ of $\Q$-type $\lambda$. For $|t|\ll 1$, if $Y_t$ is the minimal resolution of a general fiber $\overline{Y}_t$ of $\bar\pi$, then the components of the exceptional curves generate a $\lambda$-good sublattice $\Upsilon$ of $\Lambda$, and the RDP configuration on the general fiber is of type $\Upsilon$.
\end{theorem}
\begin{proof} (i) Given the $\lambda$-good sublattice $\Upsilon$, choose a $\varphi\colon \Lambda\to \C^*$ such that $\Ker \varphi\cap R = \Upsilon\cap R$. Let $\bar\varphi$ be the restriction of $\varphi$ to $\overline{\Lambda}_0= (\lambda)^\perp$. By Theorem~\ref{surj1}, after a locally trivial deformation of $Y_0$ through $d$-semistable varieties, we can assume that $\bar\varphi_{Y_0} =\bar\varphi$. By Proposition~\ref{surj2}, there exists a Type III degeneration $\mathcal{Y} \to \Delta$ with central fiber $Y_0$ such that, under the natural identification of $\Upsilon$ and $\overline{\Lambda}_0$ with sublattices of $\Lambda(Y_t)$, $\varphi_{Y_t}(\beta) = 1$ for all $\beta\in \Upsilon\cap R$ and all $t\in \Delta^*$.  Conversely, by Theorem~\ref{charrat}, for   $|t|\ll 1$, if $\beta \in R$ and $\varphi_{Y_t}(\beta) =1$, then $\varphi_{Y_t}(\beta)$ is identically $1$ and $\beta \in (\lambda)^\perp$, hence $\beta \in\overline{\Lambda}_0$. Using  the continuity of $\varphi_{Y_t}(\beta)$ (Proposition~\ref{extpermapholom}), $\bar\varphi_{Y_0}(\beta) =1$ as well. Thus $\beta\in \Upsilon\cap R$. In particular, by Lemma~\ref{Lemma5.1}, the $-2$-curves on $Y_t$ correspond to a set of simple roots for $\Upsilon$. By Proposition~\ref{charexcepfibers}, these are exactly the exceptional fibers of the blowdown $Y_t \to \overline{Y}_t$. Hence the singular locus of $Y_t$ is an
RDP configuration of type $\Upsilon$.

\smallskip
\noindent (ii) After a base change, we can assume that $\bar{\pi}\colon \overline{\mathcal{Y}} \to \Delta$ is birational to a Type III degeneration of $\Q$-type $\lambda$. For $0<|t|\ll 1$, let $\Upsilon$ be the sublattice of $\Lambda$ generated by $\Ker  \varphi_{Y_t}\cap R$. By Theorem~\ref{charrat}, if $\beta \in \Ker  \varphi_{Y_t}\cap R$, then $\beta \in (\lambda)^\perp$, and hence $\Upsilon \subseteq (\lambda)^\perp$. Thus $\Upsilon$ is good, in fact it is $\lambda$-good.  By Lemma~\ref{Lemma5.1},  the classes of the $-2$-curves on $Y_t$   are a set of simple roots for $\Upsilon$.  By Proposition~\ref{charexcepfibers}, the  $-2$-curves on $Y_t$ are exactly the fibers of the morphism $Y_t \to \overline{Y}_t$. Hence the RDP configuration   on $Y_t$ is of type $\Upsilon$.
\end{proof}

\begin{corollary} Let $p'$ be a cusp singularity with minimal resolution $D'$ and let $D$ be the dual cycle. Suppose that $(Y,D)$ is an anticanonical pair. Then the possible nearby RDP configurations on some smoothing component of $p'$ of type $(Y,D)$ are exactly  the RDP configurations of type $\Upsilon$, where $\Upsilon$ is a good  sublattice of $\Lambda(Y,D)$.  
\end{corollary}
\begin{proof} Suppose that there is a nearby RDP configuration  on some smoothing component of $p'$ of type $(Y,D)$. Then we can choose a one parameter deformation $\pi \colon (\overline{\mathcal{Y}}, \mathcal{D}) \to \Delta$, where the central fiber $(\overline{Y}_0,D)$ is the Inoue surface with $D'$ contracted. After base change, there is a birational model for $\pi \colon (\overline{\mathcal{Y}}, \mathcal{D}) \to \Delta$ which is a Type III degeneration with monodromy invariant $\lambda$ for some $\lambda\in \mathcal{B}_{\text{\rm{gen}}}$. By (ii) of Theorem~\ref{lastthm}, the RDP configuration on a general fiber of $\bar\pi$ is of type $\Upsilon$, where $\Upsilon$ is a $\lambda$-good and hence good  sublattice of $\Lambda(Y,D)$.

Conversely, suppose that $\Upsilon$ is a good  sublattice of $\Lambda(Y,D)$. By Lemma~\ref{existslambda}, $\Upsilon$ is $\lambda$-good for some $\lambda\in \mathcal{B}_{\text{\rm{gen}}}$. By Theorem~\ref{lambdas}, there exists a Type III anticanonical pair $(Y_0,D)$ with  monodromy invariant $\lambda$. Hence, by (i) of Theorem~\ref{lastthm}, $p'$ is adjacent to an RDP of type $\Upsilon$.
\end{proof}

\bigskip
\noindent
Philip Engel \\
Department of Mathematics \\
Harvard University \\
Cambridge, MA 02138 \\
USA \\
{\tt engel@math.harvard.edu}

\bigskip
\noindent
Robert Friedman \\
Department of Mathematics \\
Columbia University \\
New York, NY 10027 \\
USA \\
{\tt rf@math.columbia.edu}


\begin{thebibliography}{99}

\bibitem{abouzaid}
M. Abouzaid, D. Auroux, and L. Katzarkov,
\emph{Lagrangian fibrations on blowups of toric varieties and mirror symmetry for hypersurfaces},
{\tt{https://arxiv.org/abs/1205.0053}}, 2012.

\bibitem{Bour}
N. Bourbaki,
\emph{Groupes et Alg\`ebres de Lie}, 
Chap. 4, 5, et 6,
Masson,  Paris,
1981.

\bibitem{Clemens}
C.H. Clemens, Jr.,
\emph{Picard-Lefschetz theorem for families of nonsingular algebraic varieties acquiring ordinary singularities},
 Trans. Amer. Math. Soc. \textbf{136}~(1969),  93--108. 

\bibitem{Deligne}
P. Deligne,
\emph{\'{E}quations diff\'{e}rentielles \`{a} points singuliers r\'{e}guliers},   Lecture Notes in Mathematics, Vol. 163. Springer-Verlag, Berlin-New York, 1970.

\bibitem{Engel}
P. Engel,
\emph{A proof of Looijenga's conjecture via integral-affine geometry},
{\tt{http://arxiv.org/pdf/1409.7676}},
2014.
 
\bibitem{Friedman1}
R. Friedman,
\emph{Global smoothing of varieties with normal crossings}, 
Annals of Math. \textbf{118}~(1983), 75--114.

\bibitem{FriedmanBaseChange}
R. Friedman,
\emph{Base change, automorphisms, and stable reduction for type III $K3$ surfaces},
in \emph{The Birational Geometry of Degenerations}
pp. 277--298, Progr. Math. \textbf{29}, Birkh\"auser, Boston, Mass., 1983.

\bibitem{Friedman1.5}
R. Friedman,
\emph{On the ample cone of a rational surface with an anticanonical cycle},
Algebra \& Number Theory \textbf{7}~(2013), 
1481--1504.

\bibitem{Friedman2}
R. Friedman,
\emph{On the geometry of anticanonical pairs}, 
{\tt{arxiv.org/abs/1502.02560}},
2016.

\bibitem{FriedmanMiranda}
R. Friedman and R. Miranda,
\emph{Smoothing cusp singularities of small length},
Math. Ann. \textbf{263}~(1983),
185--212.

\bibitem{FriedmanScattone}
R. Friedman and F. Scattone,
\emph{Type III degenerations of $K3$  surfaces},
Invent. Math. \textbf{83}~(1986), 
1--39.

\bibitem{grossspecial}
M. Gross.
\emph{Special Lagrangian fibrations I: Topology},
{\tt{https://arxiv.org/abs/alg-geom/9710006}},
1997.

\bibitem{GHK}
M. Gross, P. Hacking, and S. Keel,
\emph{Moduli of surfaces with an anti-canonical cycle},
 Compos. Math. \textbf{151}~(2015),  265--291.
 
\bibitem{GHK2}
M. Gross, P. Hacking, and S. Keel,
\emph{Mirror symmetry for log Calabi-Yau surfaces I},
 Publ. Math. Inst. Hautes \'Etudes Sci. \textbf{122}~(2015), 65--168.

\bibitem{gs}
M. Gross and B. Siebert,
\emph{Mirror symmetry via logarithmic degeneration data I},
J. of Differential Geom. \textbf{72.2}~(2006), 169--338.

\bibitem{gs2}
M. Gross and B. Siebert.
\emph{Mirror symmetry via logarithmic degeneration data, II.}
J. of Algebraic Geom. \textbf{19.4}~(2010), 679--780.

\bibitem{hitchin}
N. Hitchin,
\emph{The moduli space of special Lagrangian sub manifolds},
Annali Scuola Sup. Norm. Pisa Sci. Fis. Mat. \textbf{25}~(1997), 503--515.

\bibitem{Inoue}
M. Inoue,
\emph{New surfaces with no meromorphic functions II},
in  \emph{Complex analysis and algebraic geometry},  91--106. Iwanami Shoten, Tokyo, 1977.

\bibitem{KawamataNamikawa}
Y. Kawamata and Y. Namikawa,
\emph{Logarithmic deformation of normal crossing varieties and smoothing of degenerate Calabi-Yau manifolds},
Invent. Math. \textbf{118}~(1994), 395--409.

\bibitem{hms}
M. Kontsevich,
\emph{Homological algebra of mirror symmetry},
{\tt{https://arxiv.org/abs/alg-geom/9411018}},
1994.

\bibitem{kontsoib}
M. Kontsevich and Y. Soibelman.
\emph{Affine structures and non-Archimedean analytic spaces},
The unity of mathematics, Birkh\"auser Boston~(2006), 321--385.

\bibitem{lawson}
C. L. Lawson,
\emph{Transformation triangulations}
Discrete Math. {\bf 3}~(1972), 365-272.

\bibitem{LiLiu}
T.-J.  Li and A.-K. Liu, 
\emph{Uniqueness of symplectic canonical class, surface cone and symplectic cone of 4-manifolds with $B^+=1$},  J. Differential Geom. \textbf{58}~(2001),   331--370.

\bibitem{LiMak}
T.-J. Li and C. Y. Mak,
\emph{Symplectic log Calabi-Yau surface---deformation class},
{\tt{http://arxiv.org/pdf/1510.06131}},
2015.

\bibitem{Looij1}
E. Looijenga,
\emph{On the semi-universal deformation of a simple-elliptic hypersurface singularity. II. The discriminant}, Topology \textbf{17}~(1978),   23--40.

\bibitem{Looij2}
E. Looijenga,
\emph{The smoothing components of a triangle singularity. II},  Math. Ann. \textbf{269}~(1984),   357--387.

\bibitem{Looij}
E. Looijenga,
\emph{Rational surfaces with an anti-canonical cycle},
Annals of Math. \textbf{114}~(1981),
267--322.

\bibitem{LooijengaWahl}
E. Looijenga and J. Wahl, 
\emph{Quadratic functions and smoothing surface singularities}, Topology \textbf{25}~(1986),   261--291.

\bibitem{morrison}
D. Morrison.
\emph{Compactifications of moduli spaces inspired by mirror symmetry},
{\tt{https://arxiv.org/abs/alg-geom/9304007}}, 1993.

\bibitem{Persson}
U. Persson,
\emph{On Degenerations of Algebraic Surfaces},
 Mem. Amer. Math. Soc. \textbf{11}~(1977), no. 189.

 \bibitem{PetersSteenbrink}
C.A.M. Peters and J.H.M. Steenbrink, 
\emph{Mixed Hodge structures}, Ergebnisse der Mathematik und ihrer Grenzgebiete, 3. Folge, \textbf{52}. Springer-Verlag, Berlin, 2008.

\bibitem{Pink1}
H. C. Pinkham,
\emph{Simple elliptic singularities, Del Pezzo surfaces and Cremona transformations}, in  \emph{Several complex variables} (Proc. Sympos. Pure Math., Vol. XXX, Part 1, Williams Coll., Williamstown, Mass., 1975),   69--71. Amer. Math. Soc., Providence, R. I., 1977.

\bibitem{Pink2}
H. C. Pinkham,
\emph{Groupe de monodromie des singularit\'es unimodulaires exceptionnelles},  C. R. Acad. Sci. Paris S\'er. A-B \textbf{284}~(1977),   A1515--A1518.

\bibitem{Shepherd-Barron}
N. I. Shepherd-Barron,
\emph{Extending polarizations on families of $K3$ surfaces},
in \emph{The Birational Geometry of Degenerations}
pp. 135--171, Progr. Math. \textbf{29}, Birkh\"auser, Boston, Mass., 1983.

\bibitem{Steenbrink}
J.H.M Steenbrink,
\emph{Limits of Hodge structures},
 Invent. Math. \textbf{31}~(1975/76),  229--257.

\bibitem{syz}
A. Strominger, S.-T. Yau, and E. Zaslow,
\emph{Mirror symmetry is T-duality.}
Nuclear Physics B \textbf{479.1-2}~(1996), 243--259.

\bibitem{Symington}
M. Symington,
\emph{Four dimensions from two in symplectic topology},
in {\em Proceedings of Symposia in Pure Mathematics}, Vol. 71,
pp. 153--208. Amer. Math. Soc., 2003.

\bibitem{Wahl}
J. Wahl,
\emph{Elliptic deformations of minimally elliptic singularities},
 Math. Ann. \textbf{253}~(1980),  241--262.


\end{thebibliography}
\end{document}